\documentclass[11pt]{aart}

\usepackage[letterpaper, hmargin=1in, top=1in, bottom=1.2in, footskip=0.6in]{geometry}

\usepackage{titlesec}
\titleformat{\section}[block]{\filcenter\normalfont\bfseries\large}{\thesection.}{.5em}{}\titlespacing*{\section}{0pt}{2\baselineskip}{1\baselineskip}
\titleformat{\subsection}[runin]{\normalfont\bfseries}{\thesubsection.}{.4em}{}[.]\titlespacing{\subsection}{0pt}{2ex plus .1ex minus .2ex}{.8em}
\titleformat{\subsubsection}[runin]{\normalfont\itshape}{\thesubsubsection.}{.3em}{}[.]\titlespacing{\subsubsection}{0pt}{1ex plus .1ex minus .2ex}{.5em}
\titleformat{\paragraph}[runin]{\normalfont\itshape}{\theparagraph.}{.3em}{}[.]\titlespacing{\paragraph}{0pt}{1ex plus .1ex minus .2ex}{.5em}

\usepackage[T1]{fontenc}
\usepackage[utf8]{inputenc}
\usepackage{lmodern}

\usepackage{amsmath}
\usepackage{amssymb}
\usepackage{amsfonts}
\usepackage{latexsym}
\usepackage{amsthm}
\usepackage{amsxtra}
\usepackage{amscd}
\usepackage{bbm}
\usepackage{mathrsfs}
\usepackage{bm}
\usepackage{tensor}

\usepackage{graphicx, color}

\definecolor{darkred}{rgb}{0.9,0,0.3}
\definecolor{darkblue}{rgb}{0,0.3,0.9}

\newcommand{\nc}{\normalcolor}

\definecolor{vdarkred}{rgb}{0.7,0,0.2}
\definecolor{vdarkblue}{rgb}{0,0.2,0.7}
\usepackage[pdftex, colorlinks, linkcolor=vdarkblue,urlcolor=vdarkblue,citecolor=vdarkred]{hyperref}

\usepackage[nottoc,notlof,notlot]{tocbibind}
\usepackage{cite}

\numberwithin{equation}{section}
\numberwithin{figure}{section}

\theoremstyle{plain} 
\newtheorem{theorem}{Theorem}[section]
\newtheorem*{theorem*}{Theorem}
\newtheorem{lemma}[theorem]{Lemma}
\newtheorem*{lemma*}{Lemma}
\newtheorem{corollary}[theorem]{Corollary}
\newtheorem*{corollary*}{Corollary}
\newtheorem{proposition}[theorem]{Proposition}
\newtheorem*{proposition*}{Proposition}

\newtheorem*{conjecture*}{Conjecture}

\theoremstyle{definition} 
\newtheorem{definition}[theorem]{Definition}
\newtheorem*{definition*}{Definition}

\newtheorem*{example*}{Example}
\newtheorem{remark}[theorem]{Remark}
\newtheorem*{remark*}{Remark}

\newtheorem*{assumption*}{Assumption}

\newtheorem*{convention*}{Convention}

\newcommand{\f}[1]{\boldsymbol{\mathrm{#1}}} 
\newcommand{\bb}{\mathbb} 
\renewcommand{\cal}{\mathcal}

\newcommand{\ul}[1]{\underline{#1} \!\,} 
\newcommand{\ol}[1]{\overline{#1} \!\,} 
\newcommand{\wh}{\widehat}
\newcommand{\wt}{\widetilde}
\newcommand{\txt}[1]{\text{\rm{#1}}}

\renewcommand{\P}{\mathbb{P}}
\newcommand{\E}{\mathbb{E}}
\newcommand{\R}{\mathbb{R}}
\newcommand{\C}{\mathbb{C}}
\newcommand{\N}{\mathbb{N}}

\newcommand{\e}{\mathrm{e}}

\newcommand{\ii}{\mathrm{i}}
\newcommand{\dd}{\mathrm{d}}
\newcommand{\col}{\mathrel{\vcenter{\baselineskip0.75ex \lineskiplimit0pt \hbox{.}\hbox{.}}}}
\newcommand*{\deq}{\mathrel{\vcenter{\baselineskip0.65ex \lineskiplimit0pt \hbox{.}\hbox{.}}}=}
\newcommand*{\eqd}{=\mathrel{\vcenter{\baselineskip0.65ex \lineskiplimit0pt \hbox{.}\hbox{.}}}}

\newcommand{\todist}{\overset{\dd}{\longrightarrow}}
 
\renewcommand{\leq}{\leqslant}
\renewcommand{\le}{\leqslant}
\renewcommand{\geq}{\geqslant}
\renewcommand{\ge}{\geqslant}
\renewcommand{\epsilon}{\varepsilon}

\newcommand{\floor}[1] {\lfloor {#1} \rfloor}
\newcommand{\ceil}[1]  {\lceil  {#1} \rceil}

\newcommand{\p}[1]{({#1})}
\newcommand{\pb}[1]{\bigl({#1}\bigr)}

\newcommand{\pbb}[1]{\biggl({#1}\biggr)}

\newcommand{\abs}[1]{\lvert #1 \rvert}

\DeclareMathOperator{\tr}{Tr}

\DeclareMathOperator{\supp}{supp}
\DeclareMathOperator{\re}{Re}
\DeclareMathOperator{\im}{Im}

\newcommand{\bN}{ {\mathbb N} }

\newcommand{\bP}{ {\mathbb P} }
\newcommand{\bR}{ {\mathbb R} }

\newcommand*{\rom}[1]{\expandafter\@slowromancap\romannumeral #1@}

\title{Fluctuations of extreme eigenvalues of sparse Erd\H{o}s-R\'enyi graphs}
\author{Yukun He \and Antti Knowles}

\begin{document}
\maketitle
\begin{abstract}
We consider a class of sparse random matrices which includes the adjacency matrix of the Erd\H{o}s-R\'enyi graph $\cal G(N,p)$. We show that if $N^{\varepsilon} \leq Np  \leq N^{1/3-\varepsilon}$ then all nontrivial eigenvalues away from 0 have asymptotically Gaussian fluctuations. These fluctuations are governed by a single random variable, which has the interpretation of the total degree of the graph. This extends the result \cite{HLY} on the fluctuations of the extreme eigenvalues from $Np \geq N^{2/9 + \varepsilon}$ down to the optimal scale $Np \geq N^{\varepsilon}$. The main technical achievement of our proof is a rigidity bound of accuracy $N^{-1/2-\epsilon} (Np)^{-1/2}$ for the extreme eigenvalues, which avoids the $(Np)^{-1}$-expansions from \cite{EKYY1,LS1,HLY}. Our result is the last missing piece, added to \cite{EKYY2,LS1, HLY, H19}, of a complete description of the eigenvalue fluctuations of sparse random matrices for $Np \geq N^{\varepsilon}$.
\end{abstract}

\section{Introduction and main results}  \label{sec:2}

Let $\cal A$ be the adjacency matrix of the Erd\H{o}s-R\'enyi graph $\cal G(N,p)$. Explicitly, $\cal A = (\cal A_{ij})_{i,j = 1}^N$ is a symmetric $N\times N$ matrix with independent upper triangular entries $(\cal A_{ij} \col i \leq j)$ satisfying
\begin{equation*}
{\cal A}_{ij}=\begin{cases}
1 & \txt{with probability } p
\\
0 & \txt{with probability } 1-p\,.
\end{cases}
\end{equation*} 
We introduce the normalized adjacency matrix
\begin{equation} \label{1.11}
A\deq \sqrt{\frac{1}{p(1-p)N}}\,\cal A\,,
\end{equation}
where the normalization is chosen so that the eigenvalues of $A$ are typically of order one.

The goal of this paper is to obtain the asymptotic distribution of the extreme eigenvalues of $A$.
The extreme eigenvalues of graphs are of fundamental importance in spectral graph theory and have attracted much attention in the past thirty years; see for instance \cite{Chu, HLW06, Alo98} for reviews. The Erd\H{o}s-R\'enyi graph is the simplest model of a random graph and its adjacency matrix is the canonical example of a sparse random matrix.

Each row and column of $A$ has typically $Np$ nonzero entries, and hence $A$ is sparse whenever $p \to 0$ as $N \to \infty$.  In the complementary dense regime, where $p$ is of order one, $A$ is a Wigner matrix (up to a centring of the entries). The edge statistics of Wigner matrices have been fully understood in \cite{Sosh1,EKYY2,SS,TV2,EYY3,LY}, where it was shown that the distribution of the largest eigenvalue is asymptotically given by the GOE Tracy-Widom distribution \cite{TW1,TW2}.

To discuss the edge statistics of $A$ in the sparse regime, we introduce the following conventions. Unless stated otherwise, all quantities depend on the fundamental parameter $N$, and we omit this dependence from our notation. We write $X \ll Y$ to mean $X = O_\epsilon (N^{-\epsilon} Y)$ for some fixed $\epsilon > 0$. We write $X \asymp Y$ to mean $X = O(Y)$ and $Y = O(X)$.
We denote the eigenvalues of $A$ by $\lambda_1 \leq \cdots \leq \lambda_N$.
The largest eigenvalue $\lambda_N$ of $A$ is its Perron-Frobenius eigenvalue. For $Np \gg 1$, it is typically of order $\sqrt{Np}$, while the other eigenvalues $\lambda_1, \lambda_2, \dots, \lambda_{N-1}$ are typically of order one.

The edge statistics of sparse matrices were first studied in \cite{EKYY1,EKYY2}, where it was proved that when $Np \gg N^{2/3}$ the second largest eigenvalue of $A$ exhibits GOE Tracy–Widom
fluctuations, i.e.
\[
\lim_{N\to \infty} \bb P\big(N^{2/3}(\lambda_{N-1}-\bb E \lambda_{N-1})\leq s\big)=F_1(s)\,,
\] 
where  $F_1(s)$ is the distribution function of the GOE Tracy–Widom distribution. In \cite{LS1}, this result was extended to $Np \gg N^{1/3}$, which it turns out is optimal. Indeed, in \cite{HLY} it was shown that when $N^{2/9}\ll Np \ll N^{1/3}$ the Tracy-Widom distribution for $\lambda_{N - 1}$ no longer holds, and the extreme eigenvalues have asymptotically Gaussian fluctuations. More precisely, in \cite{HLY} it was shown that if $N^{2/9}\ll Np \ll N^{1/3}$ then
\begin{equation} \label{1.0}
\sqrt{\frac{N^2 p}{2}}(\lambda_{N-1}-\bb E \lambda_{N-1}) \todist \cal N(0,1)\,.
\end{equation}

In this paper we show \eqref{1.0} for the whole range $1 \ll Np \ll N^{1/3}$. In fact, we show this for a general class of sparse random matrices introduced in \cite{EKYY1, EKYY2}. It is easy to check that the normalized adjacency matrix $A$ \eqref{1.11} of $\cal G(N,p)$ satisfies the following definition with the choice
\begin{equation} \label{pq}
q\deq \sqrt{Np} \,.
\end{equation}

\begin{definition} [Sparse matrix] \label{def:sperse} Let $1 \leq q \leq \sqrt{N}$. A \emph{sparse matrix} is a real symmetric $N\times N$ matrix $H=H^* \in \bb R^{N \times N}$ whose entries $H_{ij}$ satisfy the following conditions.
	\begin{enumerate}
		\item[(i)] The upper-triangular entries ($H_{ij}\col 1 \leq i \leq j\leq N$) are independent.
		\item[(ii)] We have $\bb E H_{ij}=0$, $ \bb E H_{ij}^2=(1+O(\delta_{ij}))/N$, and $\bb E H_{ij}^4\asymp 1/(Nq^2)$ for all $i,j$.
		\item[(iii)] For any $k\geq 3$, we have
		$\bb E|H_{ij}|^k \leq C_k/ (Nq^{k-2})$ for all $i,j$.
	\end{enumerate}
	We define the random matrix
	$$
	A = H + f \f e \f e^*\,,
	$$
	where $\f e \deq N^{-1/2}(1,1,\dots,1)^*$, and $f \geq 0$.
\end{definition}

For simplicity of presentation, in this paper we focus only on real matrices, although our results and proofs extend to matrices with complex entries with minor modifications which we omit; see also Remark \ref{rem:complex_proof} below.

To describe the fluctuations of the eigenvalues of $A$, we define the random variable
\begin{equation} \label{zzzz}
\cal Z \deq  \frac{1}{N}\tr H^2-1\,.
\end{equation}
Defining
\[
\Sigma \deq \bigg(\frac{1}{N^2}\sum_{i,j}\bb E H_{ij}^4\bigg)^{1/2}\,,
\]
one easily finds
\begin{equation} \label{1.6}
	  \frac{1}{\sqrt{2}\Sigma}\cal Z\todist\cal N(0,1) \quad \mbox{and} \quad \Sigma\asymp \frac{1}{\sqrt{N}q}
	  \,.
\end{equation}

We denote by $\gamma_{\mathrm{sc},i}$ be the $i$th $N$-quantile of the semicircle distribution, which is the limiting empirical eigenvalue measure of $A$ for $Np \gg 1$. Explicitly,
$
\int_{-2}^{\gamma_{\mathrm{sc},i}}\frac{1}{2\pi}\sqrt{4-x^2} \, \dd x =\frac{i}{N}\,.
$

Throughout the following we fix an exponent $\beta \in (0,1/2]$ and set
\begin{equation} \label{def_beta}
q = N^\beta\,.
\end{equation}
If $A$ is the normalized adjacency matrix \eqref{1.11} of $\cal G(N,p)$ then from \eqref{pq} and \eqref{def_beta} we find that the condition $1 \ll Np \ll N^{1/3}$ reads $1 \ll q \ll N^{1/6}$, i.e.\ $\beta \in (0,1/6)$.  We may now state our main result.

\begin{theorem} \label{mainthm}
Fix $\beta \in (0,1/6)$ and set
\begin{equation} \label{delta}
\delta\equiv\delta(\beta)\deq \frac{1}{10} \min \{\beta,1/6-\beta\}\,.
\end{equation}
Let $H$ be as in Definition \ref{def:sperse} with $q$ given by \eqref{def_beta}.
Fix $\epsilon > 0$ and $D > 0$. Then for large enough $N$ we have with probability at least $1 - N^{-D}$
	\begin{equation} \label{1.1}
	\Big|\lambda_i-\bb E \lambda_i-\frac{\gamma_{\mathrm{sc},i}}{2} \cal Z \Big| = O(N^{\epsilon - \delta} \Sigma)
	\end{equation}
for all $1\leq  i \leq N-1$.
\end{theorem}

Theorem \ref{mainthm} implies, for all $i\in \{1,2,\dots,N-1\}$ such that $\gamma_{\mathrm{sc},i}$ is away from 0, that the fluctuations of $\lambda_i$ are simultaneously governed by those of $\cal Z$. In fact, by the rigidity result of \cite[Theorem 2.13]{EKYY1} and a simple moment estimate of $\cal Z$ (see \eqref{cal Z} below), we deduce from \eqref{1.6} and Theorem \ref{mainthm} that under its conditions, with probability at least $1 - N^{-D}$ we have
\begin{equation} \label{lambda_i_fluct}
\lambda_i = \E \lambda_i \pbb{1 + \frac{\cal Z}{2}} + O \p{N^{- \delta/2} \Sigma}
\end{equation}
for all $i = 1, \dots, N - 1$.
Thus, for $1 \ll q \ll N^{1/6}$, the fluctuation of all eigenvalues away from $0$ is given by a global random scaling by the factor $1 + \cal Z/2$.

\begin{remark}
If $f = 0$ in Definition \ref{def:sperse}, i.e.\ $A = H$ is centred, then the conclusion of Theorem \ref{mainthm} holds for all eigenvalues $\lambda_1, \dots, \lambda_N$. Indeed, if $f = 0$ then $A$ and $-A$ both satisfy Definition \ref{def:sperse}, and $\lambda_N(A) = - \lambda_1(-A)$.
\end{remark}

Our main result is a rigidity estimate for the eigenvalues of $A$ with accuracy
\begin{equation*}
\frac{1}{N^{1/2 + \delta/2} q}\,.
\end{equation*}
In contrast, the corresponding rigidity results of \cite{EKYY1, LS1, HLY} have accuracy up to a fixed power of $q^{-1}$: up to $q^{-2}$ in \cite{EKYY1}, $q^{-4}$ in \cite{LS1}, and $q^{-6}$ in \cite {HLY}. For arbitrarily small polynomial values of $q$, the rigidity provided by an expansion up to a fixed power of $q^{-1}$ is not sufficient to analyse the fluctuations of the extreme eigenvalues. Thus, the main technical achievement of our paper is the avoidance of $q^{-1}$-expansions in the error bounds.

\begin{remark}
The variable $\cal Z$ was introduced in \cite{HLY}, where its importance for the edge fluctuations of sparse random matrices was first recognized. Using it, the authors proved \eqref{1.1} for $\beta \in (1/9,1/6)$.
\end{remark}

\begin{remark} \label{rem:scaling}
Let $A$ be the rescaled adjacency matrix \eqref{1.11} of $\cal G(N,p)$. The fluctuations of the eigenvalues of $A$ have a particularly transparent interpretation in terms of the fluctuation of the average degree of $\cal G(N,p)$, or, equivalently, its total number of edges. To that end, denote by $\cal D \deq \frac{1}{N} \sum_{i,j} \cal A_{ij}$ the average degree of $\cal G(N,p)$ and by $d \deq \E \cal D = Np$ its expectation. Defining the randomly rescaled adjacency matrix
\begin{equation} \label{def_A_hat}
\wh A \deq \frac{1}{\sqrt{\cal D}} \cal A\,,
\end{equation}
we claim that under the assumptions of Theorem \ref{mainthm} we have
\begin{equation} \label{wh_A_fluct}
\lambda_i(\wh A) = \E \lambda_i(A) + O(N^{-\delta/2} \Sigma)
\end{equation}
with probability at least $1 - N^{-D}$. Indeed, a short calculation yields $\cal D = d \pb{1 + (1 - p) \cal Z + O (p)}$, from which \eqref{wh_A_fluct} follows using \eqref{lambda_i_fluct} and the bounds $p = O(N^{-\delta} \Sigma)$ and $\abs{\cal Z}^2 = O(N^{-\delta} \Sigma)$ with probability at least $1 - N^{-D}$ (by \eqref{cal Z} below).

In \eqref{wh_A_fluct}, the Gaussian fluctuations \eqref{lambda_i_fluct} present for $\lambda_i(A)$ are absent for $\lambda_i(\wh A)$. Hence, the fluctuations of the eigenvalues of $A$ can be all simultaneously eliminated to leading order by an appropriate \emph{random} rescaling. Note that we can write $A = d^{-1/2} \cal A (1 + O(N^{-\delta} \Sigma))$, in analogy to \eqref{def_A_hat}. Thus, \eqref{wh_A_fluct} states that if one replaces the deterministic normalization $d^{-1/2}$ with the random normalization $\cal D^{-1/2}$ the fluctuations vanish to leading order. In fact, although it is not formulated that way, our proof can essentially be regarded as a rigidity result for the matrix $\wh A$.
\end{remark}

Remark \ref{rem:scaling} is consistent with the fact that for more rigid graph models where the average degree is fixed, $\cal Z$ does not appear: for a random $d$-regular graph, the second largest eigenvalue of the adjacency matrix has Tracy-Widom fluctuations for $N^{2/9}\ll d \ll N^{1/3}$ \cite{BHKY19}. Moreover, in \cite{HLY} it was proved that the second largest eigenvalue of $\wh A$ has Tracy-Widom fluctuations for $q \gg N^{1/9}$.

Theorem \ref{mainthm} trivially implies the following result.

\begin{corollary}
We adopt the conditions in Theorem \ref{mainthm}. Fix $\varepsilon>0$. Define
	\[
	X_{i}\deq \frac{\lambda_{i}-\bb E \lambda_{i}}{\gamma_{\mathrm{sc},i}\Sigma/\sqrt{2}}
	\] 
	for all $i \in \{1,2,.., \floor{(\frac{1}{2}-\varepsilon) N}, \floor{(\frac{1}{2}+\varepsilon) N},\dots,N-1\}\eqd \cal I$.
	We have
	\begin{equation*}
	(X_{i_1},\dots,X_{i_k}) \todist \cal N_k({\f 0}, \cal J)\,,
	\end{equation*} 
	for all fixed $k$ and $i_1,\dots,i_k \in \cal I$. Here $\cal J\in \bR^{k\times k}$ is the matrix of ones, i.e.\ $\cal J_{ij}= 1$ for all $i,j\in\{1,2,\dots,k\}$.
\end{corollary}

Next, we remark on the fluctuations of single eigenvalues inside the bulk. This problem was first addressed in \cite{G2005} for GUE, extended to GOE in \cite{O2010}, and recently extended to general Wigner matrices in \cite{BM18,LS18}. In these works, it was proved that the bulk eigenvalues of Wigner matrices fluctuate on the scale $\sqrt{\log N}/N$. More precisely,
\[
\frac{\mu_i-\gamma_{\mathrm{sc},i}}{\sqrt{\frac{8\log N}{(4-\gamma_{\mathrm{sc},i}^2)N^2}}} \todist \cal N(0,1)
\]
for all bulk eigenvalues $\mu_i$, $\epsilon N \leq i \leq (1 - \epsilon) N$, of a real Wigner matrix. The bulk eigenvalue fluctuation of sparse matrices was studied in \cite{H19}, where it was shown that for fixed $\beta \in (0,1/2)$, there exists $c\equiv c(\beta)>0$ such that with probability at least $1 - N^{-D}$
\begin{equation*}
\Big|\lambda_i-\bb E \lambda_i-\frac{\gamma_{\mathrm{sc},i}}{2}\cal Z \Big| = O(N^{-c} \Sigma)
\end{equation*}
for all bulk eigenvalues $\lambda_i$, $\epsilon N \leq i \leq (1 - \epsilon) N$, of $A$.

In summary, we have the following general picture of fluctuations of eigenvalues for sparse random matrices. The fluctuations of any single eigenvalue consists of two components: a \emph{random matrix} component and a \emph{sparseness} component. The random matrix component is independent of the sparseness and coincides with the corresponding fluctuations of GOE. It has order $N^{-2/3}$ at the edge and order $\sqrt{\log N} / N$ in the bulk. The sparseness component is captured by the random variable $\cal Z$ and has order $1/(\sqrt{N} q)$ throughout the spectrum except near the origin. Thus, the sparseness component dominates in the bulk as soon as $q \ll \sqrt{N}$ and at the edge as soon as $q \ll N^{1/6}$. In fact, our proof suggests that $\cal Z$ is only the leading order such Gaussian contribution arising from the sparseness, and that there is an infinite hierarchy of strongly correlated and asymptotically Gaussian random variables of which $\cal Z$ is the largest and whose magnitudes decrease in powers of $q^{-2}$. In order to obtain random matrix Tracy-Widom statistics near the edge, one would have to subtract all of such contributions up order $N^{-2/3}$. For $q = N^{\beta}$ with $\beta$ arbitrarily small, the number of such terms becomes arbitrarily large.

For completeness, we mention that the bulk eigenvalue statistics have also been analysed in terms of their correlation functions and eigenvalue spacings, which have a very different behaviour from the single eigenvalue fluctuations described above. It was proved in \cite{EKYY1, EKYY2, HLY15, LY15} that the asymptotics of the local eigenvalue correlation functions in the bulk coincide with those of GOE for any $q \gg 1$. Thus, the sparseness has no impact on the asymptotic behaviour of the correlation functions and the eigenvalue spacings.
 
We conclude this section with a few words about the proof. The fluctuations of the extreme eigenvalues are considerably harder to analyse than those of the bulk eigenvalues, and in particular the method of \cite{H19} breaks down at the edge because the self-consistent equations on which it relies become unstable. The key difficulty near the edge is to obtain strong rigidity estimates on the locations of the extreme eigenvalues, while no such estimates are needed in the bulk. Indeed,
 the central step of the proof is Proposition \ref{prop: upper bound} below, which provides an upper bound for the fluctuations of the largest eigenvalue of $H$. This is obtained by showing, for suitable $E$ outside the bulk of the spectrum and $\eta>0$, that  the imaginary part of the Green's function $G(E+\ii\eta)\deq(H-E-\ii \eta)^{-1}$ satisfies $ \im \tr G(E+\ii\eta) \ll 1/\eta$.
Our basic approach is the self-consistent polynomial method for sparse matrices developed in \cite{HLY, LS1}.
Thus, we first obtain a highly precise bound on the self-consistent polynomial $P$ of the Green's function, which provides a good estimate of $\tr G$ outside the bulk. The key observation in this part is that the cancellation built into $P$ persists also in the derivative of $P$. Armed with the good estimate of $\tr G$, our second key idea is to estimate the imaginary part of $P$, which turns out to be much smaller than $P$ itself; from this we deduce strong enough bounds on the imaginary part of $G$. These two estimates together conclude the proof. We refer to Section \ref{sec} below for more details of the proof strategy.

The rest of the paper is organized as follows. In Section \ref{sec:pre} we introduce the notations and previous results that we use in this paper. In Section \ref{sec} we explain the strategy of the proof. In Section \ref{sec44} we prove Theorem \ref{mainthm}, assuming key rigidity estimates at the edge (Proposition \ref{prop: upper bound}) and inside the bulk (Lemma \ref{theorem 2.1}). In Section \ref{sec6.1} we give a careful construction of the self-consistent polynomial $P$ of the Green's function. In Sections \ref{sec3}--\ref{sec5}, we prove Proposition \ref{prop: upper bound}, by assuming several improved estimates for large classes of polynomials of Green's functions. In Section \ref{sec8} we prove Lemma \ref{theorem 2.1}. 
Finally in Section \ref{sec inf} we prove the estimates that we used in Sections \ref{sec3}--\ref{sec5}.

\paragraph{Acknowledgments}
The authors would like to thank Zhigang Bao, Jiaoyang Huang, and Benjamin Schlein for helpful discussions. Y.H.\ gratefully acknowledges partial support from the NCCR SwissMAP and the Swiss National Science Foundation through the Grant 200020\_172623 ``Dynamical and energetic properties of Bose-Einstein condensates''. A.K.\ gratefully acknowledges funding from the European Research Council (ERC) under the European Union’s Horizon 2020 research and innovation programme (grant agreement No.\ 715539\_RandMat) and from the Swiss National Science Foundation through the NCCR SwissMAP grant.

\section{Preliminaries} \label{sec:pre}
In this section we collect notations and tools that will be used.  For the rest of this paper we fix $\beta \in (0,1/6)$ and define $\delta$ as in \eqref{delta}.

Let $M$ be an $N \times N$ matrix. We denote $M^{*n}\deq (M^{*})^n$, $M^{*}_{ij}\deq (M^{*})_{ij} = \ol M_{ji}$, $M^n_{ij}\deq (M_{ij})^n$, and the normalized trace of $M$ by $\ul M \deq \frac{1}{N} \tr M$. We denote the \emph{Green's function of $H$} by
\begin{equation} \label{def_G}
G \equiv G(z)\deq (H-z)^{-1}\,.
\end{equation}
\begin{convention*}
Throughout the paper, the argument of $G$ and of any Stieltjes transform is always denoted by $z \in \C \setminus \R$, and we often omit it from our notation.
\end{convention*}
The Stieltjes transform of the eigenvalue density at $z$ is denoted by $\ul{G}(z)$. For deterministic $z$ we have the differential rule
\begin{equation} \label{diff}
\frac{\partial G_{ij}}{\partial H_{kl}} =-(G_{ik}G_{jl}+G_{il}G_{kj})(1+\delta_{kl})^{-1}.
\end{equation}
If $h$ is a real-valued random variable with finite moments
of all order, we denote by $\cal C_k(h)$ the $k$th cumulant of $h$, i.e.
\[
\cal C_k(h)\deq (-\ii)^k \cdot\big(\partial_{\lambda}^k \log \bb E \e^{\ii  \lambda h}\big) \big{|}_{\lambda=0}\,.
\]
We state the cumulant expansion formula, whose proof is given in e.g. \cite[Appendix A]{HKR}. 
\begin{lemma}[Cumulant expansion] \label{lem:cumulant_expansion}
	Let $f\col\R\to\C$ be a smooth function, and denote by $f^{(k)}$ its $k$th derivative. Then, for every fixed $\ell \in\N$, we have 
	\begin{equation}\label{eq:cumulant_expansion}
	\mathbb{E}\big[h\cdot f(h)\big]=\sum_{k=0}^{\ell}\frac{1}{k!}\mathcal{C}_{k+1}(h)\mathbb{E}[f^{(k)}(h)]+\cal R_{\ell+1},
	\end{equation}	
	assuming that all expectations in \eqref{eq:cumulant_expansion} exist, where $\cal R_{\ell+1}$ is a remainder term (depending on $f$ and $h$), such that for any $t>0$,
	\begin{equation} \label{remainder}
	\cal R_{\ell+1} = O(1) \cdot \bigg(\E\sup_{|x| \le |h|} \big|f^{(\ell+1)}(x)\big|^2 \cdot \E \,\big| h^{2\ell+4} \mathbf{1}_{|h|>t} \big| \bigg)^{1/2} +O(1) \cdot \bb E |h|^{\ell+2} \cdot  \sup_{|x| \le t}\big|f^{(\ell+1)}(x)\big|\,.
	\end{equation}
\end{lemma}

The following result gives bounds on the cumulants of the entries of $H$, whose proof follows from Definition \ref{def:sperse} and the homogeneity of the cumulants.
\begin{lemma} \label{Tlemh}
	For every $k \in \bb N$ we have
	\begin{equation*}
	\cal C_{k}(H_{ij})=O_{k}(1/(Nq^{k-2}))
	\end{equation*}
	uniformly for all $i,j$.
\end{lemma}

We use the following convenient notion of high-probability bound from \cite{EKY2}.
\begin{definition}[Stochastic domination] \label{def:2.3} 
	Let $$X=\pb{X^{(N)}(u)\col N \in \bN, u \in U^{(N)}}\,,\qquad Y=\pb{Y^{(N)}(u)\col N \in \bN, u \in U^{(N)}}$$ be two families of random variables, where $Y^{(N)}(u)$ are nonnegative and $U^{(N)}$ is a possibly $N$-dependent parameter set. We say that $X$ is stochastically dominated by $Y$, uniformly in $u$, if for all (small) $\varepsilon>0$ and (large) $D>0$ we have
	\begin{equation*} 
	\sup\limits_{u \in U^{(N)}}	\bP \left[ \big|X^{(N)}(u)\big| > N^{\varepsilon} Y^{(N)}(u) \right] \le N^{-D}
	\end{equation*} 
	for large enough $N \ge N_0(\varepsilon,D)$. If $X$ is stochastically dominated by $Y$, uniformly in $u$, we use the notation $X \prec Y$, or, equivalently $X=O_{\prec}(Y)$.
\end{definition}

Note that for
	deterministic $X$ and $Y$, $X =O_\prec(Y)$ means $X= O_{\epsilon}(N^{\epsilon}Y)$ for any $\epsilon> 0$. Sometimes we say that an event  $\Xi \equiv \Xi^{(N)}$ holds with \emph{very high probability} if for all $D > 0$ we have $\P(\Xi) \geq 1 - N^{-D}$ for $N \geq N_0(D)$.

By estimating the moments of $\cal Z$ defined in \eqref{zzzz} and invoking Chebyshev's inequality, we find
\begin{equation} \label{cal Z}
\cal Z \prec \frac{1}{\sqrt{N}q}\,.
\end{equation}
We have the following elementary result about stochastic domination.

\begin{lemma} \label{prop_prec}
	\begin{enumerate}
		\item
		If $X_1 \prec Y_1$ and $X_2 \prec Y_2$ then $X_1 X_2 \prec Y_1 Y_2$.
		\item
		Suppose that $X$ is a nonnegative random variable satisfying $X \leq N^C$ and $X \prec \Phi$ for some deterministic $\Phi \geq N^{-C}$. Then $\E X \prec \Phi$.
	\end{enumerate}
\end{lemma}

Fix (a small) $c>0$ and define the spectral domains
\begin{equation} \label{ss}
\f S  \deq \{z=E + \ii \eta \col   |E |\leq 10, N^{-1 + c}\leq\eta \leq 10\}\,, \quad \widetilde{\f S}\deq \{z=E+\ii \eta\col |E|\leq 10, 0<\eta\leq 10\}\,.
\end{equation}
We recall the local semicircle law for Erd\H{o}s-R\'enyi graphs from \cite{EKYY1}.

\begin{proposition}[Theorem 2.8,\cite{EKYY1}] \label{refthm1}
	Let $H$ be a sparse matrix defined as in Definition \ref{def:sperse}, and $m_{\mathrm{sc}}$ be the Stieltjes transform of the semicircle distribution. We have
	\begin{equation*} 
	\max\limits_{i,j}|G_{ij}(z)-\delta_{ij}m_{\mathrm{sc}}(z)| \prec \frac{1}{q}+\sqrt{\frac{1}{N\eta}}
	\end{equation*} 
	uniformly in $z =   
	E+\mathrm{i}\eta \in \f S$. 
\end{proposition}

As a standard consequence  of the local law, we have the complete delocalization of eigenvectors.

\begin{lemma} \label{lem:delocalization}
	Let $\f u_1,\dots,\f u_N$ be the ($L^2$-normalized) eigenvectors of $H$. We have
	\[
    \f u_i(k)^2 \prec \frac{1}{N}
	\]
	uniformly for all $i,k \in \{1,2,\dots,N\}$.
\end{lemma}

\begin{remark}
	Proposition \ref{refthm1} was proved in \cite{EKYY1} under the additional assumption $\bb E H^2_{ii}=1/N$ for all $i$. However, the proof is insensitive to the variance of the diagonal entries, and one can easily repeat the steps in \cite{EKYY1} under the general assumption $\bb E H_{ii}^2=C_i/N$. A weak local law for $H$ with general variances on the diagonal can also be found in \cite{HKM18}.
\end{remark}

\begin{lemma}[Ward identity] \label{Ward}
	We have 
	\begin{equation*}
	\sum_j |G_{ij}|^2 =\frac{\im G_{ii}}{\eta}
	\end{equation*}
	for all $z=E+\ii \eta \in \f S$.
\end{lemma}

The following Lemmas \ref{lem:exp P_0}--\ref{lem 3.9} characterize the asymptotic eigenvalues density of $H$. The proof of the following result is postponed to Section \ref{sec6.1}.
\begin{lemma} \label{lem:exp P_0}
	There exists a deterministic polynomial 
	\[
	P_0(z,x)=1+zx+x^2+\frac{a_2}{q^2}x^4+\frac{a_3}{q^4}x^6+\cdots
	\]	
	of degree $2\ceil{\beta^{-1}}$ such that
	\[
	\bb E P_0(z,\ul{G}(z))\prec \frac{\bb E\im \ul{G}(z)}{N\eta}+\frac{1}{N}
	\]
	uniformly for all deterministic $z \in \f S$. Here $a_2, a_3, \dots$ are real, deterministic, and bounded. They depend on the law of $H$.
\end{lemma}

Lemma \ref{lem:exp P_0} states that when $x$ is replaced with $\ul G(z)$, the expectation of $P_0(z,x)$ is very small. This is because of a cancellation built into $P$, which however holds only in expectation and not with high probability.
The following two results are essentially proved in \cite[Propositions 2.5--2.6]{HLY}, and we state them without proof. We denote by $\C_+$ the complex upper half-plane.

\begin{lemma} \label{lem:m_0} There exists a deterministic algebraic function $m_0\col \bb C_+ \to \bb C_+$ satisfying $P_0(z,m_0(z))=0$, such that $m_0$ is the Stieltjes transform of a deterministic symmetric probability measure $\varrho_0$. We have $\supp \varrho_0=[-L_0,L_0]$, where
	\[
	L_0=2+O(1/q^2)\,.
	\]
	Moreover,
	\[
	\im m_0(z)\asymp \begin{cases} \sqrt{\tau_0+\eta}\quad & \mbox{if} \quad  E \in [-L_0,L_0]
	\\
	\frac{\eta}{\sqrt{\tau_0+\eta}}  \quad &  \mbox{if} \quad  E \notin [-L_0,L_0]\,,
	\end{cases}
	\]
	and
	\[
	|\partial_2 P_0(z,m_0(z))| \asymp\sqrt{\tau_0+\eta}\,, \quad  \quad	|\partial^2_2 P_0(z,m_0(z))|= 2+O(q^{-2})
	\]
	for all $z \in \wt {\f S}$, where $\tau_0 \equiv \tau_0(z)\deq |E^2-L_0^2|$.
\end{lemma}

Next, define $P(z,x)\deq P_0(z,x)+\cal Z x^2$.
\begin{lemma} \label{lem:m} 
	There exists a random algebraic function $m\col \bb C_+ \to \bb C_+$ satisfying $P(z,m(z))=0$, such that $m$ is the Stieltjes transform of a random symmetric probability measure $\varrho$. We have $\supp \varrho=[-L,L]$, where
	\[
	L=L_0+\cal Z+O_{\prec}\Big(\frac{1}{\sqrt{N}q^3}\Big)\,.
	\]
	Moreover,
	\[
	\im m(z)\asymp \begin{cases} \sqrt{\tau+\eta}\quad & \mbox{if} \quad  E \in [-L,L]
	\\
	\frac{\eta}{\sqrt{\tau+\eta}}  \quad &  \mbox{if} \quad  E \notin [-L,L]\,,
	\end{cases}
	\]
	and
	\begin{equation} \label{2.5}
	|\partial_2 P(z,m(z))| \asymp\sqrt{\tau+\eta}\,, \quad  \quad	|\partial^2_2 P(z,m(z))|= 2+O(q^{-2})
	\end{equation}
	for all $z \in \wt {\f S}$, where $\tau \equiv \tau(z) \deq |E^2-L^2|$.
\end{lemma}

Let $\gamma_i$ denote the $i$th $N$-quantile of $\varrho$, i.e.
\[
\int_{-L}^{\gamma_i} \varrho(x)\,\dd x =\frac{i}{N}\,.
\]  
Similarly, let $\gamma_{0,i}$ and $\gamma_{\mathrm{sc},i}$ denote the $i$th $N$-quantile of $\varrho_0$ and the semicircle distribution respectively. We have the following result, whose proof is given in Appendix \ref{appA} below.
\begin{lemma} \label{lem 3.9}
	We have
	\[
	\gamma_i =\gamma_{0,i} +\frac{\gamma_{\mathrm{sc},i}}{2}\cal Z +O_{\prec}\Big(\frac{1}{\sqrt{N}q^3}\Big)
	\]
	uniformly for $i \in \{1,2,\dots,N\}$. Here $\gamma_{0,i}$ is deterministic and satisfies $\gamma_{0,i}=\gamma_{\mathrm{sc},i}+O(q^{-2})$. 
\end{lemma}

\section{Outline of the proof} \label{sec}
In this section we describe the strategy of the proof. The foundation of the proof is the method of recursive self-consistent estimates for high moments using the cumulant expansion introduced in \cite{HK}, building on the previous works \cite{KKP96, Kho1, Kho2}. It  was first used to study sparse matrices in \cite{LS1}, which also introduced the important idea of estimating moments of a self-consistent polynomial in the trace of the Green's function. There, the authors derived a precise local law near the edge and obtained the extreme eigenvalue fluctuations for $p \gg N^{-2/3}$. Subsequently, in \cite{HLY}, by developing the key insight that for $N^{-7/9} \ll p \ll  N^{-2/3}$ the leading fluctuations are fully captured by the random variable $\cal Z$ from \eqref{zzzz}, the authors obtained the extreme eigenvalue fluctuations for $N^{-7/9}\ll p \ll N^{-2/3}$. In this paper we use the same basic strategy as \cite{HLY,LS1}. As in most results on the extreme eigenvalue statistics, the main difficulty is to establish rigidity bounds for the extreme eigenvalues.

The proof of Theorem \ref{mainthm} consists of essentially two separate results: an upper bound on the largest eigenvalue of $H$ (Proposition \ref{prop: upper bound} below) and a rigidity estimate in the bulk (Lemma \ref{theorem 2.1} below). The latter is a modification of \cite[Proposition 2.9]{HLY}, and our main task is to show the former.

We use the random spectral parameter $z=L_0+\cal Z +w$ introduced in \cite{HLY}, where $w=\kappa+\ii \eta$ is deterministic. In order to obtain the estimate of Proposition \ref{prop: upper bound} for the largest eigenvalue of $H$ using the Green's function, one has to preclude the existence of an eigenvalue near $\re z$ for a suitable $z$, which follows provided one can show
\begin{equation} \label{kfc}
\im \ul{G} \ll \frac{1}{N\eta}
\end{equation}
(see \eqref{leijun} and the discussions afterwards for more details). The proof of \eqref{kfc} is the main work of our proof. It relies on the following key new ideas.

\begin{enumerate}
	\item[1.] In the previous works \cite{HLY,LS1}, following the work \cite{EYY3} on Wigner matrices, \eqref{kfc} is always proved using
	\[
	\im \ul{G}\leq \im m + |\ul{G}-m|
	\]
	and estimating the two terms on right-hand side separately. There, the term $|\ul{G}-m|$ is estimated by obtaining an estimate on $\abs{P(z, \ul G)}$ from which an estimate on $|\ul{G}-m|$ follows by inverting a self-consistent equation associated with the polynomial $P$. In our current setting,
	$\abs{\ul G - m}$ turns out to be much larger than $\im \ul{G}$ and hence this approach does not work. Thus, we have to estimate $|\im (\ul G - m)|$ instead of $|\ul{G}-m|$ and take advantage of the fact that it is much smaller than $|\ul{G}-m|$. To that end, we first estimate $\abs{\im P(z, \ul G)}$ by exploiting a crucial cancellation arising from taking the imaginary part, which yields stronger bounds on $\abs{\im P(z,\ul G)}$ than are possible for $\abs{P(z, \ul G)}$.
	
	\item[2.]
	To estimate $\abs{\im (\ul G - m)}$ from $\abs{\im P(z, \ul G)}$, we have to invert a self-consistent equation associated with $\im P$. This equation is only stable provided that $\abs{\ul G - m}$ is small enough.
	\item[3.]
	The main work is to derive a strong enough bound on $\abs{\ul G - m}$ to ensure the stability of the self-consistent equation for $\im (\ul G - m)$. The precision required for this step is much higher than that obtained in \cite{HLY}. Our starting point is the same as in \cite{LS1,HLY}: estimating high moments $\E \abs{P}^{2n}$ of $P \equiv P(z,\ul{G})$ using the cumulant expansion. Note that $P$ is constructed in such a way that the expectation $\E P(z, \ul G)$ is very small by a near-exact cancellation (see Lemma \ref{lem:exp P_0}). In the high moments, the interactions between different factors of $P$ and $\ol P$, corresponding to the fluctuations of $P$, give rise to error terms whose control is the key difficulty of the proof. They cannot be estimated naively and have to be re-expanded to arbitrarily high order using a recursive application of the cumulant expansion. These error terms typically contain the partial derivative $\partial_2 P$ of $P$ in the second argument $\ul G$. As soon as $P$ is differentiated, the cancellation built into $P$ is lost. However, we nevertheless need to exploit remnants of this cancellation that are inherited by these higher-order terms containing derivatives of $P$. We track them by rewriting the partial derivative $\partial_2 P$ in terms of the derivative $\partial_w P = \partial_1 P + \partial_2 P \partial_w \ul G$ and an error term, and then use that $\partial_w$ commutes with the derivative $\frac{\partial}{\partial H_{ij}}$ from the cumulant expansion to obtain a form where the cancellation from the next cumulant expansion is obvious also for the derivative of $P$.
\end{enumerate}

Let us explain the above points in more detail. The proof of \eqref{kfc} contains two steps. The main step is to bound the high moments of $P$ in Proposition \ref{prop1}. We start with
\[
\bb E |P|^{2n}=\frac{1}{N}\sum_{i,j} \bb E H_{ij}G_{ji}P^{n-1}P^{*n}+\bb E (P-\ul{HG}) P^{n-1}P^{*n}\,.
\]
We expand the first term on the right-hand side by Lemma \ref{lem:cumulant_expansion} to get 
\begin{multline} \label{3.2}
\bb E |P|^{2n}=\frac{1}{N}\sum_{k=1}^\ell\frac{1}{k!}\sum_{s=1}^k {k \choose s}\sum_{i,j} \cal C_{k+1}(H_{ij}) \bb E \bigg[  \frac{\partial^s( P^{n-1}P^{*n})}{\partial H_{ij}^s} \frac{\partial^{k-s} G_{ij}}{\partial H_{ij}^{k-s}}\bigg]\\
+\frac{1}{N}\sum_{k=1}^\ell\frac{1}{k!}\sum_{i,j} \cal C_{k+1}(H_{ij}) \bb E \bigg[  \frac{\partial^k G_{ij}}{\partial H_{ij}^k} P^{n-1}P^{*n}\bigg]
+\bb E (P-\ul{HG}) P^{n-1}P^{*n}+\mbox{error}\,.
\end{multline}
Note that the polynomial $P$ is designed such that
\begin{equation*}
\bb E P=\frac{1}{N}\sum_{i,j} \bb E H_{ij}G_{ji}+\bb E (P-\ul{HG}) =\frac{1}{N}\sum_{k=1}^\ell\frac{1}{k!}\sum_{i,j} \cal C_{k+1}(H_{ij}) \bb E \bigg[  \frac{\partial^k G_{ij}}{\partial H_{ij}^k}\bigg]
+\bb E (P-\ul{HG}) +\mbox{error}\approx 0\,,
\end{equation*}
and for the same reason there are cancellations between the second and third terms on right-hand side of \eqref{3.2}.  It turns out that the most dangerous terms on right-hand side of \eqref{3.2} are contained within the first sum. One representative error term, arising from $k = 3$ and $s = 2$ in \eqref{3.2}, is
\begin{equation}  \label{2.1}
\frac{1}{N}\sum_{i,j} \cal C_{4}(H_{ij}) \bb E \Big[  (\partial_2 \ol P) N^{-1}(G^{*2})_{ii}G^*_{jj} G_{ii}G_{jj}|P|^{2n-2}\Big]\,,
\end{equation}
which involves the interaction of $P$ and $\ol P$ and hence depends on the fluctuations of $P$. 

To get a sharp enough estimate of \eqref{2.1}, it is not enough to take absolute value inside the expectation and then estimate $|\partial_2 \ol P|$ and $|N^{-1}(G^{*2})_{jj}|$ by Lemmas \ref{lem:m} and \ref{Ward} respectively. Instead, the key idea is to rewrite the error term, so that it becomes amenable to another expansion step, as\footnote{Here $\partial_{\ol w}$ denotes the antiholomorphic derivative in the complex variable $w$.}
\begin{equation} \label{2.3}
(\partial_2\ol P) N^{-1}(G^{*2})_{ii}G^*_{jj}G_{ii}G_{jj}=N^{-1} \partial_{\ol w} P(\ol z ,\ul{G}^*) \ul{G^*}\,\ul{G}^2+\mbox{error}\,,
\end{equation}
which comes from the approximations
\[
(G^{*2})_{ii}\approx \ul{G^{*2}}\,, \quad G^*_{jj} \approx \ul{G^*}\,, \quad G_{ii},G_{jj} \approx \ul{G}\,, \quad \mbox{and} \quad (\partial_2\ol P) \ul{G^{*2}}= \partial_{\ol w} P(\ol z,\ul{G}^*) -\ul{G^*}
\]
which of course have to be justified.
Ignoring the error terms generated in this process, we find that \eqref{2.1} is reduced to
\begin{multline*}
\frac{1}{N^2}\sum_{i,j}\cal C_4(H_{ij})\bb E\Big[ \partial_{\ol w} P(\ol z,\ul{G}^*) \ul{G^*}\,\ul{G}^2|P|^{2n-2}\Big]\\=\frac{1}{N^3}\sum_{i,j,k,l}\cal C_4(H_{ij})\bb E \Big[ \partial_{\ol w} (H_{kl}G^{*}_{lk}) \ul{G^*}\,\ul{G}^2|P|^{2n-2}\Big]
+\frac{1}{N^2}\sum_{i,j}\cal C_4(H_{ij})\bb E\Big[ \partial_{\ol w} ( \ol P-\ul{HG^*}) \ul{G^*}\,\ul{G}^2|P|^{2n-2}\Big]\,.
\end{multline*}
Since $\partial_{\ol w}$ and $\partial/\partial H_{ij}$ commute, we can again expand the first term on the right-hand side  with Lemma \ref{lem:cumulant_expansion}. In this way the operator $\partial_{\ol w}$ plays no role in our computation, and we can get the desired estimate using the smallness of $\bb E \ol P$. 

A major difficulty in the above argument results from the fact that we need to track carefully the algebraic structure of the error terms arising from repeated applications of simplifications of the form \eqref{2.3}. In particular, such terms occur inside expectations multiplying lots of other terms, and we need to ensure that such approximations remain valid in general expressions. In order to achieve this, we implement the
ideas in \cite{HK2,H19} to construct a hierarchy of Schwinger-Dyson equations for a sufficiently large class of polynomials in the entries of the Green's function.

A desired bound for $P$, Proposition \ref{prop1}, together with the stability analysis of the self-consistent equation associated with $P$ (Lemma \ref{lem6.2} below), yields the key estimate
\begin{equation} \label{rookie}
|\ul{G}-m |\ll \sqrt{\kappa}\,,
\end{equation}
where we recall that $\re z = L_0 + \cal Z + \kappa$.
This estimate is crucial in establishing the stability of the self-consistent equation associated with $\im P$ (see Lemma \ref{mainprop}). More precisely, a Taylor expansion shows
\[
P(z,\ul{G})=\partial_2 P(z,m)(\ul{G}-m)+\frac{1}{2}\partial_2^2 P(z,m)(\ul{G}-m)^2 +\cdots\,.
\]
As $\partial_2^2 P(z,m)\approx 2$, taking the imaginary part and rearranging terms yields
\begin{multline} \label{jkl}
\re \partial_2 P(z,m) \im (\ul{G}-m)=\im P(z,\ul{G})-\im \partial_2 P(z,m) \re (\ul{G}-m)\\
-2  \re \nc (\ul{G}-m)\im (\ul{G}- m)+\cdots\,.
\end{multline}
It can be showed that $\abs{\re \partial_2 P(z,m)} \nc \asymp \sqrt{\kappa}$, and we move this factor to the right-hand side of \eqref{jkl} to obtain a recursive estimate of $\im (\ul{G}-m)$. The third term on right-hand side of \eqref{jkl} says that in order for this estimate to work, we need 
\[
|\ul{G}-m| \ll \abs{\re \partial_2 P(z,m)} \nc \asymp \sqrt{\kappa}\,,
\]
which is exactly \eqref{rookie}.

The final step in showing \eqref{3.2} is to bound the high moments of $\im P$ in Proposition \ref{prop3.3}. As $\im P$ is much smaller than $P$ near the edge, compared to $\bb E |P|^{2n}$, we obtain a much smaller bound for $\bb E \abs{\im P}^{2n}$. The proof is similar to that of Proposition \ref{prop1}, but contains significantly fewer expansions. Combining Proposition \ref{prop3.3} and Lemma \ref{mainprop} leads to our desired estimate of $\im G$, which is
\[
|\im \ul{G}-\im m| \prec \frac{1}{N^{1+\delta}\eta} \,.
\]
As we prove the above for $z$ satisfying $\im m \ll \frac{1}{N\eta}$, we get \eqref{kfc} as desired.

\section{Proof of Theorem \ref{mainthm}} \label{sec44}
In this section we prove Theorem \ref{mainthm}. The key result is the following upper bound on the largest eigenvalue of $H$. The proof is postponed to Section \ref{sec3}.

\begin{proposition} \label{prop: upper bound}
Denoting by $\mu_N$ the largest eigenvalue of $H$, we have
	\begin{equation} \label{goal}
(	\mu_{N}-L_0-\cal Z )_+\prec \frac{1}{N^{1/2+\delta}q}\,.
	\end{equation}
\end{proposition}
We also need the following result to estimate the eigenvalues away from the spectral edges. The proof is postponed to Section \ref{sec8}.
\begin{lemma}\label{theorem 2.1}
Let $\rho$ denote the empirical eigenvalue density of $H$, and set
$$
I_1\deq \left[-\frac{1}{2},L_0+\cal Z-\frac{2}{N^{1/2+\delta}q}\right]\,, \quad I_2\deq \left[L_0+\cal Z-\frac{2}{N^{1/2+\delta}q},L_0+\cal Z+\frac{2}{N^{1/2+\delta}q}\right]\,.
$$ 
We have
\begin{equation} \label{eqnappA}
|\rho(I)-\varrho(I)| \prec \frac{1}{N}+\sqrt{\frac{|I|}{Nq^3}}
\end{equation}
for all $I \subset I_1$ and $I=I_2$.
\end{lemma}

\begin{proof}[Proof of Theorem \ref{mainthm}]
We prove \eqref{1.1} for $i \in \{\floor{N/2}-1,\dots,N-1\}$, and the same analysis works for the other half of the spectrum.
Let $i \in \{\floor{N/2}-1,\dots,N-1\}$ and suppose first that
\begin{equation} \label{zzz}
\gamma_i\,, \lambda_i \geq L_0+\cal Z-\frac{2}{N^{1/2+\delta}q}\,.
\end{equation}
Then trivially we have $\gamma_i \in I_2$ with very high probability. In addition, by the Cauchy interlacing theorem we have $\lambda_i \leq \mu_N$, and together with Proposition \ref{prop: upper bound} we obtain
\[
(\lambda_i-L_0-\cal Z)_+ \prec  \frac{1}{N^{1/2+\delta}q}\,.
\]
Thus by the triangle inequality we get
\begin{equation} \label{eqn:edge}
\lambda_i-\gamma_i \prec \frac{1}{N^{1/2+\delta}q}\,.
\end{equation}

Next, suppose \eqref{zzz} does not hold, namely
\[
\min \{\gamma_i, \lambda_i\}=L_0 + \cal Z -\frac{2}{N^{1/2+\delta}q}-a
\]
for some $a\in (0,3)$. Let $\nu$ be the empirical eigenvalue density of $A$. By the Cauchy interlacing theorem,
\begin{equation*}
|\rho(I) -\nu(I)| \leq \frac{1}{N}
\end{equation*} 
for any $I \subset \bb R$. Together with \eqref{eqnappA}, we have
\begin{equation}\label{huangjinshidai}
|\nu(I)-\varrho(I)| \prec \frac{1}{N}+\sqrt{\frac{|I|}{Nq^3}}
\end{equation}
for all $I \subset I_1$ or $I=I_2$. Let $f(E)\deq \varrho([E,\infty))$. Then
\[
f(\gamma_i)=\frac{N+1-i}{N}=\nu((\lambda_i,\infty])=f(\lambda_i)+O_{\prec}\bigg(\frac{1}{N}+\sqrt{\frac{|I_2|+a}{Nq^3}}\,\bigg)\,,
\]
where in the last step we used \eqref{huangjinshidai}. By the definition of $I_2$ we get $ |f(\gamma_i)-f(\lambda_{i})| \prec N^{-\delta}(|I_2|+a)^{3/2}.$ Together with the uniform square root behaviour of the density of $\varrho$ near $L$ from Lemma \ref{lem:m} we therefore have
\[
f(\lambda_i) \vee f(\gamma_i) \geq c (|I_2|+a)^{3/2}\geq N^{\delta} |f(\gamma_i)-f(\lambda_i)|
\]
with very high probability, where $c > 0$ is a constant. Thus
\[
f(\gamma_i)=f(\lambda_i)(1+O(N^{-\delta}))
\]
with very high probability. Since $f(x)\asymp (L-x)^{3/2}$ for $x \in I_1$,  we deduce that $L-\gamma_i\asymp L-\lambda_{i}$ with very high probability. Moreover,  by Lemma \ref{lem:m} we have $f'(x)\asymp (L-x)^{1/2}$ for $x \in I_1$, which implies $f'(\lambda_i)\asymp f'(\gamma_i)$ with very high probability, and hence
that $f'(x)\asymp f'(\gamma_i)$ with very high probability for any $x$ between $\lambda_{i}$ and $\gamma_i$. Thus the mean value theorem yields
\[
|\lambda_i-\gamma_i|\asymp \frac{|f(\lambda_i)-f(\gamma_i)|}{f'(\gamma_i)}\prec \frac{1}{N\sqrt{|I_2|+a}}+\frac{1}{\sqrt{N}q^{3/2}} \prec \frac{1}{\sqrt{N}q^{3/2}}\,.
\]
Using the above relation, together with \eqref{eqn:edge} and Lemma \ref{lem 3.9}, we conclude that
\[
\lambda_i-\gamma_{0,i}-\frac{\gamma_{\mathrm{sc},i}}{2}\cal Z \prec \frac{1}{N^{1/2+\delta}q}\,.
\]
We then take the expectation using Lemma \ref{prop_prec}, which yields
\[
\bb E \lambda_i -\gamma_{0,i} \prec \frac{1}{N^{1/2+\delta}q}\,.
\]
Combining the above two formulas we have \eqref{1.1} as desired.
\end{proof}

\section{Abstract polynomials and the construction of $P_0$}\label{sec6.1}

\begin{convention*}
Throughout this section, $z \in \f S$ is deterministic.
\end{convention*}

In this section we construct the polynomial $P_0$ and prove Lemma \ref{lem:exp P_0}. It was essentially proved in \cite[Proposition 2.9]{HLY}; here we follow a more systematic approach, based on a class of abstract polynomials in the Green's function entries, which provides an explicit proof. We shall generalize this class further in Section \ref{sec4}.

\subsection{Abstract polynomials, Part I} \label{sec4.1}
We start by introducing a notion of formal monomials in a set of formal variables, which are used to construct $P_0$.  Here the word \emph{formal} refers to the fact that these definitions are purely algebraic and we do not assign any values to variables or monomials.

\begin{definition} \label{def:cal_V}
Let $\{i_1,i_2,\dots\}$ be an infinite set of formal indices. To $\sigma, \nu_1 \in \N$, $\theta \in \R$, $x_1, y_1, \dots, x_\sigma, y_\sigma \in \{i_1, \dots, i_{\nu_1}\}$, and a family $(a_{i_1,\dots,i_{\nu_1}})_{1\leq i_1,\dots,i_{\nu_1}\leq N}$ of uniformly bounded complex numbers we assign a formal monomial
\begin{equation} \label{def_T}
T = a_{i_1,\dots,i_{\nu_1}}N^{-\theta} G_{x_1 y_1} \cdots G_{x_\sigma y_\sigma}\,.
\end{equation}
We denote $\sigma(T) = \sigma$, $\nu_1(T) = \nu_1$, $\theta(T) = \theta$, and $\nu_2(T) \deq \sum_{k = 1}^\sigma \f 1_{x_k \neq y_k}$. Thus, $\sigma(T)$ is the degree of $T$ and $\nu_2(T)$ is the number of off-diagonal $G$s. We denote by $\cal T$ the set of formal monomials $T$ of the form \eqref{def_T}.
\end{definition}

\begin{definition} \label{def:evaluation}
We assign to each monomial $T \in  \cal T$ with $\nu_1 = \nu_1(T)$ its \emph{evaluation}
\begin{equation*}
T_{i_1,\dots,i_{\nu_1}} \equiv T_{i_1,\dots,i_{\nu_1}}(z)\,,
\end{equation*}
which is a random variable depending on an $\nu_1$-tuple $(i_1,\dots,i_{\nu_1})\in \{1,2,\dots,N\}^{\nu_1}$. It is obtained by replacing, in the formal monomial $T$, the formal indices $i_1,\dots,i_{\nu_1}$ with the integers $i_1,\dots,i_{\nu_1}$ and the formal variables $G_{xy}$ with elements $G_{xy}$ of the Green's function \eqref{def_G} with parameter $z$. We define
\begin{equation} \label{5.11}
\cal S (T) \deq \sum_{i_1,\dots,i_{\nu_1}}  T_{i_1,\dots,i_{\nu_1}}\,.
\end{equation}
\end{definition}

Defining the random variable
\begin{equation} \label{WWard}
\Gamma\equiv \Gamma(z)\deq \frac{ \im \ul{G}(z)}{N\eta}\,.
\end{equation}
we have the following result, whose proof is given in Section \ref{sec10.1} below.
\begin{lemma} \label{lem4.2}
	For any fixed $T \in \cal T$ we have
	\begin{equation} \label{417}
	\bb E \,\cal S(T) \prec N^{\nu_1(T)-\theta(T)}(\delta_{ 0 \nu_2(T)}+ \bb E\Gamma+N^{-1}) \,.
	\end{equation}
\end{lemma}

\begin{remark} \label{remark 4.3}
	When $\nu_2(T)\ne 1$, Lemma \ref{lem4.2} is a straightforward consequence of Lemma \ref{Ward} and Proposition \ref{refthm1}. When $\nu_2(T)=1$, naively applying the Ward identity shows 
	\[
	\bb E \,\cal S(T) \prec N^{\nu_1(T)-\theta(T)}\bb E\sqrt{\Gamma}\,.
	\]
	In this case, therefore, Lemma \ref{lem4.2} extracts an additional factor of $\sqrt{\Gamma}$.
\end{remark}

In the sequel we also need the subset
\[
\cal T_0 \deq \{T\in \cal T \col \nu_2(T)=0\}
\]
of formal monomials without off-diagonal entries.
We define an \emph{averaging map} $\cal M$ from $\cal T_0$ to the space of random variables through
\begin{equation} \label{MT}
\cal M(T)=\sum_{i_1,\dots,i_{\nu_1}}a_{i_1,\dots,i_{\nu_1}}N^{-\theta} \ul{G}^{\sigma}\,,
\end{equation}
for $T =a_{i_1,\dots,i_{\nu_1}}N^{-\theta}G_{x_1x_1}G_{x_2x_2}\cdots G_{x_{\sigma}x_{\sigma}} \in \cal T_0$. The interpretation of $\cal M(T)$ is that it replaces all diagonal entries of $G$ in $T$ by their average $\ul G$ and then applies $\cal S$. Note that it is only applied to monomials $T \in \cal T_0$ without off-diagonal entries.
The following result is proved in Section \ref{sec10.2} below. 

\begin{lemma} \label{lem4.22}
	For any fixed $T \in \cal T_0$ there exists $k \in \N$ and $T^{(1)},\dots,T^{(k)} \in \cal T_0$ such that
	\begin{equation}  \label{o2}
	\bb E \,\cal S(T) =\bb E \cal M(T)+\sum_{l=1}^k\bb E\, \cal S(T^{(l)})+O_{\prec}\big(N^{\nu_1(T)-\theta(T)}(\bb E \Gamma+N^{-1})\big)\,.
	\end{equation}
Each $T^{(l)}$ satisfies $\sigma(T^{(l)})-\sigma(T) \in 2\bb N+4$,
	\[
\nu_1(T^{(l)})=\nu_1(T)+1\,, 
	 \quad  \mbox{and} \quad  \theta(T^{(l)})=\theta(T)+1+\beta (\sigma(T^{(l)})-\sigma(T)-2)\,.
	\]
\end{lemma}

Lemma \ref{lem4.22} leads to the following result.

\begin{lemma} \label{lem:nte}
	Fix $T \in \cal T_0$. Fix $r \in \bb N_+$. Then there exists deterministic and bounded $b_1,\dots,b_{r}$ such that
	\begin{equation} \label{5.6}
	\cal M(r,T)\deq \cal M(T)+N^{\nu_1(T)-\theta(T)}\sum_{l=2}^{r} b_{l} N^{-l\beta} \ul{G}^{\sigma(T)+2l}
    \end{equation}
	satisfies
	\[
	\bb E \cal S(T)=\bb E \cal M(r, T) +O_{\prec}\big(N^{\nu_1(T)-\theta(T)}(\bb E \Gamma+N^{-1}+N^{-\beta(r+1)})\big)\,.
	\]
\end{lemma}

\begin{proof}
When $r=1$, the Lemma is trivially true from Lemma \ref{lem4.22}. When $r \geq 2$, the proof  is essentially a repeated use of Lemma \ref{lem4.22}. More precisely, by Lemma \ref{lem4.22}, 
\begin{equation} \label{joyside}
\bb E \,\cal S(T) =\bb E \cal M(T)+\sum_{l=1}^k\bb E\, \cal S(T^{(l)})+O_{\prec}\big(N^{\nu_1(T)-\theta(T)}(\bb E \Gamma+N^{-1})\big)
\end{equation}
for some fixed $k \in \bb N$, where each $T^{(l)}$ satisfies $\sigma(T^{(l)})-\sigma(T) \in 2\bb N+4$,
$\nu_1(T^{(l)})=\nu_1(T)+1$  and $\theta(T^{(l)})=\theta(T)+1+\beta (\sigma(T^{(l)})-\sigma(T)-2)$. As a result, $\bb E \cal S(T^{(l)})=O_{\prec}\big(N^{\nu_1(T)-\theta(T)-2\beta}\big)$ for each $l$. Now we apply Lemma \ref{lem4.22} to each $T^{(l)}$ on RHS of \eqref{joyside}, and get
\begin{equation} \label{joyside2}
	\begin{aligned}
	\bb E \,\cal S(T^{(l)}) &=\bb E \cal M(T^{(l)})+\sum_{l_1=1}^{k_l}\bb E\, \cal S(T^{(l,l_1)})+O_{\prec}\big(N^{\nu_1(T^{(l)})-\theta(T^{(l)})}(\bb E \Gamma+N^{-1})\big)\\
	&=\bb E \cal M(T^{(l)})+\sum_{l_1=1}^{k_l}\bb E\, \cal S(T^{(l,l_1)})+O_{\prec}\big(N^{\nu_1(T)-\theta(T)}(\bb E \Gamma+N^{-1})\big)
	\end{aligned}
\end{equation}
for some fixed $k_l \in \bb N$, where each $T^{(l,l_1)}$ satisfies $\sigma(T^{(l,l_1)})-\sigma(T) \in 2\bb N+8$,
$\nu_1(T^{(l,l_1)})=\nu_1(T)+2$  and $\theta(T^{(l)})=\theta(T)+2+\beta (\sigma(T^{(l,l_1)})-\sigma(T)-4)$. Moreover, by our conditions on $\theta(T^{(l)})$, $\nu_1(T^{(l)})$ and $\theta(T^{(l)})$, we can write
\begin{equation} \label{wer2}
	\sum_{l=1}^{k}\bb E \cal M({T^{(l)}})=N^{\nu_1(T)-\theta(T)}\sum_{l=2}^{r} b_{l,1} N^{-l\beta} \bb E\ul{G}^{\sigma(T)+2l}
\end{equation}
for some deterministic and bounded $b_{l,1},...,b_{r,1}$. Combining \eqref{joyside} -- \eqref{wer2}, we have
\begin{equation} \label{sh}
\begin{aligned}
\bb E \,\cal S(T) =&\,\bb E \cal M(T)+N^{\nu_1(T)-\theta(T)}\sum_{l=2}^{r} b_{l,1} N^{-l\beta} \bb E\ul{G}^{\sigma(T)+2l}\\
&+\sum_{l=1}^k\sum_{l_1=1}^{k_l}\bb E\, \cal S(T^{(l,l_1)})+O_{\prec}\big(N^{\nu_1(T)-\theta(T)}(\bb E \Gamma+N^{-1})\big)\,.
\end{aligned}
\end{equation}
Note that $\bb E \cal S(T^{(l,l_1)})=O_{\prec}\big(N^{\nu_1(T)-\theta(T)-4\beta}\big)$ for each $(l,l_1)$. Thus we can again apply Lemma \ref{lem4.22} to each $T^{(l,l_1)}$ on RHS of \eqref{sh}. Repeating the above steps finitely many times completes the proof.
\end{proof}

Note that we in particular have $\cal M(1,T)=\cal M(T)$ through \eqref{5.6}.

\subsection{The construction of $P_0$ and proof of Lemma \ref{lem:exp P_0}} \label{sec4.2}

We compute
\[
\bb E (1+z\ul{G})=
\bb E\ul{HG}=\frac{1}{N}\sum_{i,j} \bb E H_{ij}G_{ji}\,,
\]
and we shall find a polynomial $Q_0$ such that 
$$
\bb E (1+z\ul{G})+\bb E Q_0(\ul{G}) \prec \bb E \Gamma+\frac{1}{N}\,.
$$ 
We then set $P_0(z,x)=1+zx+Q_0(x)$. Using Lemma \ref{lem:cumulant_expansion} with $h=H_{ij}$ and $f=f_{ji}(H)=G_{ji}$, we have
\begin{equation} \label{?}
\bb E (1+z\ul{G})=\frac{1}{N}\sum_{k=1}^{\ell}\frac{1}{k!}\sum_{i,j}\cal C_{k+1}(H_{ij}) \bb E\frac{\partial^k G_{ji}}{\partial H_{ij}^k}+\frac{1}{N}\sum_{i,j}\bb E\cal R^{(ji)}_{\ell+1}\eqd \sum_{k=1}^{\ell}\widetilde{X}_k+\frac{1}{N}\sum_{i,j}\bb E\cal R^{(ji)}_{\ell+1}\,,
\end{equation}
where $\ell$ is a fixed positive integer to be chosen later, and $\cal R^{(ji)}_{\ell+1}$ is a remainder term defined analogously to $\cal R_{\ell +1}$ in \eqref{remainder}. One can follow, e.g.\,the proof of Lemma 3.4 (iii) in \cite{HKR}, and readily check that
\[
\frac{1}{N}\sum_{i,j}\bb E\cal R^{(ji)}_{\ell+1}=O(N^{-1})
\]
for $\ell\equiv \ell (\beta) $ large enough.  From now on, we always assume the remainder term in cumulant expansion is negligible.

Now let us look at each $\widetilde{X}_k$. For $k=1$, by the differential rule \eqref{diff} and $\cal C_{2}(H_{ij})=1/N$ for $i \ne j$, we have
\begin{equation} \label{4.22}
\widetilde{X}_1=-\frac{1}{N^2}\sum_{i,j} \bb E (G_{ij}^2+G_{ii}G_{jj})-\frac{1}{N^2}\sum_{i} (N\cal C_2(H_{ii})-2)\bb E G_{ii}^2=-\bb E \ul{G}^2+O_{\prec}\Big(\bb E \Gamma+\frac{1}{N}\Big)\,.
\end{equation}
For $k=2$, the most dangerous term is
\[
\frac{1}{N} \sum_{i,j} \cal C_3(H_{ij})\bb E G_{ij}G_{ii}G_{jj}\eqd \sum_{i,j}\bb E T_{ij}\,,
\]
and by $\cal C_3(H_{ij})=O(N^{-1-\beta})$, we see that $\nu_1(T) = 2$, $\theta(T) = 2 + \beta$, and $\nu_2(T)=1$. Thus by Lemma \ref{lem4.2} we have
\[
\sum_{i,j} T_{ij} \prec N^{-\beta}\Big(\bb E \Gamma+\frac{1}{N}\Big)\,.
\]
Other terms in $\widetilde{X}_2$ also satisfy the same bound. Similar estimates can also be done for all even $k$, which yield
\begin{equation} \label{4.3}
\sum_{s=1}^{\ceil{\ell/2}} \widetilde{X}_{2s} \prec \bb E \Gamma+\frac{1}{N}\,.
\end{equation} 
For odd $k \geq 3$, we split
\[
\widetilde{X}_k=\widetilde{X}_{k,1}+\widetilde{X}_{k,2}\,,
\] 
where terms in $\widetilde{X}_{k,1}$ contain no off-diagonal entries of $G$. Use Lemma \ref{lem4.2}, we easily find
\[
\widetilde{X}_{k,2} \prec \bb E \Gamma+\frac{1}{N}\,.
\]
By Lemma \ref{Tlemh}, we see that
\[
\widetilde{X}_{k,1}=\frac{1}{N^{2+(k-1)\beta}}\sum_{i,j} a^{(k)}_{i,j}\bb E G^{(k+1)/2}_{ii}G_{jj}^{(k+1)/2}\,,
\]
where $a^{(k)}_{ij}$ is deterministic and uniformly bounded. Combining with \eqref{?}--\eqref{4.3}, we have
\begin{equation} \label{4.6}
\bb E (1+z\ul{G})+\bb E \ul{G}^2+O_{\prec}\Big(\bb E \Gamma+\frac{1}{N}\Big)=\sum_{s=2}^{\ceil{\ell/2}}\frac{1}{N^{2+(2s-2)\beta}}\sum_{i,j} a^{(2s-1)}_{ij}\bb E G^{s}_{ii}G_{jj}^{s}\eqd \sum_{s=2}^{\ceil{\ell/2}} \bb E \,\cal S(T^{(s)}) \,,
\end{equation}
where
\begin{equation} \label{TS}
T^{(s)}=\frac{1}{N^{2+(2s-2)\beta}} a_{ij}^{(2s-1)} G_{ii}^sG_{jj}^s\,.
\end{equation}
To handle the right-hand side of \eqref{4.6} we invoke Lemma \ref{lem:nte}. Naively, we have
\begin{equation} \label{n1}
\bb E \,\cal S(T^{(s)}) \prec N^{(2-2s)\beta}
\end{equation}
for each $n$. By Lemma \ref{lem:nte}, we can write
\begin{equation} \label{5.13}
\bb E \,\cal S(T^{(s)})= \bb E \cal M\big(\ceil{\beta^{-1}}-2s+2,T^{(s)}\big)+O_{\prec}(\bb E \Gamma+N^{-1})\,.
\end{equation}
Thus \eqref{4.6} becomes
\[
\bb E (1+z\ul{G})+\bb E \ul{G}^2-\sum_{s=2}^{\ceil{\ell/2}}\bb E \cal M\big(\ceil{\beta^{-1}}-2s+2,T^{(s)}\big)=O_{\prec}\Big(\bb E \Gamma+N^{-1}\Big)\,.
\]
Thus we can set
\begin{equation*}
Q_0(\ul{G})= \ul{G^2}-\sum_{s=2}^{\ceil{\ell/2}}\bb E \cal M\big(\ceil{\beta^{-1}}-2s+2,T^{(s)}\big)\,,
\end{equation*} 
and note that $Q_0$ is a polynomial of degree $2\ceil{\beta^{-1}}$. This concludes the proof of Lemma \ref{lem:exp P_0}.

\begin{remark}
	After the construction of $P_0$ (and consequently $P$), we shall construct a more general class of abstract polynomials associated with $P$ in Section \ref{sec 6.2} below.
\end{remark}

\section{Proof of Proposition \ref{prop: upper bound}} \label{sec3}

\begin{convention*}
Throughout this section,
\begin{equation} \label{def_z_random}
z=L_0+\cal Z+w\,,
\end{equation}
where $w=\kappa+\ii \eta$ deterministic.
\end{convention*}

The proof of Proposition \ref{prop: upper bound} consists of two steps; in the first we first estimate $\ul{G}$ and in the second we apply this estimate to obtain a more precise bound of $\im\ul{G}$.
\subsection{Estimate of $\ul{G}$}
Define the spectral domain
\begin{equation} \label{hahaha}
\f Y\equiv \f Y(\delta)=\bigg\{ w=\kappa+\ii \eta \in \bb C_+\col  \frac{N^{-\delta}}{\sqrt{N}q}\leq \kappa\leq 1,  N^{-\delta} N^{-5/8}q^{-1/4}\leq \eta \leq 1\bigg\}\,.
\end{equation} 
As a guide to the reader, the lower bound on $\kappa$ is chosen to be slightly smaller than the scale $\frac{1}{\sqrt{N}q}$ on which the extreme eigenvalues fluctuate; analogously, the lower bound on $\eta$ is chose to be slightly smaller than the scale $N^{-5/8} q^{-1/4}$, which is the solution of the equation $\frac{\eta}{\sqrt{\kappa}} = \frac{1}{N \eta}$ with $\kappa = \frac{1}{\sqrt{N}q}$. Using that $\im m(z) \asymp \frac{\eta}{\sqrt{\kappa}}$ (see Lemma \ref{lem:m}), this choice of lower bound on $\eta$ will allow us to rule out the presence of eigenvalues (see \eqref{xinkuzi} below), and hence establish rigidity.

Recall the definition of $\tau $ in Lemma \ref{lem:m}, and note that the lower bound on $\kappa$ ensures, with very high probability,
\begin{equation} \label{5.22}
\tau(z) =|(L_0+\cal Z+\kappa)^2-L^2|\asymp \kappa
\end{equation} for all $w \in \f Y$. 
The main technical step is the following bound for $P(z,\ul{G})$, whose proof is postponed to Section \ref{sec4}.

\begin{proposition} \label{prop1}
	Let $w \in \f Y$. Suppose $|\ul{G}-m| \prec \Psi$ for some deterministic $\Psi \in [\sqrt{\kappa}N^{-\delta},1]$. Then
	\[
	P(z, \ul{G}) \prec \big(\kappa +\Psi^2\big) N^{-\delta}\,.
	\]
\end{proposition}

\begin{lemma} \label{lem6.2}
	Suppose $\varepsilon\col \f Y\to [N^{-1},N^{-\delta}]$ is a function so that
	$$
	P(z, \ul{G}) \prec \epsilon(w)
	$$
	for all $w \in \f Y$. Suppose $\epsilon(w)$ is Lipschitz continuous with Lipschitz constant $N$ and
	moreover that for each fixed $\kappa$ the function $\eta \to \epsilon(\kappa+\ii \eta)$ is nonincreasing for $\eta>0$. Then
	\[
	|\ul{G}-m| \prec \frac{\epsilon}{\sqrt{|\kappa|
			+\eta+\epsilon}}\,.
	\] 
\end{lemma}
\begin{proof}
	See \cite[Proposition 2.11]{HLY}.
\end{proof}

Combining Proposition \ref{prop1} and Lemma \ref{lem6.2}, we find that for any deterministic $\Psi$ that does not depend on $\eta$ we obtain the implication
\[
|\ul{G}-m| \prec \Psi \implies |\ul{G}-m| \prec \sqrt{\kappa}N^{-\delta}+\Psi N^{-\delta/2}\,.
\]
Using the initial estimate $|\ul{G}-m| \prec 1$ from Proposition \ref{refthm1}, we therefore conclude the key bound
\begin{equation} \label{3.1}
|\ul{G}-m| \prec \sqrt{\kappa} N^{-\delta}\,.
\end{equation}

\subsection{Estimate of $\im \ul{G}$} \label{sec6.2}
Define the subset
\begin{equation} \label{leijun}
\f Y_*\equiv \f Y_*(\delta)=\bigg\{ w=\kappa+\ii \eta \in \bb C_+\col  \frac{N^{-\delta}}{\sqrt{N}q}\leq \kappa\leq 1,  \eta=N^{-\delta} N^{-5/8}q^{-1/4}\bigg\} \subset \f Y\,.
\end{equation}
In this section we show that
\begin{equation} \label{xinkuzi}
\im \ul{G} \prec \frac{1}{N^{1+\delta}\eta}
\end{equation}
for all $w \in \f Y_*$. This immediately implies that whenever $\kappa+\ii \eta \in \f Y_*$, with very high probability there is no eigenvalue in the interval $(L_0+\cal Z+\kappa-\eta,L_0+\cal Z+\kappa+\eta)$. In addition, \cite[Lemma 4.4]{EKYY1} implies
\[
\|H\|-2 \prec \frac{1}{q}\,,
\]
and hence the largest eigenvalue $\mu_N$ of $H$ satisfies \eqref{goal}, and Proposition \ref{prop: upper bound} is proved.

What remains, therefore, is the proof of \eqref{xinkuzi}. In analogy to Proposition \ref{prop1}, we have the following estimate for $ \im P(z, \ul{G})$, whose proof is postponed to Section \ref{sec5}.
\begin{proposition} \label{prop3.3}
	Let $w \in \f Y_*$. Suppose $| \im \ul{G}- \im {m}| \prec \Phi$ for some deterministic $\Phi \equiv \Phi\in [N^{-1-\delta}\eta^{-1},1]$. Then
	\[
	\im P(z, \ul{G}) \prec \Big(\frac{1}{N\eta} +\Phi\Big) \sqrt{\kappa}N^{-\delta}\,.
	\]
\end{proposition}

\begin{lemma} \label{mainprop}
	Let $w \in \f Y_*$. Suppose that
	$$\im P(z, \ul{G}) \prec \epsilon$$
	for some deterministic $\epsilon \in [N^{-1},N^{-\delta}]$. Then 
	\[
	|\im \ul{G}-\im m| \prec \frac{\epsilon}{\sqrt{\kappa}}+\frac{1}{N^{1+\delta}\eta}\,.
	\] 
\end{lemma}
\begin{proof}
	A Taylor expansion gives
	\begin{equation} \label{5.2}
	P(z,\ul{G})=\partial_2 P(z,m)(\ul{G}-m)+\sum_{k=2}^{\ceil{2\beta^{-1}}}\frac{1}{k!}\partial_2^k P(z,m)(\ul{G}-m)^k\,.
	\end{equation}
	Note that $\partial^k_2 P(z,m) \prec 1$, and, recalling the definition of $P$, we find from \eqref{cal Z} and by Lemma \ref{lem:m} that
	\begin{equation} \label{senlidefa}
	\im \partial^k_2 P(z,m) \prec \im z+\im m \prec \frac{\eta}{\sqrt{\kappa+\eta}}  \leq \frac{1}{N^{1+\delta}\eta}
	\end{equation}
	for all $k \geq 1$, where the last inequality holds for any $w \in Y_*$ we have
\begin{equation} \label{5.1}
\frac{\eta}{\sqrt{\kappa+\eta}} \leq \frac{1}{N^{1+\delta}\eta}\,.
\end{equation}
 This implies, for all $k \geq 2$,
	\begin{multline} \label{6.6}
	\im \Big(	\partial_2^k P(z,m)(\ul{G}-m)^k \Big) \prec \frac{1}{N^{1 +\delta}\eta}|\ul{G}-m|^k+ |\im \ul{G}-\im m||\ul{G}-m|^{k-1}\\
	\prec\Big( \frac{1}{N^{1 +\delta}\eta} +|\im \ul{G}-\im m|\Big)\sqrt{\kappa}N^{-\delta}\,,
	\end{multline}
	where in the second step we used \eqref{3.1}. Taking imaginary part of \eqref{5.2} and rearranging the terms, we have
	\begin{multline*}
	\re \partial_2 P(z,m) \im (\ul{G}-m)=\im P(z,\ul{G})-\im \partial_2 P(z,m) \re (\ul{G}-m)\\+O_{\prec}(\sqrt{\kappa}N^{-\delta})\Big( \frac{1}{N^{1 +\delta}\eta} +|\im \ul{G}-\im m|\Big)\,.
	\end{multline*}
	Note that $|\im \partial_2 P(z,m)| \asymp |\im m| \ll \sqrt{\kappa}$, and by \eqref{2.5} we have $|\partial_2 P(z,m)|\asymp\sqrt{\kappa}$. Thus
	\[
	|\re\partial_2 P(z,m)|\asymp \sqrt{\kappa} \,,
	\]
   and together with \eqref{3.1} and \eqref{senlidefa} we have
	\begin{multline*}
	|\im \ul{G}-\im m| \prec \frac{\abs{\im P(z,\ul{G})}}{\sqrt{\kappa}}+\frac{1}{\sqrt{\kappa}}\frac{1}{N^{1+\delta}\eta}\sqrt{\kappa}N^{-\delta}\\+\frac{1}{\sqrt{\kappa}}\sqrt{\kappa}N^{-\delta}\Big( \frac{1}{N^{1 +\delta}\eta} +|\im \ul{G}-\im m|\Big)
	\prec \frac{\epsilon}{\sqrt{\kappa}}+\frac{1}{N^{1+\delta}\eta}+N^{-\delta} |\im \ul{G}-\im m|\,.
	\end{multline*}
	This yields the claim.
\end{proof}
From Proposition \ref{prop3.3} and Lemma \ref{mainprop} we obtain the implication 
\begin{equation} \label{im_G_imm}
|\im \ul{G}-\im m| \prec \Phi \implies |\im\ul{G}-\im m| \prec \frac{1}{N\eta}N^{-\delta}+\Phi N^{-\delta}\,.
\end{equation}
Iterating \eqref{im_G_imm} $O(1/\delta)$ times and recalling Definition \ref{def:2.3} yields
\begin{equation} \label{3.3}
|\im \ul{G}-\im m| \prec \frac{1}{N^{1+\delta}\eta} 
\end{equation}
for all $w \in \f Y_*$. Since
\[
\im m  \asymp\frac{\eta }{\sqrt{\eta+\kappa}} \leq \frac{1}{N^{1+\delta}\eta}\,,
\]
we thus conclude \eqref{xinkuzi}. This concludes the proof of Proposition \ref{prop: upper bound}.

\section{Proof of Proposition \ref{prop1}}\label{sec4}
\begin{convention*}
Throughout this section, $z$ is given by \eqref{def_z_random}, where $w \in \f Y$ is deterministic.
\end{convention*}

Fix $n \in \bb N_+$ and set
\[
\cal P \deq \|P(z,\ul{G})\|_{2n}=\Big(\bb E |P(z,\ul{G})|^{2n}\Big)^{\frac{1}{2n}}\,, \quad \cal E \deq \big(\kappa +\Psi^2\big) N^{-\delta}\,.
\]
We shall show, for any fixed $n \in \N$, that
\begin{equation} \label{goal sec7}
\bb E |P(z,\ul{G})|^{2n}=\cal P^{2n} \prec \cal E^{2n}\,,
\end{equation}
from which Proposition \ref{prop1} follows by Chebyshev's inequality.  The rest of this section is therefore devoted to the proof of \eqref{goal sec7}.

Set
\[
Q_0(\ul{G})\deq P(z,\ul{G})-(1+z\ul{G}+\cal Z \ul{G}^2)\,, \quad Q(\ul{G})\deq P(z,\ul{G})-(1+z\ul{G})\,,
\]
and abbreviate 
\begin{equation} \label{eqn PQQ_0}
P\equiv P(z,\ul{G})\,,\quad P'\deq \partial_2 P(z,\ul{G})\,, \quad Q_0=Q_0(\ul{G}) \quad \mbox{and} \quad Q=Q(\ul{G})\,.
\end{equation}
Note that argument $z$ of $G$ is random, and
\begin{equation}\label{diffz}
\frac{\partial G_{kl}}{\partial H_{ij}}=-(G_{ki}G_{jl}+G_{li}G_{jk})(1+\delta_{ij})^{-1}+4N^{-1}(G^2)_{kl}H_{ij}(1+\delta_{ij})^{-1}\,,
\end{equation}
and as a result
\begin{equation} \label{diffP}
\frac{\partial P}{\partial H_{ij}}= (-2P'N^{-1}(G^2)_{ij}+4P'N^{-1}H_{ij}\ul{G^2}+4N^{-1}H_{ij}\ul{G}^2)(1+\delta_{ij})^{-1}\,.
\end{equation}
We define the parameter
\begin{equation} \label{Upsilon}
\Upsilon \deq \frac{\Psi+\sqrt{\kappa+\eta}}{N\eta}\,.
\end{equation}
Recalling the random variable $\Gamma$ from \eqref{WWard}, we find
\begin{equation} \label{theshy}
\Gamma \prec \Upsilon\,.
\end{equation}
Moreover, we have
\begin{equation} \label{ruo}
(\Psi+\sqrt{\kappa+\eta})\Upsilon\geq \sqrt{\kappa+\eta}\cdot \frac{\sqrt{\kappa+\eta}}{N\eta} \geq N^{-1}\,.
\end{equation}
The next lemma collects basic estimates for the derivatives of $P$.

\begin{lemma} \label{lemP}
	Under the assumptions of Proposition \ref{prop1}, for any fixed $k \in \N_+$ we have
	\begin{equation} \label{1}
	P' \prec \Psi+\sqrt{|\kappa|+\eta}\,, \quad	 \quad \bigg|\frac{\partial^k \ul{G}}{\partial H_{ij}^k}\bigg| \prec  \max_{x,y} N^{-1}|(G^2)_{xy} |\prec \Upsilon
	\end{equation}
	and
	\begin{equation} \label{2}
	\frac{\partial^k P}{\partial H_{ij}^k} \prec \Big(\Psi+\sqrt{\abs{\kappa}+\eta}\,\Big)\Upsilon\,.
	\end{equation}
\end{lemma}

\begin{proof}
	By the mean value theorem,
	\[
	P'=\partial_2 P(z,\ul{G})=\partial_2P(z,m)+\partial_2^2 P(z,\xi)(m-\ul{G})
	\]
	for some $\xi$ between $m$ and $\ul{G}$. Then the first estimate in \eqref{1} is proved using Lemma \ref{lem:m} and \eqref{5.22}. The second estimate in \eqref{1} is proved by Lemmas \ref{lem:delocalization} and \ref{Ward}. By \eqref{diffP} and \eqref{1}, one easily checks that
	\[
	\frac{\partial^k P}{\partial H_{ij}^k} \prec \sup_{x,y}P' N^{-1}|(G^2)_{xy}|+\sup_{x,y}\big(N^{-1}|(G^2)_{xy}|\big)^2+N^{-1}\prec (\Psi+\sqrt{\kappa+\eta})\Upsilon+N^{-1}\,,
	\]
	and combing with \eqref{ruo} one concludes \eqref{2}.
 \end{proof}

\subsection{The first expansion}
By $(H-z)G=I$, we have
   \[
	\bb E |P|^{2n}=\bb E \Big(\ul{HG} +\cal Z \ul{G}^2+Q_0 \Big) P^{n-1}P^{*n}=\bb E \Big(\cal Z \ul{G}^2+Q_0 \Big) P^{n-1}P^{*n}+\frac{1}{N} \sum_{i,j} \bb E H_{ij}G_{ji}P^{n-1}P^{*n}\,.
	\]
	We use Lemma \ref{lem:cumulant_expansion} to calculate the last term. By setting $h=H_{ij}$, $f=f_{ji}(H)=G_{ji}P^{n-1}P^{*n}$, we get
	\begin{equation} \label{KKK}
	\frac{1}{N} \sum_{i,j} \bb E H_{ij}G_{ji}P^{n-1}P^{*n}=\frac{1}{N}\sum_{k=1}^\ell\frac{1}{k!}\sum_{i,j} \cal C_{k+1}(H_{ij}) \bb E \bigg[  \frac{\partial^k (G_{ij}P^{n-1}P^{*n})}{\partial H_{ij}^k} \bigg]+O_{\prec}(N^{-4n})
	\end{equation}
	where, as in \eqref{?}, we choose a large enough $\ell\in \bb N_+$ such that the remainder term is is negligible. By splitting the differentials in \eqref{KKK} basing on if $P,\ol P$ are differentiated, we have 
	\begin{multline} \label{7.1}
	\bb E |P|^{2n}	=\bb E Q_0   P^{n-1}P^{*n}+\bb E\cal Z \ul{G}^2P^{n-1}P^{*n}+\frac{1}{N}\sum_{k=1}^\ell\frac{1}{k!}\sum_{i,j} \cal C_{k+1}(H_{ij}) \bb E \bigg[  \frac{\partial^k G_{ij}}{\partial H_{ij}^k} P^{n-1}P^{*n}\bigg]\\
	+\frac{1}{N}\sum_{k=1}^\ell\frac{1}{k!}\sum_{s=1}^k {k \choose s}\sum_{i,j} \cal C_{k+1}(H_{ij}) \bb E \bigg[  \frac{\partial^s( P^{n-1}P^{*n})}{\partial H_{ij}^s} \frac{\partial^{k-s} G_{ij}}{\partial H_{ij}^{k-s}}\bigg]+O_{\prec}(N^{-4n})\\
	\eqd\mbox{(I)}+\mbox{(II)}+\mbox{(III)}+\mbox{(IV)}+O_{\prec}(N^{-4n})\,.
	\end{multline}
	 We have the following result, which handles the terms on right-hand side of \eqref{7.1} and directly implies \eqref{goal sec7}.
	
\begin{lemma} \label{lem7.1}
	Let $\mathrm{(I)} - \mathrm{(IV)}$ be as in \eqref{7.1}. We have
\begin{equation} \label{7.2}
\mathrm{(II)}+\mathrm{(IV)}\prec\sum_{r=1}^{2n}\cal E^{r} \cal P^{2n-r}
\end{equation}
as well as
\begin{equation} \label{7.3}
\mathrm{(I)}+\mathrm{(III)}\prec \sum_{r=1}^{2n}\cal E^{r} \cal P^{2n-r}\,.
\end{equation}
\end{lemma}

The rest of Section \ref{sec4} is devoted to showing Lemma \ref{lem7.1}. To simplify notation, we drop the complex conjugates in (I)--(IV)  (which play no role in the subsequent analysis), and estimate the quantities
\begin{equation} \label{pengpeng}
\mbox{(II')}+\mbox{(IV')}\deq  \bb E\cal Z \ul{G}^2P^{2n-1}	+\frac{1}{N}\sum_{k=1}^\ell\frac{1}{k!}\sum_{s=1}^k {k \choose s}\sum_{i,j} \cal C_{k+1}(H_{ij}) \bb E \bigg[  \frac{\partial^s P^{2n-1}}{\partial H_{ij}^s} \frac{\partial^{k-s} G_{ij}}{\partial H_{ij}^{k-s}}\bigg]
\end{equation}
and
\begin{equation} \label{dingman}
\mbox{(I')}+\mbox{(III')}\deq \bb E Q_0 P^{2n-1}+\frac{1}{N}\sum_{k=1}^\ell\frac{1}{k!}\sum_{i,j} \cal C_{k+1}(H_{ij}) \bb E \bigg[  \frac{\partial^k G_{ij}}{\partial H_{ij}^k} P^{2n-1}\bigg]\,.
\end{equation}

\subsection{Abstract polynomials, Part II} \label{sec 6.2}
In order to estimate \eqref{pengpeng} and \eqref{dingman}, we introduce the following class of abstract polynomials, which generalizes the class $\cal T$ from Definition \ref{def:cal_V}.

\begin{definition}\label{def:cal_T}
Let $\{i_1,i_2,\dots\}$ be an infinite set of formal indices. To integers $s,k, \nu_1,\nu_3 \in \N$, digits $\nu_4,\nu_5 \in \{0,1\}$ satisfying $\nu_4 \leq \nu_5$, a real number $\theta \in \R$, formal indices $x,y,x_1, y_1, \dots, x_\sigma, y_\sigma \in \{i_1, \dots, i_{\nu_1}\}$, and a family $(a_{i_1,\dots,i_{\nu_1}})_{1\leq i_1,\dots,i_{\nu_1}\leq N}$ of uniformly bounded complex numbers we assign a formal monomial
\begin{equation} \label{def_V}
V = a_{i_1,\dots,i_{\nu_1}}N^{-\theta} (P')^{\nu_4}(N^{-1}(G^2)_{xy})^{\nu_5} G_{x_1y_1}G_{x_2y_2}\cdots G_{x_{k}y_{k}}\ul{G}^s P^{\nu_3}\,,
\end{equation}
We denote $\sigma(V) = s+k$, $\nu_i(V) = \nu_i$ for $i = 1,3,4,5$, $\theta(V) = \theta$, and \begin{equation*}
\nu_2(V) \deq \f 1_{x \neq y} + \sum_{l = 1}^k \f 1_{x_l \neq y_l}\,.
\end{equation*}
We denote by $\cal V$ the set of formal monomials $V$ of the form \eqref{def_V}.
\end{definition}

We extend the evaluation from Definition \ref{def:evaluation} to the set $\cal V$, and denote the evaluation of $V$ as in \eqref{def_V} by $V_{i_1, \dots, i_{\nu_1}}$. We also extend the operation $\cal S$ from \eqref{5.11} to $\cal V$.

The next lemma is an analogue of Lemma \ref{lem4.2}, whose proof is postponed to Section \ref{sec10.3}.

\begin{lemma} \label{lem7.4}
Let $V \in \cal V$ and abbreviate $\nu_i = \nu_i(V)$ and $\theta = \theta(V)$. Suppose that $\nu_2 \ne 0$.
\begin{enumerate}
\item
We have
\begin{multline*}
\bb E \cal S (V) \prec N^{\nu_1-\theta}(\Psi+\sqrt{\kappa+\eta})^{\nu_4}  (N\eta)^{-\nu_5}\Upsilon\bb E |P^{\nu_3}|
\\
+\sum_{t=1}^{\nu_3}  N^{\nu_1-\theta}(\Psi+\sqrt{\kappa+\eta})^{\nu_4} \Upsilon^{\nu_5}((\Psi+\sqrt{\kappa+\eta}) \Upsilon)^{t}\bb E |P^{\nu_3-t}|\,.
\end{multline*}
\item
Moreover, when $\nu_4(V)=\nu_5(V)=0$, we have the stronger estimate
\[
\bb E \cal S (V) \prec N^{\nu_1-\theta}\Upsilon\bb E |P^{\nu_3}|+ N^{\nu_1-\theta}\Upsilon^{2}\bb E |P^{\nu_3-1}|+\sum_{t=2}^{\nu_3}  N^{\nu_1-\theta}((\Psi+\sqrt{\kappa+\eta}) \Upsilon)^{t}\bb E |P^{\nu_3-t}|\,.
\]
\end{enumerate}
\end{lemma}

In the sequel, we also need the subset
\[
\cal V_0 \deq \{V\in \cal V \col \nu_2(V)=0, \nu_4(V)=\nu_5(V)\}\,.
\]
In analogy to \eqref{MT}, we define an \emph{averaging map} $\cal M$ from $\cal V_0$ to the space of random variables through
\[ 
\cal M(V)=\sum_{i_1,\dots,i_{\nu_1}}a_{i_1,\dots,i_{\nu_1}}N^{-\theta} (P'N^{-1}\ul{G^2})^{\nu_4}\,\ul{G}^{s+k}P^{\nu_3}
\]
for
\begin{equation*}
V =a_{i_1,\dots,i_{\nu_1}}N^{-\theta}(P'N^{-1}(G^2)_{xx})^{\nu_4}G_{x_1x_1}G_{x_2x_2}\cdots G_{x_{k}x_{k}}\ul{G}^sP^{\nu_3}\,.
\end{equation*}
The following is an analogue of Lemma \ref{lem4.22}, whose proof is given in Section \ref{sec10.4}. 

\begin{lemma}  \label{lem7.5}
	Let $V \in \cal V_0$. There exist $V^{(1)},\dots,V^{(k)} \in \cal V_0$ such that, abbreviating $\nu_i = \nu_i(V)$ and $\theta = \theta(V)$,
	\begin{multline*} 
	\bb E \,\cal S(V) =\bb E \cal M(V)+\sum_{l=1}^k\bb E\, \cal S(V^{(l)})\\ +O_{\prec}\big( N^{\nu_1-\theta}\Upsilon^{1+\nu_4}\bb E |P^{\nu_3}|\,\big)+\sum_{t=1}^{\nu_3}  O_{\prec}\big(N^{\nu_1-\theta}((\Psi+\sqrt{\kappa+\eta}) \Upsilon)^{\nu_4+t}\bb E |P^{\nu_3-t}|\,\big)\,,
	\end{multline*}
	where $k$ is fixed, and each $V^{(l)}$ satisfies $ V^{(l)} \in \cal V_0$, $\sigma(V^{(l)})- \sigma(V) \in 2\bb N+4$, $\nu_i(V^{(l)})=\nu_i(V)$ for $i=2,3,4,5$,
	\[
	\nu_1(V^{(l)})=\nu_1(V)+1\,,
	\quad  \mbox{and} \quad  \theta(V^{(l)})=\theta(V)+1+\beta (\sigma(V^{(l)})-\sigma(V)-2)\,.
	\]
	As a result, each $V^{(l)}$ satisfies
	\[
	\bb E \cal S (V^{(l)}) \prec  N^{\nu_1(V)-\theta(V)-2\beta}(\Psi+\sqrt{\kappa+\eta}) \Upsilon\bb E |P^{\nu_3}|\,.
	\]
\end{lemma}

Repeatedly using Lemma \ref{lem7.5}, and together with \eqref{theshy}, we obtain the following result.

\begin{lemma}  \label{lem7.6}
	Let $V \in \cal V_0$ and abbreviate $\nu_i = \nu_i(V)$, $\theta = \theta(V)$, and $\sigma=\sigma(V)$. Then there exist deterministic uniformly bounded $b_1,\dots,b_{\ceil{\beta^{-1}}}$ such that
	\begin{equation*}
	\cal M_{\infty}(V)\deq M(V)+N^{\nu_1-\theta}\sum_{l=2}^{\ceil{\beta^{-1}}} b_{l} N^{-l\beta} (P'N^{-1}\ul{G^2})^{\nu_4}\,\ul{G}^{\sigma+2l} P^{\nu_3}
	\end{equation*}
	satisfies
	\[
	\bb E \cal S(V)=\bb E \cal M_{\infty}(V) +O_{\prec}\big( N^{\nu_1-\theta}\Upsilon^{1+\nu_4}\bb E |P^{\nu_3}|\,\big)+\sum_{t=1}^{\nu_3}  O_{\prec}\big(N^{\nu_1-\theta}((\Psi+\sqrt{\kappa+\eta}) \Upsilon)^{\nu_4+t}\bb E |P^{\nu_3-t}|\,\big)\,.
	\]
\end{lemma}

Finally, we have the following extension of Lemma \ref{lem:nte}, which is proved in Section \ref{sec10.5}.
\begin{lemma} \label{lem:ntt}
	Fix $r,u,v\in \bb N$. Let $T \in \cal T_0$ and let $\cal M(r,T)$ be as in Lemma \ref{lem:nte}. Then 
\begin{multline*}
\bb E [\partial_{w}(\cal S(T))\ul{G}^u P^v]=\bb E [\partial_{w}(\cal M(r, T)) \ul{G}^u P^v] +O_{\prec}\big(N^{\nu_1(T)-\theta(T)+1}\Upsilon((N\eta)^{-1}+N^{-\beta(r+1)})\bb E |P|^v\big)\\
+\sum_{t=1}^v O_{\prec}\big(N^{\nu_1(T)-\theta(T)+1}\Upsilon((\Psi+\sqrt{\kappa+\eta})\Upsilon)^t\bb E |P|^{v-t}\big)\,.
\end{multline*}
\end{lemma}

\subsection{The computation of (IV') in \eqref{pengpeng}} \label{sec7.2}
   We write $\mbox{(IV')}=\sum_{k=1}^lX_k$, where
	\begin{equation} \label{X_k}
	X_k\deq\frac{1}{N}\frac{1}{k!}\sum_{s=1}^k {k \choose s}\sum_{i,j} \cal C_{k+1}(H_{ij}) \bb E \bigg[  \frac{\partial^s P^{2n-1}}{\partial H_{ij}^s} \frac{\partial^{k-s} G_{ij}}{\partial H_{ij}^{k-s}}\bigg]\,.
	\end{equation}
	
\subsubsection{The estimate of $X_1$} \label{sec7.3.1}	
By \eqref{diffP} and $\cal C_2(H_{ij})=N^{-1}(1+O(\delta_{ij}))$, we have
\begin{multline} \label{7.13}
X_1=\frac{1}{N}\sum_{i,j}\cal C_2(H_{ij})  \bb E \bigg[  \frac{\partial P^{2n-1}}{\partial H_{ij}}  G_{ij}\bigg]=\frac{2n-1}{N}\sum_{i,j}\cal C_2(H_{ij})  \bb E \bigg[  \frac{\partial P}{\partial H_{ij}} P^{2n-2}G_{ij} \bigg]\\
=\frac{2n-1}{N^2}\sum_{i,j} \bb E (-2P'N^{-1}(G^2)_{ij}+4P'N^{-1}H_{ij}\ul{G^2}+4N^{-1}H_{ij}\ul{G}^2)P^{2n-2}G_{ij}\\+\frac{2n-1}{N^2}\sum_{i} \cal (N \cal C_2(H_{ii})-2) \bb E \big(-2P'N^{-1}(G^2)_{ii}+4P'N^{-1}H_{ii}\ul{G^2}+4N^{-1}H_{ii}\ul{G}^2\big)P^{2n-2}G_{ij}\,.
\end{multline}
Estimating the last term using Lemma \ref{lemP}, we conclude
\begin{equation} \label{7.19}
X_1=\frac{2n-1}{N^2}\bb E (-2P'\ul{G^3}+4P'\ul{HG}\,\ul{G^2}+4\ul{HG}\,\ul{G}^2)P^{2n-2}+O_{\prec}\big(N^{-1}(\Psi+\sqrt{\kappa+\eta}) \Upsilon\big)\bb E |P^{2n-2}|\,.
\end{equation}
By $HG=zG+I$ and $z \prec  1$, we deduce that $\ul{HG} \prec 1$. In addition, it is easy to check that $\ul{G^3} \prec \Upsilon (N\eta)^{-1}$. Thus the first term on right-hand side of \eqref{7.19} can be estimated by
\[
O_{\prec}\Big((\Psi+\sqrt{\kappa+\eta})\Upsilon N^{-1}\eta^{-1}+(\Psi+\sqrt{\kappa+\eta})\Upsilon N^{-1}+N^{-2}\Big) \bb E |P^{2n-2}|\prec \Upsilon^2 \bb E |P^{2n-2}|\,.
\]
As a result,
\[
X_1 \prec \Upsilon^2 \bb E |P^{2n-2}| \leq \Upsilon^2  \cal P^{2n-2}\,.
\]
where in the second step we used H\"{o}lder's inequality. Since
\[
\Upsilon =\frac{\Psi+\sqrt{\kappa+\eta}}{N\eta} \leq \frac{\Psi+\sqrt{\kappa}}{N\eta}+\frac{1}{N\sqrt{\eta}} \leq  (\Psi^2+\kappa) N^{-\delta} +\frac{1}{N^2\eta^2}N^{\delta} +\frac{1}{N\sqrt{\eta}}\prec \cal E  
\]
for all $w \in \f Y$, we have $X_1 \prec \cal E^2 \cal P^{2n-2}$ as desired.

\subsubsection{The estimate of $X_2$} \label{sec7.2.3}
Let us split
\[
X_2=X_{2,1}+X_{2,2}\deq \frac{1}{2N}\sum_{i,j} \cal C_{3}(H_{ij}) \bb E \bigg[  \frac{\partial^2 P^{2n-1}}{\partial H_{ij}^2}  G_{ij}\bigg]+\frac{1}{N}\sum_{i,j} \cal C_{3}(H_{ij}) \bb E \bigg[  \frac{\partial P^{2n-1}}{\partial H_{ij}} \frac{\partial G_{ij}}{\partial H_{ij}}\bigg]\,.
\]
Since $\cal C_3(H_{ij})=O(N^{-1-\beta})$, we have
\begin{multline} \label{7.14}
X_{2,1}\prec \frac{1}{N^{2+\beta}}\sum_{i,j}\bb E\Big|\frac{\partial^2 P}{\partial H_{ij}^2} P^{2n-2}G_{ij}\Big|+\frac{1}{N^{2+\beta}}\sum_{i,j}\bb E\Big|\Big(\frac{\partial P}{\partial H_{ij}}\Big)^2 P^{2n-3}G_{ij}\Big|\\
\prec \frac{(\Psi+\sqrt{\kappa+\eta})\Upsilon}{N^{2+\beta}}\sum_{i,j}\bb E|P^{2n-2}G_{ij}|+ \frac{((\Psi+\sqrt{\kappa+\eta})\Upsilon)^2}{N^{2+\beta}}\sum_{i,j}\bb E|P^{2n-3}G_{ij}|\\
 \prec \frac{((\Psi+\sqrt{\kappa+\eta})\Upsilon)\sqrt{\Upsilon}}{N^{\beta}}\bb E|P^{2n-2}|+ \frac{((\Psi+\sqrt{\kappa+\eta})\Upsilon)^2\sqrt{\Upsilon}}{N^{\beta}}\bb E|P^{2n-3}|\,,
\end{multline}
where in the second and third step we used Lemma \ref{lemP} and Lemma \ref{Ward} respectively. Note that
\begin{equation} \label{7.17}
 \frac{(\Psi+\sqrt{\kappa+\eta})\Upsilon\sqrt{\Upsilon}}{N^{\beta}} = \frac{(\Psi+\sqrt{\kappa+\eta})^{5/2}}{(N\eta)^{3/2}N^{\beta}} \prec  \frac{(\Psi+\sqrt{\kappa})^4}{N^{\beta}}+\frac{1}{(N\eta)^4N^{\beta}}+\frac{1}{N^{3/2}\eta^{1/4}N^{\beta}} \prec \cal E^2\,,
\end{equation}
and similarly
\[
 \frac{((\Psi+\sqrt{\kappa+\eta})\Upsilon)^2\sqrt{\Upsilon}}{N^{\beta}} \prec \cal E^3\,.
\]
Thus by \eqref{7.14} we get
\begin{equation} \label{7.18}
X_{2,1}\prec \cal E^2 \bb E |P^{2n-2}|+\cal E^3 \bb E |P^{2n-3}| \prec \cal E^2 \cal P^{2n-2}+\cal E^3 \cal P^{2n-3}
\end{equation}
as desired. As for the term $X_{2,2}$, we see from \eqref{diffz} that the most dangerous term is
\begin{multline} \label{7.16}
 \frac{1}{N}\sum_{i,j} \cal C_3(H_{ij}) \bb E \Big[(2n-1)P^{2n-2} \frac{\partial P}{\partial H_{ij}}(-G_{ii}G_{jj})(1+\delta_{ij})^{-1}\Big]\\
=\frac{1}{N^{2+\beta}}\sum_{i,j} a_{ij}\bb E \Big[P^{2n-2}(-2P'N^{-1}(G^2)_{ij}+4P'N^{-1}H_{ij}\ul{G^2}+4N^{-1}H_{ij}\ul{G}^2G_{ii}G_{jj}\Big]\\
\eqd X_{2,2,1}+X_{2,2,2}+X_{2,2,3}\,,
\end{multline}
where $a_{ij}$ is deterministic and uniformly bounded. Note that we write 
$$
X_{2,2,1}=\bb E\cal S(V)\,,
$$
where $\nu_1(V)=2,\theta(V)=2+\beta,\nu_3(V)=2n-2,\nu_4(V)=\nu_5(V)=1$. By Lemma \ref{lem7.4} (i), we have
\begin{equation*}
X_{2,2,1} \prec N^{-\beta}(\Psi+\sqrt{\kappa+\eta}) (N\eta)^{-1}\Upsilon\bb E |P^{2n-2}|+\sum_{t=1}^{2n-2}  N^{-\beta}((\Psi+\sqrt{\kappa+\eta}) \Upsilon)^{t+1}\bb E |P^{2n-2-t}|\,.
\end{equation*}
One can easily check
\begin{equation} \label{7.21}
N^{-\beta}(\Psi+\sqrt{\kappa+\eta}) (N\eta)^{-1}\Upsilon\prec \cal E^2 \quad \mbox{and}\quad N^{-\beta}((\Psi+\sqrt{\kappa+\eta}) \Upsilon)^{t+1} \prec \cal E^{2+t}
\end{equation}
for all $t\geq 1$. Thus we have
\[
X_{2,2,1} \prec \sum_{r=2}^{2n} \cal E^r \bb E |P^{2n-r} |\prec \sum_{r=2}^{2n} \cal E^r \cal P^{2n-r}\,.
\]
For $X_{2,2,2}$, we can again apply Lemma \ref{lem:cumulant_expansion} for $h=H_{ij}$ and get
\[
X_{2,2,2}=\frac{4}{N^2q}\sum_{i,j} a_{ij}\sum_{k=1}^\ell\cal C_{k+1}(H_{ij})\bb E \bigg[\frac{\partial^k P^{2n-2} P'N^{-1}\ul{G^2}}{\partial H_{ij}^k}\bigg]+O_{\prec}(N^{-4n})\,.
\]
Note that
\[
\frac{\partial^s P' N^{-1}\ul{G^2}}{\partial H_{ij}^s} \prec (\Psi+\sqrt{\kappa+\eta})\Upsilon
\]
for all fixed $s\geq 0$. Together with \eqref{2} and the trivial bound $N^{-1} \prec \cal E$, we see that
\[
X_{2,2,2} \prec \frac{1}{Nq} \sum_{t=0}^{2n-2} ((\Psi+\sqrt{\kappa+\eta})\Upsilon)^{t+1}\cal P^{2n-2-t}+\cal E^{4n}\prec\sum_{r=2}^{2n}\cal E^r \cal P^{2n-r}\,,
\]
where in the last step we also used $\Upsilon\prec \cal E$. Similar steps also work for $X_{2,2,3}$. As a result, we have $\eqref{7.16}  \prec  \sum_{r=2}^{2n}\cal E^r \cal P^{2n-r}$. Other terms in $X_{2,2}$ can be estimated in a similar fashion, which leads to
\[
X_{2,2} \prec \sum_{r=2}^{2n}\cal E^r \cal P^{2n-r}\,.
\]
Combining with \eqref{7.18} we get
\[
X_2\prec \sum_{r=2}^{2n}\cal E^r \cal P^{2n-r} 
\]
as desired.

\subsubsection{The computation of $X_3$} \label{sec7.3.3}
Let us split
\[
X_{3}=X_{3,1}+X_{3,2}+X_{3,3}\,,
\]
where
\[
X_{3,s}\deq \frac{1}{N}\frac{1}{3!} {3 \choose s}\sum_{i,j} \cal C_{4}(H_{ij}) \bb E \bigg[  \frac{\partial^s P^{2n-1}}{\partial H_{ij}^s} \frac{\partial^{3-s} G_{ij}}{\partial H_{ij}^{3-s}}\bigg]
\]
for $s=1,2,3$.

\paragraph{Step 1} When $s=1,3$, it is easy to see from \eqref{diffz} that
\[
\frac{\partial^{3-s} G_{ij}}{\partial H_{ij}^{3-s}}\prec |G_{ij}|+ \Upsilon\,.
\]
Using \eqref{2}, we can deduce
\[
\frac{\partial^s P^{2n-1}}{\partial H_{ij}^s} \prec \sum_{t=0}^{2n-2} ((\Psi+\sqrt{\kappa+\eta})\Upsilon)^{t+1} P^{2n-2-t}\,.
\]
Thus 
\begin{multline*}
X_{3,s} \prec \frac{1}{N^{2+2\beta}} \sum_{i,j}\sum_{t=0}^{2n-2} ((\Psi+\sqrt{\kappa+\eta})\Upsilon)^{t+1}   \bb E |P^{2n-2-t}(|G_{ij}|+ \Upsilon)|\\
 \prec \frac{1}{N^{2\beta}} \sum_{t=0}^{2n-2} ((\Psi+\sqrt{\kappa+\eta})\Upsilon)^{t+1} \sqrt{\Upsilon}  \bb E |P^{2n-2-t}| \prec \frac{1}{N^{2\beta}} \sum_{t=0}^{2n-2} ((\Psi+\sqrt{\kappa+\eta})\Upsilon)^{t+1} \sqrt{\Upsilon}  \cal P^{2n-2-t}\,,
\end{multline*}
where in the second step we used Lemma \ref{Ward}. As in \eqref{7.17} and \eqref{7.21}, we have
\[
\frac{1}{N^{2\beta}}((\Psi+\sqrt{\kappa+\eta})\Upsilon)^{t+1} \sqrt{\Upsilon} \prec \cal E^{2+t}
\]
for all $ t \geq 0$. Thus $X_{3,s} \prec \sum_{r=2}^{2n} \cal E^r \cal P^{2n-r}$ for $s=1,3$.

\paragraph{Step 2} Let us consider 
\[
X_{3,2}= \frac{1}{2N}\sum_{i,j} \cal C_{4}(H_{ij}) \bb E \bigg[  \frac{\partial^2 P^{2n-1}}{\partial H_{ij}^2} \frac{\partial G_{ij}}{\partial H_{ij}}\bigg]\,.
\]
Similar as in the previous steps, we can show that
\begin{multline} \label{7.26}
X_{3,2}=\frac{1}{2N}\sum_{i,j} \cal C_{4}(H_{ij}) \bb E \bigg[  \frac{\partial^2 P^{2n-1}}{\partial H_{ij}^2} (-G_{ii}G_{jj}-(1-\delta_{ij})G_{ij}^2)+4N^{-1}(G^2)_{ij}H_{ij}(1+\delta_{ij})^{-1}\bigg]\\
=-\frac{1}{2N}\sum_{i,j} \cal C_{4}(H_{ij}) \bb E \bigg[  \frac{\partial^2 P^{2n-1}}{\partial H_{ij}^2} G_{ii}G_{jj}\bigg]+O_{\prec}(\cal E^2 )\bb E |P^{2n-2}|\,.
\end{multline}
By Lemma \ref{lemP} and \eqref{diffP}, we have
\[
\frac{\partial^2 P^{2n-1}}{\partial H_{ij}^2}=(2n-1)P^{2n-2}\frac{\partial^2 P}{\partial H_{ij}^2}+O_{\prec} ((\Psi+\sqrt{\kappa+\eta})^2\Upsilon^2) P^{2n-3}
\]
and
\begin{multline*}
\frac{\partial^2 P}{\partial H_{ij}^2}=\frac{\partial (-2P'N^{-1}(G^2)_{ij}+4P'N^{-1}H_{ij}\ul{G^2}+4N^{-1}H_{ij}(\ul{G}+\ul{G}^2))}{\partial H_{ij}}(1+\delta_{ij})^{-1}\\
=2P'N^{-1}((G^2)_{ii}G_{jj}+(G^2)_{jj}G_{ii})(1+\delta_{ij})^{-2}+4P'N^{-1}\ul{G^2}(1+\delta_{ij})^{-1}+4N^{-1}(\ul{G}+\ul{G}^2)(1+\delta_{ij})^{-1}\\
+O_{\prec}((\Psi+\sqrt{\kappa+\eta})\Upsilon) (|G_{ij}|+N^{-1}\eta^{-1})\,.
\end{multline*}
Together with Lemma \ref{Ward} we get
\begin{multline*}
X_{3,2}=-\frac{4n-2}{N^2}\sum_{i,j} \cal C_4(H_{ij}) \bb E[ P^{2n-2}P'(G^2)_{ii}G_{ii}G^2_{jj}+P^{2n-2}P'\ul{G^2}G_{ii}G_{jj}+P^{2n-2}(\ul{G}+\ul{G}^2)G_{ii}G_{jj}]\\
+O_{\prec}((\Psi+\sqrt{\kappa+\eta})\Upsilon^{3/2})\bb E |P^{2n-2}|+O_{\prec}(\cal E^2 )\bb E |P^{2n-2}|\,.
\end{multline*}
As the last two terms can be estimated by $O_{\prec}(\cal E^2\cal P^{2n-2})$, we have
\begin{multline} \label{7.29}
X_{3,2}=-\frac{4n-2}{N^2}\sum_{i,j} \cal C_4(H_{ij}) \bb E[ P^{2n-2}P'(G^2)_{ii}G_{ii}G^2_{jj}+P^{2n-2}P'\ul{G^2}G_{ii}G_{jj}+P^{2n-2}(\ul{G}+\ul{G}^2)G_{ii}G_{jj}]\\+O_{\prec}(\cal E^2\cal P^{2n-2})
\eqd X_{3,2,1}+X_{3,2,2}+X_{3,2,3}+O_{\prec}(\cal E^2\cal P^{2n-2})\,.
\end{multline}

\paragraph{Step 3} 
Let us compute $X_{3,2,1}$. We write
\begin{equation} \label{hello}
X_{3,2,1}=\bb E \cal S(V)\,, \quad \mbox{where} \quad V_{ij}=-(4n-2)N^{-1}\cal C_4(H_{ij}) P^{2n-2}P'N^{-1}(G^2)_{ii} G_{ii}G^2_{jj}\,.
\end{equation}
Note that $V \in \cal V_0$ with $\nu_1(V)=2,\theta(V)=2+2\beta,\nu_3(V)=2n-2$ and $\nu_4(V)=\nu_5(V)=1$. By Lemma \ref{lem7.6} we have
\begin{equation} \label{7.27}
X_{3,2,1}-\bb E \cal M_{\infty}(V) \\\prec N^{-2\beta} (\Psi+\sqrt{\kappa+\eta}) (N\eta)^{-1}\Upsilon\bb E |P^{2n-2}|+\sum_{t=1}^{2n-2}  N^{-2\beta}((\Psi+\sqrt{\kappa+\eta}) \Upsilon)^{t+1}\bb E |P^{2n-2-t}|\,,
\end{equation}
where 
\begin{equation*} 
\bb E \cal M_{\infty}(V)= \bb E \cal M(V)+ N^{-2\beta}\sum_{l=2}^{\ceil{\beta^{-1}}} b_{l} N^{-l\beta} \bb EP'N^{-1}\ul{G^2}\,\ul{G}^{2+2l} P^{2n-2}\,,
\end{equation*}
and $b_2,\dots,b_{\ceil{\beta^{-1}}}$ are bounded. We can estimate the right-hand side of \eqref{7.27} by $\sum_{r=2}^{2n} O_{\prec}(\cal E^{r} \cal P^{2n-2})$, so that
\begin{equation} \label{7.28}
X_{3,2,1}= \bb E \cal M(V)+ N^{-2\beta}\sum_{l=2}^{\ceil{\beta^{-1}}} b_{l} N^{-l\beta} \bb EP'N^{-1}\ul{G^2}\,\ul{G}^{2+2l} P^{2n-2}+\sum_{r=2}^{2n} O_{\prec}(\cal E^{r} \cal P^{2n-2})\,.
\end{equation}

\paragraph{Step 4} Let us consider the term $\bb E \cal M(V)$ in \eqref{7.28}. Explicitly,
\[
\bb E \cal M(V)=-(4n-2)N^{2\beta-1}\sum_{i,j}\cal C_4(H_{ij}) \cdot N^{-2\beta} \bb E P'N^{-1} \ul{G^2}\,\ul{G}^2 P^{2n-2} \eqd b_1 N^{-2\beta} \bb E P' N^{-1} \ul{G^2}\,\ul{G}^2 P^{2n-2}\,,
\]
where
\begin{equation}\label{eqnb_1}
b_1=-(4n-2)N^{2\beta-1}\sum_{i,j}\cal C_4(H_{ij})
\end{equation}
is bounded by Lemma \ref{Tlemh}. Since
\begin{equation} \label{7.30}
\partial_{w} G_{ij}=(G^2)_{ij}\,,
\end{equation}
we have
$$
P'N^{-1}\ul{G^2}=N^{-1}(\partial_w P -\ul{G})=N^{-1}\Big(\partial_{w}(\ul{HG})+\partial_{w}(Q)-\ul{G}\Big)\,.
$$ 
In addition,
\[
b_{1} N^{-2\beta-1}
\bb E\partial_w(Q)\ul{G}^{3} P^{2n-2}=b_{1} N^{-2\beta-1}
\bb E\partial_w(Q_0)\ul{G}^{3} P^{2n-2} + O_{\prec}(\cal E^2 \cal P^{2n-2})\,.
\]
Thus,
\begin{multline} \label{wings}
\bb E \cal M(V)=
b_{1} N^{-2\beta-1}
 \bb E\partial_{w}(\ul{HG})\ul{G}^3P^{2n-2}+ b_{1} N^{-2\beta-1}
\bb E\partial_w(Q_0)\ul{G}^{3} P^{2n-2}\\
-N^{-2\beta-1} b_{1}
\bb E \ul{G}^{4} P^{2n-2} + O_{\prec}(\cal E^2 \cal P^{2n-2}) \eqd \mbox{(A)}+\mbox{(B)}+\mbox{(C)}+ O_{\prec}(\cal E^2 \cal P^{2n-2})\,.
\end{multline}

\paragraph{Step 5} We expand the term (A) again by Lemma \ref{lem:cumulant_expansion}, and get
\begin{equation*}
\mbox{(A)}=b_{1} N^{-2-2\beta}\sum_{k=1}^{\ell}\sum_{i,j}\cal C_{k+1}(H_{ij})\bb E \bigg[\frac{\partial^k \partial_{w}(G_{ji}) \ul{G}^3 P^{2n-2}}{\partial H_{ij}^k}\bigg]+O_{\prec}(\cal E^{2n})\,.
\end{equation*}
By Lemma \ref{lemP}, whenever the derivative $\partial^k/\partial H_{ij}^k$ on the right-hand side hits $\ul{G}^3P^{2n-2}$, the corresponding term can be bounded by $O_{\prec}\big(\sum_{r=2}^{2n} \cal E^r \cal P^{2n-r}\big)$. Furthermore, since $\partial^k/\partial H_{ij}^k$ commutes with $\partial_{w}$,
\begin{multline*}
(A)
=\frac{b_1}{N^{2+2\beta}}\sum_{k=1}^{\ell}\sum_{i,j}\cal C_{k+1}(H_{ij})\bb E \bigg[ \partial_{w}\Big(\frac{\partial^kG_{ji}}{\partial H_{ij}^k}\Big) \ul{G}^3 P^{2n-2}\bigg]
+O_{\prec}\Big(\sum_{r=2}^{2n} \cal E^r \cal P^{2n-r}\Big)\\\eqd \sum_{k=1}^{\ell}Y_k +O_{\prec}\Big(\sum_{r=2}^{2n} \cal E^r \cal P^{2n-r}\Big)\,.
\end{multline*}
The analysis of $Y_k$ is similar to that of $\widetilde{X}_k$ in Section \ref{sec4.2}. For $k=1$, by \eqref{diffz}, \eqref{7.30} and Lemma \ref{lem7.4} (i), we have
\begin{multline} \label{7.32}
Y_1=\frac{b_1}{N^{3+2\beta}}\sum_{i,j}\bb E \big[\partial_{w} (-G_{ii}G_{jj}-G_{ij}^2) \ul{G}^3P^{2n-2}\big]+O_{\prec}\Big(\sum_{r=2}^{2n} \cal E^r \cal P^{2n-r}\Big)\\
=-\frac{b_1}{N^{1+2\beta}}\bb E \big[\partial_{w} (\ul{G}^2) \ul{G}^3P^{2n-2}\big]-\frac{b_1}{N^{2+2\beta}}\bb E \big[\ul{G^3}\, \ul{G}^3P^{2n-2}\big]+O_{\prec}\Big(\sum_{r=2}^{2n} \cal E^r \cal P^{2n-r}\Big)\,.
\end{multline}
For $k=2$,  by \eqref{diffz} and \eqref{7.30} we see that the most dangerous term is
\[
\frac{b_1}{N^{1+2\beta}}\sum_{i,j}\cal C_3(H_{ij})\bb E \big[G_{ii}G_{jj}N^{-1}(G^2)_{ij} \ul{G}^3P^{2n-2}\big] \eqd \bb E\cal S (\widetilde{V})\,.
\]
Since $\cal C_3(H_{ij})=O(N^{-1-\beta})$, we see that $\nu_1(\wt V) = 2$, $\nu_2(\widetilde{V})=1$, $\nu_4(\widetilde{V})=0$, $\nu_5(\widetilde{V})=1$, and $\theta(\wt V) = 2 + 3 \beta$. Thus by, Lemma \ref{lem7.4} (i),
\[
\bb E\cal S (\widetilde{V}) \prec N^{-3\beta}\Upsilon (N\eta)^{-1} \bb E |P|^{2n-2}+\sum_{t=1}^{2n-2}  N^{-3\beta} \Upsilon((\Psi+\sqrt{\kappa+\eta}) \Upsilon)^{t}\bb E |P^{2n-2-t}| \prec \sum_{r=2}^{2n} \cal E^r \cal P^{2n-r}\,,
\]
where in the last step we again used $w \in \f Y$ and H\"older's inequality.
Other terms in $Y_2$ also satisfy the same bound. A similar estimate can also be obtained for all even $k$, which yields
\begin{equation} \label{7.33}
\sum_{s=1}^{\floor{\ell/2}} Y_{2s} \prec \sum_{r=2}^{2n} \cal E^r \cal P^{2n-r}\,.
\end{equation} 
For odd $k \geq 3$, we split
\[
Y_k=Y_{k,1}+Y_{k,2}\,,
\] 
where by definition terms in $Y_{k,1}$ contain no off-diagonal entries of $G$ or $G^2$. Use Lemma \ref{lem7.4} (i), we can again show that
\[
Y_{k,2} \prec\sum_{r=2}^{2n} \cal E^r \cal P^{2n-r}\,.
\]
By Lemma \ref{Tlemh}, we see that
\[
Y_{k,1}=-\frac{b_1}{N^{3+(k+1)\beta}}\sum_{i,j}a^{(k)}_{ij}\bb E \big[\partial_{w}(G^{(k+1)/2}_{ii}G^{(k+1)/2}_{jj}) \ul{G}^3P^{2n-2}\big]\,,
\]
where $a^{(k)}_{ij}$ is deterministic and uniformly bounded. Combining with \eqref{7.32}--\eqref{7.33}, we obtain
\begin{multline} \label{7.35}
\mbox{(A)}+\frac{b_1}{N^{2+2\beta}}\bb E \big[\ul{G^3}\, \ul{G}^3P^{2n-2}\big]+\frac{b_1}{N^{1+2\beta}}\bb E \big[\partial_{w} (\ul{G}^2) \ul{G}^3P^{2n-2}\big]\\
=-\sum_{s=2}^{\ceil{\ell/2}}\frac{b_1}{N^{3+2s\beta}}\sum_{i,j} a^{(2s-1)}_{ij}\bb E [\partial_{w}(G^{s}_{ii}G_{jj}^{s})\ul{G}^3P^{2s-2}]+O_{\prec}\Big(\sum_{r=2}^{2n} \cal E^r \cal P^{2n-r}\Big)\\
\eqd -\sum_{s=2}^{\ceil{\ell/2}}\frac{b_1}{N^{1+2\beta}} \bb E[\cal \partial_{w}(\cal S(T^{(s)}))\ul{G}^3P^{2n-2}] +O_{\prec}\Big(\sum_{r=2}^{2n} \cal E^r \cal P^{2n-r}\Big)\,,
\end{multline}
where 
\begin{equation} \label{7.36}
T^{(s)}=\frac{1}{N^{2+(2s-2)\beta}} a_{ij}^{(2s-1)} G_{ii}^sG_{jj}^s\in \cal T_0\,.
\end{equation}
By Lemma \ref{lem:ntt}, we have
\begin{multline*}
\frac{b_1}{N^{1+2\beta}} \bb E[\cal \partial_{w}(\cal S(T^{(s)}))\ul{G}^3P^{2n-2}]=\frac{b_1}{N^{1+2\beta}} \bb E[\cal \partial_{w}(\cal M(\ceil{\beta^{-1}-2s+2},T^{(s)}))\ul{G}^3P^{2n-2}]\\
+O_{\prec}\big(N^{-2s\beta}\Upsilon((N\eta)^{-1}+N^{-1+(2s-2)\beta})\bb E |P|^{2n-2}\big)
+\sum_{t=1}^{2n-2} O_{\prec}\big(N^{-2s\beta}\Upsilon((\Psi+\sqrt{\kappa+\eta})\Upsilon)^t\bb E |P|^{2n-2-t}\big)\,.
\end{multline*}
Since $s \geq 2$, one readily checks that the last two terms can be bounded by $O_{\prec}(\sum_{r=2}^{2n} \cal E^r \cal P^{2n-r})$. Thus \eqref{7.35} reads
\begin{multline} \label{7.37}
\mbox{(A)}=\frac{b_1}{N^{1+2\beta}}\sum_{s=2}^{\ceil{\ell/2}}\bb E[\cal \partial_{w}(\cal M(\ceil{\beta^{-1}-2s+2},T^{(s)}))\ul{G}^3P^{2n-2}]-\frac{b_1}{N^{1+2\beta}}\bb E \big[\partial_{w} (\ul{G}^2) \ul{G}^3P^{2n-2}\big]\\-\frac{b_1}{N^{2+2\beta}}\bb E \big[\ul{G^3}\, \ul{G}^3P^{2n-2}\big]+O_{\prec}\Big(\sum_{r=2}^{2n} \cal E^r \cal P^{2n-r}\Big)\,.
\end{multline}
Note that by construction, $T^{(s)}$ in \eqref{7.36} is the same as in \eqref{TS}. From Lemma \ref{lem:ntt}, we see that the term $\cal M(\ceil{\beta^{-1}-2s+2},T^{(s)})$ in \eqref{7.37} is the same as in \eqref{5.13}, which implies
\[
\ul{G}^2-\sum_{s=2}^{\ceil{\ell/2}}\bb E \cal M\big(\ceil{\beta^{-1}}-2s+2,T^{(s)}\big)=Q_0(\ul{G})\,.
\]
Thus \eqref{7.37} reduces to 
\begin{equation} \label{xiaomaomao}
\mbox{(A)}=-\frac{b_1}{N^{1+2\beta}}\bb E \big[\partial_{w} (Q_0) \ul{G}^3P^{2n-2}\big]-\frac{b_1}{N^{2+2\beta}}\bb E \big[\ul{G^3}\, \ul{G}^3P^{2n-2}\big]+O_{\prec}\Big(\sum_{r=2}^{2n} \cal E^r \cal P^{2n-r}\Big)\,.
\end{equation}

\paragraph{Final Step} By \eqref{wings} and \eqref{xiaomaomao}, we see that there is a cancellation between $\mathrm{(A)}$ and $\mathrm{(B)}$, which leads to
\begin{equation}  \label{imagine}
\bb E \cal M(V)=-\frac{b_1}{N^{2+2\beta}}\bb E \big[\ul{G^3}\, \ul{G}^3P^{2n-2}\big]
-\frac{b_1}{N^{1+2\beta}}
\bb E \ul{G}^{4} P^{2n-2}+O_{\prec}\Big(\sum_{r=2}^{2n} \cal E^r \cal P^{2n-r}\Big)\,.
\end{equation}
The first two terms on right-hand side of \eqref{imagine} are stochastically dominated by
\begin{equation} \label{people}
\frac{\Upsilon}{N\eta}\bb E|P^{2n-2}|+\frac{1}{N^{1+2\beta}}\bb E |P^{2n-2}|\,,
\end{equation}
and one can check that $\Upsilon/(N\eta)\gg \cal E^2$ and $N^{-1-2\beta} \gg \cal E^2$, so that we need to keep track of these terms in order to obtain a further cancellation.

So far we have been dealing with $\bb E \cal M(V)$ in \eqref{7.28}, and other terms in \eqref{7.28} can be handled in the same way as in Steps 4 and 5. Compared to $\bb E\cal M(V)$, each $N^{-2\beta}b_{l} N^{-l\beta} \bb EP'N^{-1}\ul{G^2}\,\ul{G}^{2+2l} P^{2n-2}$ contains an additional factor $N^{-l\beta}$. Similarly to \eqref{imagine} and \eqref{people}, it can be shown that
\begin{multline*}
N^{-2\beta}b_{l} N^{-l\beta} \bb EP'N^{-1}\ul{G^2}\,\ul{G}^{2+2l} P^{2n-2} \\\prec N^{-l\beta}\Big(\frac{\Upsilon}{N\eta}\bb E|P^{2n-2}|+\frac{1}{N^{1+2\beta}}\bb E |P^{2n-2}|\Big)+\sum_{r=2}^{2n} \cal E^r \cal P^{2n-r}
\prec \sum_{r=2}^{2n} \cal E^r \cal P^{2n-r}
\end{multline*}
for all $l \geq 2$. As a result, we have
\begin{equation} \label{7.41}
X_{3,2,1}=-\frac{b_1}{N^{2+2\beta}}\bb E \big[\ul{G^3}\, \ul{G}^3P^{2n-2}\big]
-\frac{b_1}{N^{1+2\beta}}
\bb E [\ul{G}^{4} P^{2n-2}]+O_{\prec}\Big(\sum_{r=2}^{2n} \cal E^r \cal P^{2n-r}\Big)\,,
\end{equation}
where $b_1$ is defined as in \eqref{eqnb_1}.

Next, we consider the other terms on right-hand side of \eqref{7.29}. Similarly to \eqref{7.41}, we can also show that
\[
X_{3,2,2}=-\frac{b_1}{N^{2+2\beta}}\bb E \big[\ul{G^3}\, \ul{G}^2P^{2n-2}\big]
-\frac{b_1}{N^{1+2\beta}}
\bb E [\ul{G}^{3} P^{2n-2}]+O_{\prec}\Big(\sum_{r=2}^{2n} \cal E^r \cal P^{2n-r}\Big)
\]
as well as
\[
X_{3,2,3}=\frac{b_1}{N^{1+2\beta}} \bb E [(\ul{G}^3+\ul{G}^4) P^{2n-2}]+O_{\prec}\Big(\sum_{r=2}^{2n} \cal E^r \cal P^{2n-r}\Big)\,.
\]
Note that this results in two cancellations on right-hand side of \eqref{7.29}, and we have
\[
X_{3,2}=-\frac{b_1}{N^{2+2\beta}}\bb E \big[\ul{G^3}\, \ul{G}^3P^{2n-2}\big]
-\frac{b_1}{N^{2+2\beta}}\bb E \big[\ul{G^3}\, \ul{G}^2P^{2n-2}\big]+O_{\prec}\Big(\sum_{r=2}^{2n} \cal E^r \cal P^{2n-r}\Big)\,.
\] 
As we have already estimated $X_{3,1}$ and $X_{3,3}$ in Step 1, we conclude that
\begin{equation} \label{7.42}
X_{3}=-\frac{b_1}{N^{2+2\beta}}\bb E \big[\ul{G^3}\, \ul{G}^3P^{2n-2}\big]
-\frac{b_1}{N^{2+2\beta}}\bb E \big[\ul{G^3}\, \ul{G}^2P^{2n-2}\big]+O_{\prec}\Big(\sum_{r=2}^{2n} \cal E^r \cal P^{2n-r}\Big)\,.
\end{equation}

\begin{remark}
	The crucial step in analysing $X_3$ is the computation of $X_{3,2,1}$ in \eqref{7.41}. As in \eqref{hello}, we can write $X_{3,2,1}=\bb E\cal S(V)$, with $V\in \cal V_0$, $\nu_1(V)-\theta(V)=-2\beta$, $\nu_3(V)=2n-2$, and $\nu_4(V)=\nu_5(V)=1$. Since
\[
\Big|\frac{1}{N^2}\ul{G^3}\Big|\leq \frac{\im \ul{G}}{N^2\eta^2}\leq \frac{\Upsilon}{N\eta}\,,
\]
the formula \eqref{7.41} implies the estimate 
\[
X_{3,2,1} \prec \Big(\frac{1}{N^{1+2\beta}}+\frac{\Upsilon}{N^{1+2\beta}\eta}\Big)\cal P^{2n-2}+  \sum_{r=2}^{2n}\cal E^{r} \cal P^{2n-r}\,.
\]
The argument for $X_{3,2,1}$ can be repeated for general $\bb E \cal S(V)$, which allows one to show the following result.
\end{remark}

\begin{lemma} \label{lem7.8}
	Let $V\in \cal V_0$, with $\nu_1(V)-\theta(V)\leq -2\beta$, $\nu_3(V)=2n-2$, and $\nu_4(V)=\nu_5(V)=1$. Then
	\[
	\bb E \cal S(V) \prec \Big(\frac{1}{N^{1+\theta-\nu_1}}+\frac{\Upsilon}{N^{1+\theta-\nu_1}\eta}\Big)\cal P^{2n-2}+  \sum_{r=2}^{2n}\cal E^{r} \cal P^{2n-r}
	\]
\end{lemma}

\subsubsection{Conclusion} \label{sec7.3.4} After the steps in Sections \ref{sec7.3.1} -- \ref{sec7.3.3}, it remains to estimate $X_k$ for $k\geq 4$.

When $k\geq 4$ is even, the estimate of $X_k$ is similar to that of $X_2$ in Section \ref{sec7.2.3}. In fact, by Lemma \ref{Tlemh}, we see that there will be additional factors of $N^{-\beta}$ in $X_k$ when $k \geq 4$, which makes the estimate easier. Using Lemma \ref{lem7.4} (i), one can show that
\[
\sum_{s=2}^{\ceil{\ell/2}}X_{2s} \prec \sum_{r=2}^{2n}\cal E^{r} \cal P^{2n-r}\,.
\] 

When $k\geq 4$ is odd, the estimate of $X_k$ is similar to that of $X_3$ in Section \ref{sec7.3.3}. By Lemma \ref{Tlemh}, we see that there will be additional factors of $N^{-(k-2)\beta}$ in $X_k, k\geq 4$. Using Lemmas \ref{lemP}, \ref{lem7.4} and \ref{lem7.8}, one can show that
\[
\sum_{s=2}^{\ceil{\ell/2}}X_{2s+1} \prec 
\Big(\frac{1}{N^{1+4\beta}}+\frac{\Upsilon}{N^{1+4\beta}\eta}\Big)\cal P_n^{2n-2}+  \sum_{r=2}^{2n}\cal E^{r} \cal P^{2n-r} \prec \sum_{r=2}^{2n}\cal E^{r} \cal P^{2n-r}\,.
\]
As a result, we arrive at
\begin{equation} \label{dingding}
\mbox{(IV')}=\sum_{k=1}^{\ell}X_k=-\frac{b_1}{N^{2+2\beta}}\bb E \big[\ul{G^3}\, \ul{G}^3P^{2n-2}\big]
-\frac{b_1}{N^{2+2\beta}}\bb E \big[\ul{G^3}\, \ul{G}^2P^{2n-2}\big]+O_{\prec}\Big(\sum_{r=2}^{2n} \cal E^r \cal P^{2n-r}\Big)\,,
\end{equation}
where $b_1=-(4n-2)N^{2\beta-1}\sum_{i,j}\cal C_4(H_{ij})$ is bounded.

\subsection{The computation of (II') in \eqref{pengpeng}} 
Using Lemma \ref{lem:cumulant_expansion} with $h=H_{ij}$, we have
\begin{multline} \label{7.43}
\mbox{(II')}=\bb E \cal Z \ul{G}^2P^{2n-1}=\frac{1}{N}\sum_{i,j} \bb E \Big(H_{ij}^2-\frac{1}{N}\Big) \ul{G}^2 P^{2n-1}\\
=\frac{1}{N}\sum_{k=1}^{\ell}  \frac{1}{k!}  \sum_{i,j}\cal C_{k+1}(H_{ij})\bb E \bigg[\frac{\partial^k H_{ij}\ul{G}^2P^{2n-1}}{\partial H_{ij}^k}\bigg]+O_{\prec}(\cal E^{2n})-\bb E\ul{G}^2 P^{2n-1}\,.
\end{multline}
For each $k$, we write
\begin{multline*}
\frac{1}{N} \frac{1}{k!}  \sum_{i,j}\cal C_{k+1}(H_{ij})\bb E \bigg[\frac{\partial^k H_{ij}\ul{G}^2P^{2n-1}}{\partial H_{ij}^k}\bigg]\\
=\frac{1}{N} \frac{1}{k!}  \sum_{i,j}\cal C_{k+1}(H_{ij})\bb E \bigg[H_{ij}\frac{\partial^k \ul{G}^2P^{2n-1}}{\partial H_{ij}^k}\bigg]+\frac{1}{N}  \frac{1}{(k-1)!}  \sum_{i,j}\cal C_{k+1}(H_{ij})\bb E \bigg[\frac{\partial^{k-1} \ul{G}^2P^{2n-1}}{\partial H_{ij}^{k-1}}\bigg]\eqd Z_k+\widehat{X}_k\,.
\end{multline*}
Each $Z_k$ can be handled again by applying Lemma \ref{lem:cumulant_expansion} with $h=H_{ij}$. One easily shows that
\begin{equation} \label{7.44}
\sum_{k=1}^{\ell}Z_k \prec \sum_{r=1}^{2n} \cal E^r \cal P^{2n-r}\,.
\end{equation}
By $\cal C_2(H_{ij})=N^{-1}(1+O(\delta_{ij}))$, we have
\[
\widehat{X}_{1}=\frac{1}{N}\sum_{i,j} \frac{1}{N} \bb E \ul{G}^2P^{2n-1}+\frac{1}{N}\sum_{i} \Big( \cal C_2(H_{ii})-\frac{1}{N}\Big)\bb \ul{G}^2 E P^{2n-1}
=\bb E \ul{G}^2P^{2n-1}+O_{\prec}(N^{-1}\cal P^{2n-1})\,.
\]
Combining with \eqref{7.43} and \eqref{7.44}, we have
\begin{equation} \label{7.45}
\mathrm{(II')}= \sum_{k=2}^{\ell}\widehat{X}_{k}+O_{\prec}\bigg(\sum_{r=1}^{2n} \cal E^r \cal P^{2n-r}\bigg)\,.
\end{equation}
The analysis of $\widehat{X}_{k}$ is similar to those of $X_k$ in Section \ref{sec7.2}, and we only sketch the key steps.

For $k=2$, we see from \eqref{diffz} that the most dangerous term in $\widehat{X}_{2}$ is
\begin{equation} \label{maomao}
\frac{1}{N}\sum_{i,j}\cal C_{3}(H_{ij})\bb E \bigg[\ul{G}^2 (2n-1)P^{2n-2}\frac{\partial P}{\partial H_{ij}}\bigg]\,,
\end{equation}
which is very close to the left-hand side of \eqref{7.16}. We can apply Lemma \ref{lem7.4} (i) and show that \eqref{maomao} is bounded by $O_{\prec}(\sum_{r=2}^{2n} \cal E^r \cal P^{2n-r})$. Similarly, we can also handle all the other terms in $\widehat{X}_{2}$, which leads to
\begin{equation}
\widehat{X}_{2} \prec \sum_{r=2}^{2n} \cal E^r \cal P^{2n-r}\,.
\end{equation}

For $k=3$, by the differential rule \eqref{diffz}, we see that the most dangerous term in $\widehat{X}_{3}$ is
\[
\widehat{X}_{3,2}\deq\frac{1}{2N}\sum_{i,j}\cal C_{4}(H_{ij})\bb E \bigg[\ul{G}^2 \frac{\partial^2 P^{2n-1}}{\partial H_{ij}^2}\bigg]\,,
\]
which is very close to the right-hand side of \eqref{7.26}. Similarly to \eqref{7.29}, we have
\begin{multline*}
\widehat{X}_{3,2}=\frac{4n-2}{N^2}\sum_{i,j} \cal C_4(H_{ij}) \bb E[ P^{2n-2}P'(G^2)_{ii}G_{jj}\ul{G}^2+P^{2n-2}P'\ul{G^2}\,\ul{G}^2+P^{2n-2}(\ul{G}+\ul{G}^2)\ul{G}^2]\\+O_{\prec}(\cal E^2\cal P^{2n-2})
\eqd \widehat{X}_{3,2,1}+\widehat{X}_{3,2,2}+\widehat{X}_{3,2,3}+O_{\prec}(\cal E^2\cal P^{2n-2})\,,
\end{multline*}
and the right-hand side can be computed similarly to $X_{3,2,1}, X_{3,2,2}, X_{3,2,3}$ in \eqref{7.29}. As a result, we can show that
\begin{equation} \label{7.46}
\widehat{X}_3=\frac{b_1}{N^{2+2\beta}}\bb E \big[\ul{G^3}\, \ul{G}^3P^{2n-2}\big]
+\frac{b_1}{N^{2+2\beta}}\bb E \big[\ul{G^3}\, \ul{G}^2P^{2n-2}\big]+O_{\prec}\Big(\sum_{r=2}^{2n} \cal E^r \cal P^{2n-r}\Big)\,,
\end{equation}
where $b_1=-(4n-2)N^{2\beta-1}\sum_{i,j}\cal C_4(H_{ij})$.

For $k \geq 4$, the argument is similar to that in Section \ref{sec7.3.4}. We can show that
\[
\sum_{k=4}^{\ell}\widehat{X}_k \prec \sum_{r=2}^{2n} \cal E^r \cal P^{2n-r}\,.
\]
Combining the above with \eqref{7.45}--\eqref{7.46}, we have
\begin{equation} \label{7.48}
\mbox{(II')}=\frac{b_1}{N^{2+2\beta}}\bb E \big[\ul{G^3}\, \ul{G}^3P^{2n-2}\big]
+\frac{b_1}{N^{2+2\beta}}\bb E \big[\ul{G^3}\, \ul{G}^2P^{2n-2}\big]+O_{\prec}\Big(\sum_{r=2}^{2n} \cal E^r \cal P^{2n-r}\Big)\,.
\end{equation}

Now observe the cancellation between \eqref{dingding} and \eqref{7.48}, which leads to 
\[
\mbox{(II')}+\mbox{(IV')} \prec\sum_{r=1}^{2n}\cal E^{r} \cal P^{2n-r}
\]
as desired.

\subsection{The estimate of \eqref{dingman}} \label{sec7.4}
From the construction of $P_0$ in Section \ref{sec4.2}, we can easily show that
\[
\bb E Q_0 +\frac{1}{N}\sum_{k=1}^\ell\frac{1}{k!}\sum_{i,j} \cal C_{k+1}(H_{ij}) \bb E \bigg[  \frac{\partial^k G_{ij}}{\partial H_{ij}^k} \bigg] \prec \Upsilon\,,
\]
and in this section we shall see that the analogue holds when the factor $P^{2n-1}$ is added inside the expectations. Let us write
\[
\sum_{k=1}^{\ell} X^{(1)}_k\deq \sum_{k=1}^\ell\frac{1}{N}\frac{1}{k!}\sum_{i,j} \cal C_{k+1}(H_{ij}) \bb E \bigg[  \frac{\partial^k G_{ij}}{\partial H_{ij}^k} P^{2n-1}\bigg]=\mathrm{(III')}
\]
and analyse each $X_k^{(1)}$. 

Let us first consider the case when $k$  is odd. For $k=1$, it is easy to see from \eqref{diffz} and Lemma \ref{Ward} that
\begin{equation} \label{varus}
X_1^{(1)}=\frac{1}{N^2}\sum_{i,j}\bb E[-G_{ii}G_{jj}P^{2n-1}]+O_{\prec}\big(\Upsilon \bb E |P^{2n-1}|\big)=-\bb E \ul{G}^2 P^{2n-1}+O_{\prec} \big(\cal E \cal P^{2n-1}\big)\,.
\end{equation}
For odd $k \geq 3$, we see from \eqref{diffz} and Lemma \ref{Ward} that
\begin{equation} \label{varus1}
X_k^{(1)}=\frac{1}{N^{2+(k-1)\beta}} \sum_{i,j} a_{ij}^{(k)}\bb E \big[G_{ii}^{(k+1)/2}G_{jj}^{(k+1)/2}P^{2n-1}\big]+O_{\prec} \big(\cal E \cal P^{2n-1}\big)\,,
\end{equation}
where $a^{(k)}_{ij}$ is deterministic and uniformly bounded.

For even $k$, we follow a similar strategy as in Section \ref{sec7.2.3}. We see from \eqref{diffz} and Lemma \ref{Ward} that
\begin{equation} \label{7.47}
X_k^{(1)}=\frac{1}{N^{2+(k-1)\beta}} \sum_{i,j} a_{ij}^{(k)}\bb E \big[G_{ij}G_{ii}^{k/2}G_{jj}^{k/2}P^{2n-1}\big]+O_{\prec} \big(\cal E \cal P^{2n-1}\big)\,,
\end{equation}
where $a^{(k)}_{ij}$ is deterministic and uniformly bounded. The first term on right-hand side of \eqref{7.47} can be written as $\bb E \cal S(V)$, where $V\in \cal V$, $\nu_2(V)\ne 0$ and $\nu_4(V)=\nu_5(V)=0$. Thus we can apply Lemma \ref{lem7.4} (ii) to estimate this term, and show that it is bounded by $ O_{\prec}(\sum_{r=1}^{2n} \cal E^r \cal P^{2n-r})$. This implies
\begin{equation} \label{varus3}
\sum_{s=2}^{\ceil{\ell/2}}X^{(1)}_{2s} \prec \sum_{r=1}^{2n}\cal E^{r} \cal P^{2n-r}\,.
\end{equation}

Combining \eqref{varus}, \eqref{varus1} and \eqref{varus3}, we have
\begin{multline} \label{7.444}
\mbox{(III')}=-\bb E[ \ul{G}^2 P^{2n-1}]+\sum_{s=2}^{\ceil{\ell/2}}\frac{1}{N^{2+(2s-2)\beta}}\sum_{i,j}a_{ij}^{(2s-1)}\bb E [G_{ii}^sG_{jj}^sP^{2n-1}]+O_{\prec}\Big(\sum_{r=1}^{2n}\cal E^{r} \cal P^{2n-r}\Big)\\\eqd -\bb E [\ul{G}^2 P^{2n-1}]+\sum_{s=2}^{\ceil{\ell/2}}\bb E [\cal S(T^{(s)})P^{2n-1}]+O_{\prec}\Big(\sum_{r=1}^{2n}\cal E^{r} \cal P^{2n-r}\Big)\,,
\end{multline}
 where we recall the definition of $\cal S(T)$ in \eqref{5.11}. Observe that from the above steps, $T^{(s)}$ in \eqref{7.444} is the same as in \eqref{TS}. To handle $\bb E[ \cal S(T^{(s)})P^{2n-1}]$, we introduce the following analogue of Lemmas \ref{lem:nte} and \ref{lem:ntt}.

\begin{lemma} \label{lem7.10}
	Let $T \in \cal T_0$ with $\nu_1(T)-\theta(T) \leq -2\beta$. Fix $r \in \bb N$ and let $\cal M(r,T)$ be as in Lemma \ref{lem:nte}. Then
	\[
	\bb E [\cal S(T)P^{2n-1}]=\bb E [\cal M(r, T)P^{2n-1}] +O_{\prec}\big(N^{\nu_1(T)-\theta(T)}(\Upsilon+N^{-\beta(r+1)})\cal P^{2n-1}\big)+O_{\prec}\Big(\sum_{r=1}^{2n}\cal E^{r} \cal P^{2n-r}\Big)\,.
	\]
\end{lemma}
\begin{proof}
	The proof analogous to those of Lemmas \ref{lem:nte} and \ref{lem:ntt}. We use the identity
	\[
	G_{ii}=\ul{G}+G_{ii}\ul{HG}-(HG)_{ii}\ul{G}
	\]
	to replace the diagonal entries in $\cal S(T)$, and then expand the terms containing $H$ using Lemma \ref{lem:cumulant_expansion}. We omit the details. 
\end{proof}

By Lemma \ref{lem7.10} we have, for any $s \in \{2,3,\dots,\ceil{\ell/2}\}$,
\[
\bb E [\cal S(T^{(s)})P^{2n-1}]=\bb E [\cal M(\ceil{\beta^{-1}-2s+2},T^{(s)})P^{2n-1}]+O_{\prec}\Big(\sum_{r=1}^{2n}\cal E^{r} \cal P^{2n-r}\Big)\,.
\]
Together with \eqref{7.444} we have
\begin{equation} \label{zed}
\mbox{(III')}=-\bb E[ \ul{G}^2 P^{2n-1}]+\sum_{s=2}^{\ceil{\ell/2}}\bb E [\cal M(\ceil{\beta^{-1}-2s+2},T^{(s)})P^{2n-1}]+O_{\prec}\Big(\sum_{r=1}^{2n}\cal E^{r} \cal P^{2n-r}\Big)\,.
\end{equation}
From Lemma \ref{lem7.10}, we see that the term $\cal M(\ceil{\beta^{-1}-2s+2},T^{(s)})$ in \eqref{zed} is the same as in \eqref{5.13}, which implies
\[
\ul{G}^2-\sum_{s=2}^{\ceil{\ell/2}}\bb E \cal M\big(\ceil{\beta^{-1}}-2s+2,T^{(s)}\big)=Q_0(\ul{G})\,.
\]
Thus
\[
\mbox{(III')}=-\bb E [Q_0P^{2n-1}]+O_{\prec}\Big(\sum_{r=1}^{2n}\cal E^{r} \cal P^{2n-r}\Big)\,,
\]
and together with \eqref{dingman} we conclude that 
\[
\mbox{(I')}+\mbox{(III')} \prec \sum_{r=1}^{2n}\cal E^{r} \cal P^{2n-r}
\]
as desired. This concludes the proof of Lemma \ref{lem7.1} and hence also that of Proposition \ref{prop1}.

\section{Proof of Proposition \ref{prop3.3}} \label{sec5}

\begin{convention*}
Throughout this section, $z$ is given by \eqref{def_z_random}, where $w \in \f Y_*$ is deterministic.
\end{convention*}

Let us fix $n \in \bb N_+$ and set
\[
\cal P_{\im} \deq \|\im P(z,\ul{G})\|_{2n}=\Big(\bb E |\im P(z,\ul{G})|^{2n}\Big)^{\frac{1}{2n}}\,, \quad \cal E_{\im} \deq \Big(\frac{1}{N\eta} +\Phi\Big) \sqrt{\kappa}N^{-\delta}\,.
\]
We shall show that
\begin{equation} \label{goal sec8}
\bb E |\im P(z,\ul{G})|^{2n}=\cal P_{\im}^{2n} \prec \cal E_{\im}^{2n}\,,
\end{equation}
from which Proposition \ref{prop3.3} follows by Chebyshev's inequality. We shall see that the proof of \eqref{goal sec8} is much simpler than that of \eqref{goal sec7}, as it does not require a secondary expansion as in Section \ref{sec7.3.3}. We define the parameter
\[
\Theta \deq \frac{\Phi+\frac{\eta}{\sqrt{\kappa}}}{N\eta}\,.
\]
Recall the definition of $\Gamma$ from \eqref{WWard}. It is easy to check that
\[
\Gamma \prec \Theta\,.
\]
In addition, recall the definitions of $P'$, $Q$ and $Q_0$ from \eqref{eqn PQQ_0}. With the help of \eqref{3.1}, we obtain the following improved version of Lemma \ref{lemP}.

\begin{lemma} \label{lemimP}
	We have
	\begin{equation} \label{1im}
	P' \prec \sqrt{\kappa}\,, \quad	 \quad \bigg|\frac{\partial^k \ul{G}}{\partial H_{ij}^k}\bigg| \prec  \max_{x,y} N^{-1}|(G^2)_{xy} |\prec \Theta
	\end{equation}
	and
	\begin{equation} \label{2im}
	\frac{\partial^k P}{\partial H_{ij}^k} \prec \sqrt{\kappa}\,\Theta\,.
	\end{equation}
\end{lemma} 

By $zG=HG-I$, we have
\begin{multline} \label{ning}
\bb E (\im P(z,\ul{G}))^{2n} =\bb E (\im \ul{HG}+\im Q)(\im P)^{2n-1}
\\
=\bb E [\im Q(\im P)^{2n-1}]+\frac{1}{N}\sum_{i,j} \bb E [H_{ij} \im G_{ji} (\im P)^{2n-1}]\,,
\end{multline}
where in the second step we used that $H$ has real entries.
 
\begin{remark} \label{rem:complex_proof}
Although we used that the entries of $H$ are real in \eqref{ning}, our argument easily extends to complex entries of $H$. To see how, for any holomorphic $f:\bb C_+\to \bb C$ we define $J f(z) \deq \frac{1}{2 \ii} (f(z) - f(\ol z))$. We view all quantities appearing in our arguments as functions of $z$ and use the operator $J$ instead of $\im$.  Then it is easy to check that in both real and complex cases, Proposition \ref{prop3.3} as well as all its consequences remain true if we replace $\im$ by $J$ everywhere. Note that $\im \ul G = J \ul G$ and $\im P = J P$, but in general $\im G_{ij} \neq J G_{ij}$. An alternative point of view is to regard all of our quantities as functions of $z$ and $H$, and to take the imaginary part with respect to the Hermitian conjugation of $z$ and $H$. \end{remark}

Similarly to \eqref{7.1}, we can use Lemma \ref{lem:cumulant_expansion} on the last term of \eqref{ning}, and get
\begin{multline} \label{8.1}
\bb E (\im P(z,\ul{G}))^{2n} =\bb E[ \im Q  (\im P)^{2n-1}]+\frac{1}{N}\sum_{k=1}^\ell\frac{1}{k!}\sum_{i,j} \cal C_{k+1}(H_{ij}) \bb E \bigg[  \frac{\partial^k \im G_{ij}}{\partial H_{ij}^k} (\im P)^{2n-1}\bigg]\\
+\frac{1}{N}\sum_{k=1}^\ell\frac{1}{k!}\sum_{s=1}^k {k \choose s}\sum_{i,j} \cal C_{k+1}(H_{ij}) \bb E \bigg[  \frac{\partial^s( \im P)^{2n-1}}{\partial H_{ij}^s} \frac{\partial^{k-s} \im G_{ij}}{\partial H_{ij}^{k-s}}\bigg]+O_{\prec}(N^{-4n})\\
\eqd\mbox{(V)}+\mbox{(VI)}+\mbox{(VII)}+O_{\prec}(N^{-4n})\,.
\end{multline}
We shall prove the following result, which directly implies \eqref{goal sec8}.

\begin{lemma} \label{lem 8.2}
	Let $\mathrm{(V)}-\mathrm{(VII)}$ be as in \eqref{8.1}. Then
	\begin{equation} \label{8.3}
	\mathrm{(V)}+\mathrm{(VI)} \prec \sum_{r=1}^{2n}\cal E_{\im}^{r} \cal P_{\im}^{2n-r} 
	\end{equation}
	and
	\begin{equation} \label{8.4}
	\mathrm{(VII)} \prec \sum_{r=1}^{2n} \cal E_{\im}^{r} \cal P_{\im}^{2n-r}\,.
	\end{equation}
\end{lemma}

\subsection{Proof of \eqref{8.4}}
Define
\begin{equation} \label{faker}
X_k^{(2)}\deq \frac{1}{N}\frac{1}{k!}\sum_{s=1}^k {k \choose s}\sum_{i,j} \cal C_{k+1}(H_{ij}) \bb E \bigg[  \frac{\partial^s( \im P)^{2n-1}}{\partial H_{ij}^s} \frac{\partial^{k-s} \im G_{ij}}{\partial H_{ij}^{k-s}}\bigg]\,,
\end{equation} 
so that $\mathrm{(VII)}=\sum_{k=1}^{\ell}X_{k}^{(2)}$. Note that for $f\col \bb R \to \bb C$ and $h$ real, $\frac{\dd \im f(h)}{\dd h}=\im \frac{\dd f(h)}{\dd h}$, so that the derivatives in \eqref{faker} can be computed through \eqref{diffz}. Let us estimate each $X_k^{(2)}$.

For any fixed $k \in \bb N_+$, it is easy to see from \eqref{2im} that
\[
X_{k}^{(2)}\prec \frac{1}{N^2}\sum_{s=1}^k \sum_{r=1}^{2n-1}
\sum_{i,j}\bb E \bigg|(\sqrt{\kappa}\Theta)^r(\im P)^{2n-1-r} \im \frac{\partial^{k-s} G_{ij}}{\partial H_{ij}^{k-s}}\bigg|\,.
\]
By \eqref{diffz} and Proposition \ref{refthm1}, we see that
\begin{multline} \label{marin}
\im \frac{\partial^{k-s} G_{ij}}{\partial H_{ij}^{k-s}} \prec \big|\im\big( G_{ii}^{\floor{(k-s+1)/2}}G_{jj}^{\floor{(k-s+1)/2}}\big)\big|+|G_{ij}|+\max_{x,y \in \{i,j\}}\big|N^{-1}(G^2)_{xy}\big|\\
\prec \max_{x \in \{i,j\}}\im G_{xx}+|G_{ij}|+\max_{x,y \in \{i,j\}}\big|N^{-1}(G^2)_{xy}\big| \prec \im \ul{G}+|G_{ij}|+\Theta \,,
\end{multline}
where in the last step we estimated $\im G_{xx}$ by $O_\prec(\im \ul G)$, using its spectral decomposition and Lemma \ref{lem:delocalization}. Here we see the crucial effect of taking imaginary part of $P$, which results $\im \ul{G}$ on right-hand side of \eqref{marin} instead of $\ul{G}$. Note that $\im \ul{G} \leq \Phi+\im m\asymp\Phi+\eta/\sqrt{\kappa}$, and together with Lemma \ref{Ward} we have
\[
\frac{1}{N^2} \sum_{i,j}\bigg|\im \frac{\partial^{k-s} G_{ij}}{\partial H_{ij}^{k-s}} \bigg|\prec\Phi+\frac{\eta}{\sqrt{\kappa}}+\frac{1}{N^2}\sum_{i,j}|G_{ij}|+\Theta  \prec\Phi+\frac{\eta}{\sqrt{\kappa}}+\Theta^{1/2}+\Theta \prec \Phi+ \Theta^{1/2}\,.
\]
Thus
\begin{equation} \label{last modifier}
X_k^{(2)} \prec \sum_{r=1}^{2n-1} (\sqrt{\kappa}\Theta)^{r} (\Phi+\Theta^{1/2}) \cal P_{\im}^{2n-1-r}\,.
\end{equation}
By Cauchy-Schwarz and \eqref{leijun} we have $\Theta^{1/2}\prec \Psi+\frac{1}{N\eta}+\frac{\eta}{\sqrt{\kappa}}\prec \Psi+\frac{1}{N\eta}$, thus
\[
(\sqrt{\kappa}\Theta) (\Phi+\Theta^{1/2}) \prec (\sqrt{\kappa}\Theta) \Big(\Psi+\frac{1}{N\eta}\Big)=N^{\delta}\Theta \cdot \cal E_{\im} \prec \cal E_{\im}^2\,.
\]
Together with
$
(\sqrt{\kappa}\Theta)^{s} \prec \cal E_{\im}^{s}
$
for all $s\geq 0$, we get
\[
(\sqrt{\kappa}\Theta)^{r} (\Phi+\Theta^{1/2}) \prec \cal E_{\im}^{r+1}
\]
for all $r\geq 1$. Combining the above estimate with \eqref{last modifier}, we have
\[
X_k^{(2)} \prec \sum_{r=1}^{2n-1} \cal E_{\im}^{r+1} \cal P_{\im}^{2n-1-r}\,.
\]
This concludes the proof of \eqref{8.4}.

\subsection{Proof of \eqref{8.3}} The proof is similar to the estimate of \eqref{dingman} in Section \ref{sec7.4}. Define
\[
 X^{(3)}_k\deq \frac{1}{N}\sum_{k=1}^\ell\frac{1}{k!}\sum_{i,j} \cal C_{k+1}(H_{ij}) \bb E \bigg[  \frac{\partial^k \im G_{ij}}{\partial H_{ij}^k} (\im P)^{2n-1}\bigg]\,,
\]
so that $\mathrm{(VI)}=\sum_{k=1}^{\ell}X_k^{(3)}$. We analyse each $X_k^{(3)}$. 

Consider first the case when $k$  is odd. For $k=1$, it is easy to see from \eqref{diffz} and Lemma \ref{Ward} that
\begin{multline} \label{huawei}
X_1^{(3)}=\frac{1}{N^2}\sum_{i,j}\bb E[-\im (G_{ii}G_{jj}) (\im P)^{2n-1}]+O_{\prec}\big(\Theta \bb E [|\im P|^{2n-1}]\big)
\\
=-\bb E [\im (\ul{G}^2) (\im P)^{2n-1}]+O_{\prec} \big(\cal E_{\im} \cal P_{\im}^{2n-1}\big)\,.
\end{multline}
When $k \geq 3$ is odd, we see from \eqref{diffz} and Lemma \ref{Ward} that
\begin{equation} \label{ruok}
X_k^{(3)}=\frac{1}{N^{2+(k-1)\beta}} \sum_{i,j} a_{ij}^{(k)}\bb E \big[\im \big(G_{ii}^{(k+1)/2}G_{jj}^{(k+1)/2}\big)(\im P)^{2n-1}\big]+O_{\prec} \big(\cal E_{\im} \cal P_{\im}^{2n-1}\big)\,,
\end{equation}
where $a^{(k)}_{ij}$ is deterministic and uniformly bounded.

For even $k$, we see from \eqref{diffz} and Lemma \ref{Ward} that
\begin{equation} \label{eight}
X_{k}^{(3)}=\frac{1}{N^{2+(k-1)\beta}} \sum_{i,j} a_{ij}^{(k)}\bb E \big[\im \big(G_{ij}G_{ii}^{k/2}G_{jj}^{k/2}\big)(\im P)^{2n-1}\big]+O_{\prec} \big(\cal E_{\im} \cal P_{\im}^{2n-1}\big)\,,
\end{equation}
where $a^{(k)}_{ij}$ is deterministic and uniformly bounded. Note that the analogue of \eqref{eight} has appeared in \eqref{7.47}. To handle this term, we use the following result.

\begin{lemma} \label{lem 8.3}
Fix an even $k \geq 2$. Let $\big(a^{(k)}_{ij}\big)_{i,j=1}^N$ be deterministic and uniformly bounded. Then
\[
\frac{1}{N^{2}} \sum_{i,j} a_{ij}^{(k)}\bb E \big[\im \big(G_{ij}G_{ii}^{k/2}G_{jj}^{k/2}\big)(\im P)^{2n-1}\big] \prec \sum_{r=1}^{2n}\cal E_{\im}^r \cal P_{\im}^{2n-r}\,.
\]	
\end{lemma}
\begin{proof}
	The proof essentially follows from the strategy of showing Lemmas \ref{lem4.2} and \ref{lem7.4}. We use the identity
	\[
	G_{ij}=\delta_{ij}\ul{G}+G_{ij}\ul{HG}-(HG)_{ij}\ul{G}
	\]
	to replace the $G_{ij}$ in the equation, and then expand the terms containing $H$ using Lemma \ref{lem:cumulant_expansion}. We omit the details.
\end{proof}

Lemma \ref{lem 8.3} immediately implies
\begin{equation} \label{8.8}
\sum_{s=2}^{\ceil{\ell/2}}X^{(3)}_{2s} \prec \sum_{r=1}^{2n}\cal E_{\im}^{r} \cal P_{\im}^{2n-r}\,.
\end{equation}
Combining \eqref{huawei} -- \eqref{8.8}, we have
\begin{multline*}
\mbox{(VI)}=-\bb E[ \im (\ul{G}^2) (\im P)^{2n-1}]+\sum_{s=2}^{\ceil{\ell/2}}\frac{1}{N^{2+(2s-2)\beta}}\sum_{i,j}a_{ij}^{(2s-1)}\bb E [\im (G_{ii}^sG_{jj}^s)(\im P)^{2n-1}]\\
+O_{\prec}\Big(\sum_{r=1}^{2n}\cal E_{\im}^{r} \cal P_{\im}^{2n-r}\Big)\eqd -\bb E [\im (\ul{G}^2) (\im P)^{2n-1}]+\sum_{s=2}^{\ceil{\ell/2}}\bb E [\im \cal S(T^{(s)}) (\im P)^{2n-1}]+O_{\prec}\Big(\sum_{r=1}^{2n}\cal E_{\im}^{r} \cal P_{\im}^{2n-r}\Big)\,.
\end{multline*}
Here we recall the definition of $\cal S(T)$ in \eqref{5.11}, and observe that $T^{(s)}$ above is the same as in \eqref{TS}. To handle the last relation, we introduce the following analogue of Lemmas \ref{lem:nte} and \ref{lem:ntt}.

\begin{lemma} \label{lem8.4}
	Let $T \in \cal T_0$ with $\nu_1(T)-\theta(T) \leq 0$. Fix $r \in \bb N$, and let $\cal M(r,T)$ be as in Lemma \ref{lem:nte}. Then
	\[
	\bb E [\im \cal S(T)(\im P)^{2n-1}]=\bb E [\im \cal M(r, T)(\im P)^{2n-1}] +O_{\prec}\big((\Theta+N^{-\beta(r+1)})\cal P_{\im}^{2n-1}\big)+O_{\prec}\Big(\sum_{r=1}^{2n}\cal E_{\im}^{r} \cal P_{\im}^{2n-r}\Big)\,.
	\]
\end{lemma}
\begin{proof}
	The proof is similar to those of Lemmas \ref{lem:nte} and \ref{lem:ntt}. We use the identity
	\[
	G_{ii}=\ul{G}+G_{ii}\ul{HG}-(HG)_{ii}\ul{G}
	\]
	to replace the diagonal entries in $\cal S(T)$, and then expand the terms containing $H$ using Lemma \ref{lem:cumulant_expansion}. We omit the details. 
\end{proof}

By Lemma \ref{lem8.4}, we have, for any $s \in \{2,3,\dots,\ceil{\ell/2}\}$,
\[
\bb E [\im \cal S(T^{(s)}) (\im P)^{2n-1}]=\bb E [\im \cal M(\ceil{\beta^{-1}-2s+2},T^{(s)}) (\im P)^{2n-1}]+O_{\prec}\Big(\sum_{r=1}^{2n}\cal E_{\im}^{r} \cal P_{\im}^{2n-r}\Big)\,.
\]
Thus
\[
\mbox{(VI)}=-\bb E[ \im (\ul{G}^2) (\im P)^{2n-1}]+\sum_{s=2}^{\ceil{\ell/2}}\bb E [ \im\cal M(\ceil{\beta^{-1}-2s+2},T^{(s)}) (\im P)^{2n-1}]+O_{\prec}\Big(\sum_{r=1}^{2n}\cal E_{\im}^{r} \cal P_{\im}^{2n-r}\Big)\,.
\]
From Lemma \ref{lem8.4}, we see that the term $\cal M(\ceil{\beta^{-1}-2s+2},T^{(s)})$ above is the same as in \eqref{5.13}, which implies
\[
\ul{G}^2-\sum_{s=2}^{\ceil{\ell/2}}\bb E \cal M\big(\ceil{\beta^{-1}}-2s+2,T^{(s)}\big)=Q_0(\ul{G})\,.
\]
Thus
\[
\mbox{(VI)}=-\bb E [ \im Q_0 (\im P)^{2n-1}]+O_{\prec}\Big(\sum_{r=1}^{2n}\cal E_{\im}^{r} \cal P_{\im}^{2n-r}\Big)\,.
\]
In addition, note that
\[
\bb E [ \im (\cal Z \ul{G}^2) (\im P)^{2n-1}] \prec \frac{1}{N^{1/2+\beta}} \bb E[\im \ul{G}\, |\im P|^{2n-1}]\prec \frac{\Phi}{N^{1/2+\beta}} \cal P_{\im}^{2n-1}\,.
\]
From the definition of $\mathrm{(V)}$ in \eqref{8.1}, we conclude that 
\[
\mbox{(V)}+\mbox{(VI)} \prec \sum_{r=1}^{2n}\cal E_{\im}^{r} \cal P_{\im }^{2n-r}
\]
as desired. This concludes the proof of Lemma \ref{lem 8.2}, and hence also that of Proposition \ref{prop3.3}.

\section{Proof of Lemma \ref{theorem 2.1}} \label{sec8}

\begin{convention*}
Throughout this section, $z$ is given by \eqref{def_z_random}, where $w$ is deterministic and contained in
\begin{equation} \label{D}
\f D=\bigg\{w=\kappa+\ii \eta \in \bb C_+\col \eta\in [N^{-1+c},1], \kappa \in [-3,3], \eta+|\kappa|\geq \frac{1}{N^{1/2+\delta}q} \bigg\}
\end{equation}
where $c>0$ is fixed.
\end{convention*}

The key in proving Lemma \ref{theorem 2.1} is the following result. 

\begin{proposition} \label{prop9}
	Suppose $|\ul{G}-m| \prec \Psi$ for some deterministic $\Psi \in [N^{-1},1]$. Then
	\[
	P(z, \ul{G}) \prec \Big(\frac{1}{N\eta}+\frac{1}{\sqrt{N\eta}q^{3/2}}\Big)\big(\Psi+\sqrt{\eta+|\kappa|}\,\big)\,.
  \]
	uniformly for all $w \in \f D$.
\end{proposition}

The stability analysis of $P$ was dealt for the region $\f Y$ in Lemma \ref{lem6.2}, and one easily checks that the same result holds for the region $\f D$. This leads to the next lemma.

\begin{lemma} \label{lem9.2}
Lemma \ref{lem6.2} holds provided that $\f Y$ is replaced with $\f D$.
\end{lemma}

Combining Proposition \ref{prop9} and Lemma \ref{lem9.2}, we obtain the implication
\[
|\ul{G}-m| \prec \Psi \implies |\ul{G}-m| \prec \Big(\frac{1}{N\eta}+\frac{1}{\sqrt{N\eta}q^{3/2}}\Big)^{1/2} \Psi^{1/2}+ \frac{1}{N\eta}+\frac{1}{\sqrt{N\eta}q^{3/2}}\,,
\]
and thus
	\begin{equation} \label{2.2}
|\ul{G}-{m}| \prec \frac{1}{N\eta} +\frac{1}{(N\eta)^{1/2}q^{3/2}}
\end{equation}
uniformly for all $w \in \f D$. By the rigidity estimate \eqref{2.2}, together with a standard analysis using Helffer-Sj\"{o}strand formula (e.g.\ \cite[Proposition 3.2]{H19}), one immediately concludes the proof of Lemma \ref{theorem 2.1}.

The rest of the section is devoted to the proof of Proposition \ref{prop9}. It is simpler than that of Proposition \ref{prop1}, and we only give a sketch.  A detailed proof of a slightly weaker result can be found in \cite[Proposition 2.9]{HLY}.

\subsection{Proof of Proposition \ref{prop9}}
Fix $n \in \bb N_+$ and set
\[
\cal P \deq \|P(z,\ul{G})\|_{2n}=\Big(\bb E |P(z,\ul{G})|^{2n}\Big)^{\frac{1}{2n}}\,, \quad \cal E_1 \deq \Big(\frac{1}{N\eta}+\frac{1}{\sqrt{N\eta}q^{3/2}}\Big)\big(\Psi+\sqrt{\eta+|\kappa|}\,\big)\,.
\]
We shall show that
\begin{equation} \label{goal sec9.1}
\bb E |P(z,\ul{G})|^{2n}=\cal P^{2n} \prec \cal E_1^{2n}\,,
\end{equation}
and Proposition \ref{prop9} is obtained by Chebyshev's inequality.

We shall see that the proof of \eqref{goal sec9.1} is much simpler than that of \eqref{goal sec7}, as it does not require a secondary expansion as in Section \ref{sec7.3.3}. Recall the definitions of $P'$, $Q$ and $Q_0$ from \eqref{eqn PQQ_0}, and recall the definition of $\Upsilon$ from \eqref{Upsilon}. We have the bound
\begin{equation} \label{e_1}
\Upsilon \prec \cal E_1\,.
\end{equation}
In addition, note that Lemma \ref{lemP} remains true for $w \in \f D$.

Similarly to \eqref{7.1}, we have
\begin{multline} \label{9.1}
\bb E |P|^{2n}	=\bb E Q_0   P^{n-1}P^{*n}+\bb E\cal Z \ul{G}^2P^{n-1}P^{*n}+\frac{1}{N}\sum_{k=1}^\ell\frac{1}{k!}\sum_{i,j} \cal C_{k+1}(H_{ij}) \bb E \bigg[  \frac{\partial^k G_{ij}}{\partial H_{ij}^k} P^{n-1}P^{*n}\bigg]\\
+\frac{1}{N}\sum_{k=1}^\ell\frac{1}{k!}\sum_{s=1}^k {k \choose s}\sum_{i,j} \cal C_{k+1}(H_{ij}) \bb E \bigg[  \frac{\partial^s( P^{n-1}P^{*n})}{\partial H_{ij}^s} \frac{\partial^{k-s} G_{ij}}{\partial H_{ij}}\bigg]+O_{\prec}(N^{-4n})\\
\eqd\mbox{(VIII)}+\mbox{(IX)}+\mbox{(X)}+\mbox{(XI)}+O_{\prec}(N^{-4n})\,.
\end{multline}
The following result directly implies \eqref{goal sec9.1}.

\begin{lemma} \label{lem9.3}
	We have
	\begin{equation} \label{9.5}
	\mathrm{(IX)}+\mathrm{(XI)} \prec \sum_{r=1}^{2n}\cal E_1^r \cal P^{2n-r}
	\end{equation}
	as well as
	\begin{equation} \label{9.6}
	\mathrm{(VIII)}+\mathrm{(X)} \prec \sum_{r=1}^{2n}\cal E_1^r \cal P^{2n-r}\,.
	\end{equation}
\end{lemma}

We now sketch the proof of \eqref{9.5}. The proof of \eqref{9.6} follows in a similar fashion. Let us first consider (XI). We write $\mbox{(XI)}=\sum_{k=1}^lX^{(4)}_k$, where
\begin{equation*} 
X_k^{(4)}\deq\frac{1}{N}\frac{1}{k!}\sum_{s=1}^k {k \choose s}\sum_{i,j} \cal C_{k+1}(H_{ij}) \bb E \bigg[  \frac{\partial^s (P^{n-1}P^{*n})}{\partial H_{ij}^s} \frac{\partial^{k-s} G_{ij}}{\partial H_{ij}^{k-s}}\bigg]\,.
\end{equation*}
For $k=1$, one can repeat the steps in Section \ref{sec7.3.1} and show that
\begin{equation*}
X_{1}^{(4)} \prec \Upsilon^2 \bb E |P^{2n-2} | \prec \Upsilon^2 \cal P^{2n-2}\,.
\end{equation*}
Note that we have the bound \eqref{e_1}, which implies
\begin{equation} \label{9.8}
X_{1}^{(4)} \prec \cal E_1^2 \cal P^{2n-2}\,.
\end{equation}
For $k=2$, one can follow the steps in Section 2 of \cite{HLY},  and show that $X_2^{(4)} \prec \sum_{r=2}^{2n} \Upsilon^{r} \cal P^{2n-r}$. Thus,
\begin{equation*}
X_2^{(4)} \prec \sum_{r=2}^{2n} \cal E_1^{r} \cal P^{2n-r}\,.
\end{equation*}
A similar strategy works for all even $k \geq 4$. This gives
\begin{equation} \label{9.9}
\sum_{s=2}^{\ceil{\ell/2}}X^{(4)}_{2s} \prec \sum_{r=2}^{2n}\cal E_1^{r} \cal P^{2n-r}\,.
\end{equation}
For $k=3$, we split $X_3^{(4)}=X_{3,1}^{(4)}+X_{3,2}^{(4)}+X_{3,3}^{(4)}$, where
\[
X_{3,s}^{(4)}\deq \frac{1}{N}\frac{1}{3!} {3 \choose s}\sum_{i,j} \cal C_{4}(H_{ij}) \bb E \bigg[  \frac{\partial^s (P^{n-1}P^{*n})}{\partial H_{ij}^s} \frac{\partial^{3-s} G_{ij}}{\partial H_{ij}^{3-s}}\bigg]
\]
for $s=1,2,3$. Similarly to Step 1 of Section \ref{sec7.3.3}, we can show that
\begin{equation*}
X_{3,1}^{(4)}+X_{3,3}^{(4)} \prec  \frac{1}{N^{2\beta}} \sum_{t=0}^{2n-2} ((\Psi+\sqrt{|\kappa|+\eta})\Upsilon)^{t+1} \sqrt{\Upsilon}  \cal P^{2n-2-t} \prec \sum_{r=2}^{2n} \cal E_1^{r} \cal P^{2n-r}\,,
\end{equation*}
where in the last step we used
\[
\frac{1}{N^{2\beta}} ((\Psi+\sqrt{|\kappa|+\eta})\Upsilon)^{t+1} \sqrt{\Upsilon} \prec \cal E_1^t \cdot \frac{1}{N^{2\beta}} ((\Psi+\sqrt{\kappa+\eta})\Upsilon) \sqrt{\Upsilon} \prec  \cal E_1^t\cdot \frac{\Psi+\sqrt{|\kappa|+\eta}}{N\eta} \cdot \frac{\Psi+\sqrt{|\kappa|+\eta}}{\sqrt{N\eta}N^{2\beta}} \prec \cal E_1^{2+t}\,.
\]
Now consider $X_{3,2}^{(4)}$. Similarly to \eqref{7.29}, we have
\begin{multline*}
X^{(4)}_{3,2}=-\frac{2n-2}{N^2}\sum_{i,j} \cal C_4(H_{ij}) \bb E[ P^{*n}P^{n-2}P'(G^2)_{ii}G_{ii}G^2_{jj}+P^{*n}P^{n-2}P'\ul{G^2}G_{ii}G_{jj}\\+P^{*n}P^{n-2}(\ul{G}
+\ul{G}^2)G_{ii}G_{jj}]
-\frac{2n}{N^2}\sum_{i,j} \cal C_4(H_{ij}) \bb E[ |P|^{2n-2}\overline{P}'(G^{*2})_{ii}G_{ii}|G_{jj}|^{2}\\
+|P|^{2n-2}\overline{P}'\ul{G^{*2}}G_{ii}G_{jj}
+|P|^{2n-2}(\ul{G^*}+\ul{G^*}^{2})G_{ii}G_{jj}]
+O_{\prec}(\cal E_1^2\cal P^{2n-2})\eqd X_{3,2,1}^{(4)}+O_{\prec}(\cal E_1^2\cal P^{2n-2})\,.
\end{multline*}
Thus
\begin{equation} \label{9.12}
X^{(4)}_3=X^{(4)}_{3,2,1}+\sum_{r=2}^{2n} O_{\prec}(\cal E_1^{r} \cal P^{2n-r})\,.
\end{equation}
A similar strategy works for all odd $k \geq 5$. Note that \eqref{9.12} implies the bound
\[
X_{3}^{(4)} \prec \frac{(\Psi+\sqrt{\eta+|\kappa|})^2}{N\eta q^2} \cal P^{2n-2}+\sum_{r=2}^{2n} \cal E_1^{r} \cal P^{2n-r}\,.
\] 
By Lemma \ref{Tlemh}, we see that, compared to $X^{(4)}_3$, there will be additional factors of $N^{-(k-2)\beta}$ in $X^{(4)}_k$ for all $k\geq 4$. Thus we can shown that
\[
\sum_{s=2}^{\ceil{\ell/2}}X^{(4)}_{2s+1} \prec \frac{1}{N^{2\beta}}\bigg(\frac{(\Psi+\sqrt{\eta+|\kappa|})^2}{N\eta q^2} \cal P^{2n-2}+\sum_{r=2}^{2n} \cal E_1^{r} \cal P^{2n-r}\bigg) \prec \sum_{r=2}^{2n} \cal E_1^{r} \cal P^{2n-r}\,.
\]
Using the above relation, together with \eqref{9.8}--\eqref{9.12}, we get
\begin{equation} \label{9.11}
\mathrm{(XI)}=X^{(4)}_{3,2,1}+\sum_{r=2}^{2n} O_{\prec}(\cal E_1^{r} \cal P^{2n-r})\,.
\end{equation}

The computation of (IX) is similar, and we can show that
\begin{multline*}
\mathrm{(IX)}=\frac{2n-2}{N^2}\sum_{i,j} \cal C_4(H_{ij}) \bb E[ P^{*n}P^{n-2}P'(G^2)_{ii}G_{jj}\ul{G}^2+P^{*n}P^{n-2}P'\ul{G^2}\ul{G}^2+P^{*n}P^{n-2}(\ul{G}+\ul{G}^2)\ul{G}^2]\\
+\frac{2n}{N^2}\sum_{i,j} \cal C_4(H_{ij}) \bb E[ |P|^{2n-2}\overline{P}'(G^{*2})_{ii}G^*_{jj} \ul{G}^2
+|P|^{2n-2}\overline{P}'\ul{G^{*2}}\,\ul{G}^2
+|P|^{2n-2}(\ul{G^*}+\ul{G^*}^{2})G_{ii}G_{jj}]
+O_{\prec}(\cal E_1^2\cal P^{2n-2})\,.
\end{multline*}
By Proposition \ref{refthm1}, we have
\[
\ul{G}^2-G_{ii}G_{jj} \prec \frac{1}{q}+\frac{1}{\sqrt{N\eta}}\,.
\]
Thus, there is a cancellation between the leading order terms of (IX) and (XI), which implies
\[
\mathrm{(IX)+(XI)} \prec \sum_{r=2}^{2n} \Upsilon^{r} \cal P^{2n-r}
\]
as desired. This concludes the proof of Lemma \ref{lem9.3}, and also that of Lemma \ref{theorem 2.1}.

\section{Proof of the improved estimates for abstract polynomials} \label{sec inf}
In this section we repeatedly use the following identity.
\begin{lemma} \label{lem:replace}
	We have
	\[
	G_{ij}=\delta_{ij}\ul{G}+G_{ij}\ul{HG}-(HG)_{ij}\ul{G}\,.
	\]
\end{lemma}
\begin{proof}
	The resolvent identity $(H-z)G=I$ shows
	\[
	zG_{ij}\ul{G}=G_{ij}\ul{HG}-G_{ij}=(HG)_{ij}\ul{G}-\delta_{ij}\ul{G}\,,
	\]
	and from which the proof follows.
\end{proof}

Let $f(G)$ be a function of the entries of $G$. We compute $\bb E f(G)G_{ij}$ through
\[
\bb E f(G)G_{ij}=\bb E f(G)\delta_{ij}\ul{G}+\bb E f(G)G_{ij}\ul{HG}-\bb E f(G)(HG)_{ij}\ul{G}\,,
\]
and we shall see that the last two terms above cancel each other up to leading order, by Lemma \ref{lem:cumulant_expansion}. As a result, we can replace $\bb E f(G)G_{ij}$ by a slightly nicer quantity $\bb E f(G)\delta_{ij}\ul{G}$. This is the idea that we use throughout this section.

In each of the following subsections, the assumptions on $z$ are given by the assumptions of the corresponding lemma being proved.

\subsection{Proof of Lemma \ref{lem4.2}} \label{sec10.1}
	As discussed in Remark \ref{remark 4.3}, it suffices to look at the case $\nu_2=1$.
	
	Without loss of generality, let $T_{i_1,\dots,i_{\nu_1}}=a_{i_1,\dots,i_{\nu_1}}N^{-\theta}G_{i_1i_2}G_{x_2x_2}\cdots G_{x_\sigma x_\sigma}$, where $x_2,\dots,x_\sigma \in \{i_1,\dots,i_{\nu_1}\}$, and $a_{i_1,\dots,i_{\nu_1}}$ is uniformly bounded. Using Lemma \ref{lem:replace} for $i=i_1$ and $j=i_2$, we have	
	\begin{multline} \label{4.2}
	\bb E \,\cal S (T)=\sum_{i_2,\dots,i_{\nu_1}} a_{i_2,\dots,i_{\nu_1}}N^{-\theta}\bb E \ul{G}G_{x_2x_2}\cdots G_{x_\sigma x_\sigma}\\
   +\sum_{i_1,\dots,i_{\nu_1},x,y} a_{i_1,\dots,i_{\nu_1}}N^{-\theta-1}\bb E  H_{xy}G_{yx}G_{i_1i_2}G_{x_2x_2}\cdots G_{x_\sigma x_\sigma}\\
   -\sum_{i_1,\dots,i_{\nu_1},x} a_{i_1,\dots,i_{\nu_1}}N^{-\theta}\bb E  H_{i_1x}G_{xi_2}\ul{G}G_{x_2x_2}\cdots G_{x_\sigma x_\sigma}\,.
	\end{multline}
	By Lemma \ref{lem:cumulant_expansion} and estimating the remainder term for large enough $\ell$, the second last term in \eqref{4.2} becomes
	\begin{multline*}
	\sum_{k=1}^\ell \sum_{i_1,\dots,i_{\nu_1},x,y} a_{i_1,\dots,i_{\nu_1}}N^{-\theta-1}\frac{1}{k!}\cal C_{k+1}(H_{xy})\bb E \frac{\partial^k G_{yx}G_{i_1i_2}G_{x_2x_2}\cdots G_{x_\sigma x_\sigma}}{\partial H_{xy}^k}+O_{\prec}(N^{\nu_1(T)-\theta(T)-1})\\
	\eqd \sum_{k=1}^\ell X^{(5)}_k+O_{\prec}(N^{\nu_1(T)-\theta(T)-1})\,.
	\end{multline*}
	Similarly, the last term in \eqref{4.2} becomes
	\begin{multline*}
   -\sum_{k=1}^\ell \sum_{i_1,\dots,i_{\nu_1},x} a_{i_1,\dots,i_{\nu_1}}N^{-\theta}\frac{1}{k!}\cal C_{k+1}(H_{i_1x})\bb E \frac{\partial^k G_{xi_2}\ul{G}G_{x_2x_2}\cdots G_{x_\sigma x_\sigma}}{\partial H_{i_1x}^k}+O_{\prec}(N^{\nu_1(T)-\theta(T)-1})\\
  \eqd \sum_{k=1}^\ell X^{(6)}_k+O_{\prec}(N^{\nu_1(T)-\theta(T)-1})\,.
  \end{multline*}
  Let us estimate each $X_k^{(5)}$ and $X^{(6)}_k$. 
	
	For $k=1$, by $\cal C_2(H_{ij})=N^{-1}(1+O(\delta_{ij}))$ and Lemma \ref{Ward} we have
	\[
	X^{(5)}_1=-\sum_{i_1,\dots,i_{\nu_1},x,y} a_{i_1,\dots,i_{\nu_1}}N^{-\theta-2}\bb E G_{xx}G_{yy}G_{i_1i_2}G_{x_2x_2}\cdots G_{x_\sigma x_\sigma }+O_{\prec}(N^{\nu_1(T)-\theta(T)}(\bb E \Gamma+N^{-1}))
	\]
	and
	\[
	X^{(6)}_1=\sum_{i_1,\dots,i_{\nu_1},x} a_{i_1,\dots,i_{\nu_1}}N^{-\theta-1}\bb E G_{xx}\ul{G}G_{i_1i_2}G_{x_2x_2}\cdots G_{x_\sigma x_\sigma}+O_{\prec}(N^{\nu_1(T)-\theta(T)}(\bb E  \Gamma+N^{-1}))\,.
	\]
  Notice the cancellation between the above two equations. This gives
  \begin{equation} \label{A1}
  X^{(5)}_1+X^{(6)}_1=O_{\prec}(N^{\nu_1(T)-\theta(T)}(\bb E  \Gamma+N^{-1}))\,.
   \end{equation}
	
	For $k=2$, the most dangerous type of term in $X_2^{(6)}$ contains only one off-diagonal entry of $G$,  e.g.
	\begin{equation} \label{slightly}
	-\sum_{i_1,\dots,i_{\nu_1},x} a_{i_1,\dots,i_{\nu_1}}N^{-\theta}\cal C_{3}(H_{i_1x})\bb E G_{xi_2} G_{i_1i_1}G_{xx}G_{x_2x_2}\cdots G_{x_\sigma x_\sigma}\eqd X^{(6)}_{2,1}\,.
	\end{equation}
	Note that \eqref{slightly} can be written as $\bb E\cal S(T')$, where $T' \in \cal T$ and $\nu_1(T') = \nu_1(T) + 1$, $\theta(T') = \theta(t) + 1 + \beta$, and $\sigma(T')=\sigma(T)+1$. When a term in $X_2^{(6)}$ contains at least two off-diagonal entries of $G$, one can use Lemma \ref{Ward} to show that it is bounded by $O_{\prec}(N^{\theta-\nu_1}\bb E \Gamma)$. A similar argument works for all $X_k^{(5)}$ and $X^{(6)}_k$ when $k \geq 2$.
	
	To sum up, we have
\begin{equation} \label{baobaobao}
   \bb E \,\cal S(T)=\sum_{l=1}^m \bb E \,\cal S(T^{(l)})+O_{\prec}(N^{\nu_1(T)-\theta(T)}(\bb E\Gamma+N^{-1}))
\end{equation}
	for some fixed integer $m$. Each $T^{(l)}$ satisfies $\nu_1(T^{(l)})=\nu_1(T)+1$,  $\sigma(T^{(l)}) \geq \sigma(T)+1$, $\theta(T^{(l)})=\theta(T)+1+\beta (\sigma(T^{(l)})-\sigma(T))$ and $\nu_2(T^{(l)})=1$, which implies
	\[
    \bb E\, \cal S(T^{(l)})\prec N^{\nu_1(T)-\theta(T)-\beta (\sigma(T^{(l)})-\sigma(T))}\prec N^{\nu_1(T)-\theta(T)-\beta}\,.
	\]
Note that we can repeat \eqref{baobaobao} for each $\bb E\, \cal S(T^{(l)})$, and get
\[
\bb E \,\cal S(T^{(l)})=\sum_{l'=1}^{m'} \bb E \,\cal S(T^{(l,l')})+O_{\prec}(N^{\nu_1(T)-\theta(T)}(\bb E\Gamma+N^{-1}))\,,
\]	
and each $\bb E \,\cal S(T^{(l,l')})$ satisfies
$
\bb E\, \cal S(T^{(l,l)})\prec  N^{\nu_1(T)-\theta(T)-2\beta}.
$
Repeating the step \eqref{baobaobao} $\ceil{\beta^{-1}}$ times concludes the proof of Lemma \ref{lem4.2}.
	
	\subsection{Proof of Lemma \ref{lem4.22}} \label{sec10.2}
	Let $T_{i_1,\dots,i_{\nu_1}}=a_{i_1,\dots,i_{\nu_1}}N^{-\theta}G_{x_1x_1}G_{x_2x_2}\cdots G_{x_\sigma x_\sigma}$, where $x_1,\dots,x_\sigma \in \{i_1,\dots,i_{\nu_1}\}$, and $a_{i_1,\dots,i_{\nu_1}}$ is uniformly bounded. Using Lemma \ref{lem:replace} we have
	\begin{multline*} 
	\bb E \,\cal S (T)=\sum_{i_1,\dots,i_{\nu_1}} a_{i_2,\dots,i_{\nu_1}}N^{-\theta}\bb E \ul{G}G_{x_2x_2}\cdots G_{x_\sigma x_\sigma}
	-\sum_{i_1,\dots,i_{\nu_1},x} a_{i_1,\dots,i_{\nu_1}}N^{-\theta}\bb E  H_{x_1x}G_{xx_1}\ul{G}G_{x_2x_2}\cdots G_{x_\sigma x_\sigma}\\
	+\sum_{i_1,\dots,i_{\nu_1},x,y} a_{i_1,\dots,i_{\nu_1}}N^{-\theta-1}\bb E  H_{xy}G_{yx}G_{x_1x_1}G_{x_2x_2}\cdots G_{x_\sigma x_\sigma}\,.
	\end{multline*}
	Now let us expand the last two terms by Lemma \ref{lem:cumulant_expansion}. As in Section \ref{sec10.1}, we shall  see a cancellation among the leading terms, which gives
	\begin{multline} \label{104}
	\bb E \,\cal S (T)=\sum_{i_1,\dots,i_{\nu_1}} a_{i_2,\dots,i_{\nu_1}}N^{-\theta}\bb E \ul{G}G_{x_2x_2}\cdots G_{x_\sigma x_\sigma}\\
	-\sum_{k=2}^{\ell}\sum_{i_1,\dots,i_{\nu_1},x} a_{i_1,\dots,i_{\nu_1}}N^{-\theta}\frac{1}{k!}\cal C_{k+1}(H_{x_1x})\bb E \frac{\partial^k G_{xx_1}\ul{G}G_{x_2x_2}\cdots G_{x_\sigma x_\sigma}}{\partial H_{x_1x}^k}\\
 +\sum_{k=2}^\ell \sum_{i_1,\dots,i_{\nu_1},x,y} a_{i_1,\dots,i_{\nu_1}}N^{-\theta-1}\frac{1}{k!}\cal C_{k+1}(H_{xy})\bb E \frac{\partial^k G_{yx}G_{x_1x_1}G_{x_2x_2}\cdots G_{x_\sigma x_\sigma}}{\partial H_{xy}^k}\\
+O_{\prec}\big(N^{\nu_1(T)-\theta(T)}(\bb E \Gamma+N^{-1})\big)\,.
	\end{multline}
 For the terms on right-hand side of \eqref{104} that are not in $\cal T_0$, we can use Lemma \ref{lem4.2} and show that they are bounded by $O_{\prec}(N^{\nu_1(T)-\theta(T)}(\bb E\Gamma+N^{-1}))$. As a result, we find
	\begin{equation} \label{axiba}
	\begin{aligned}
	\bb E \,\cal S(T)=&\sum_{i_1,\dots,i_{\nu_1}} a_{i_2,\dots,i_{\nu_1}}N^{-\theta}\bb E \ul{G}G_{x_2x_2}\cdots G_{x_kx_k}\\
	&+\sum_{l=1}^m \bb E \,\cal S(T^{(l)})+O_{\prec}(N^{\nu_1(T)-\theta(T)}(\bb E\Gamma+N^{-1}))
	\end{aligned}
	\end{equation}
    for some fixed integer $m$. Each $T^{(l)}$ satisfies $T^{(l)} \in \cal T_0$, $\nu_1(T^{(l)})=\nu_1(T)+1$,  $\sigma(T^{(l)})- \sigma(T)\in 2\bb N+4$, and $\theta(T^{(l)})=\theta(T)+1+\beta (\sigma(T^{(l)})-\sigma(T)-2)$. We can then repeat \eqref{axiba} on the term
    \[
    \sum_{i_1,\dots,i_{\nu_1}} a_{i_2,\dots,i_{\nu_1}}N^{-\theta}\bb E \ul{G}G_{x_2x_2}\cdots G_{x_\sigma x_\sigma}\,.
    \]
    After $k-1$ times of repetition we get the desired result. This concludes the proof of Lemma \ref{lem4.22}.
    
    \subsection{Proof of Lemma \ref{lem7.4}} \label{sec10.3}
   (i) Let $V$ be of the form \eqref{def_V}. By Lemma \ref{Ward}, we see that the result is trivially true for $\nu_2\geq 2$, and hence we assume $\nu_2=1$. Define
    \[
    \cal E_2(V)\deq N^{\nu_1-\theta}(\Psi+\sqrt{\kappa+\eta})^{\nu_4}  (N\eta)^{-\nu_5}\Upsilon\bb E |P^{\nu_3}|+ \sum_{t=1}^{\nu_3}  N^{\nu_1-\theta}(\Psi+\sqrt{\kappa+\eta})^{\nu_4} \Upsilon^{\nu_5}((\Psi+\sqrt{\kappa+\eta}) \Upsilon)^{t}\bb E |P^{\nu_3-t}|\,.
    \]
    By the definition of $\nu_2$, we consider two cases.
   
 \paragraph{Case 1} The contribution of $\nu_2$ comes from $G_{x_1y_1}G_{x_2y_2}\cdots G_{x_ky_k}$. Without loss of generality, we assume $x_1 \neq y_1$, and $x_1=i_1,y_1=i_2$. Furthermore, we denote $\widehat{V}_{i_1,\dots,i_{\nu_1}}=V_{i_1,\dots,i_{\nu_1}}/G_{i_1i_2}$. From Lemma \ref{lem:replace} we know that
   \begin{equation} \label{10.4}
   \bb E \cal S(V)=\sum_{i_1,\dots,i_{\nu_1}}\bb E \widehat{V}_{i_1,\dots,i_{\nu_1}}\delta_{i_1i_2} \ul{G}+\sum_{i_1,\dots,i_{\nu_1}}\bb E \widehat{V}_{i_2,\dots,i_{\nu_1}} G_{i_1i_2}\ul{HG}-\sum_{i_1,\dots,i_{\nu_1}}\bb E \widehat{V}_{i_2,\dots,i_{\nu_1}} (HG)_{i_1i_2}\ul{G}\,.
   \end{equation}
By Lemma \ref{lem:cumulant_expansion} and estimating the remainder term for large enough $\ell$, the second last term in \eqref{10.4} becomes
\begin{equation} \label{10.5-}
\sum_{k=1}^\ell \sum_{i_1,\dots,i_{\nu_1},x,y} N^{-1}\frac{1}{k!}\cal C_{k+1}(H_{xy})\bb E \frac{\partial^k \widehat{V}_{i_1,\dots,i_{\nu_1}}G_{i_1i_2}G_{yx}}{\partial H_{xy}^k}+O_{\prec}(N^{\nu_1(V)-\theta(V)-2}\bb E |P^{\nu_3}|)\,,
\end{equation}
and we denote the first sum by $\sum_{k=1}^\ell X^{(7)}_k$. Similarly, the last term in \eqref{10.4} becomes
\begin{equation} \label{10.5}
-\sum_{k=1}^\ell \sum_{i_1,\dots,i_{\nu_1},x} \frac{1}{k!}\cal C_{k+1}(H_{i_1x})\bb E \frac{\partial^k \widehat{V}_{i_1,\dots,i_{\nu_1}}G_{xi_2}\ul{G}}{\partial H_{i_1x}^k}+O_{\prec}(N^{\nu_1(V)-\theta(V)-2}\bb E |P^{\nu_3}|)\,,
\end{equation}
 and we denote the first sum by $\sum_{k=1}^\ell X^{(8)}_k$. Similarly to Section \ref{sec10.1}, we see that when expanded by Lemma \ref{lem:cumulant_expansion}, the leading terms of $X_1^{(7)}$ and $X_1^{(8)}$ cancel, and together with Lemma \ref{lemP} we can show that
 \[
 X_1^{(7)}+X_1^{(8)} \prec \cal E_2(V)\,.
 \]
 For $k=2$, the most dangerous type of term in $X_2^{(8)}$ contains $\nu_3$ factors of $P$, and only one off-diagonal entry of $G$ or $G^2$, e.g.
 \[
 -\sum_{i_1,\dots,i_{\nu_1},x}\cal C_{3}(H_{i_1x})\bb E \widehat{V}_{i_1,\dots,i_{\nu_1}}{G_{xx}}G_{i_1i_1}G_{xi_2}\ul{G}\eqd X^{(8)}_{2,1}\,.
 \]
 Note that this term can be written as $\bb E\cal S(V')$, where $V' \in \cal V$, $\nu_1(V') = \nu_1(V) + 1$, $\theta(V') = \theta(V) + 1 + \beta$, $\sigma(V')=\sigma(V)+1$, and $\nu_i(V')=\nu_i(V)$ for $i=2,3,4,5$. When a term in $X_2^{(8)}$ contains at least two factors of off-diagonal entries of $G$ or $G^2$, or the differential $\partial^2/\partial H^2_{i_1x}$ hits $P^{\nu_3}$, one can easily use Lemma \ref{lemP} to show that it is bounded by $O_{\prec}(N^{\nu_1(V)-\theta(V)-2}\bb E |P^{\nu_3}|)$. A similar argument works for all $X_k^{(5)}$ and $X^{(6)}_k$ when $k \geq 2$.
 
 To sum up, we have
 \begin{equation} \label{faye}
 \bb E \,\cal S(V)=\sum_{l=1}^m \bb E \,\cal S(V^{(l)})+O_{\prec}(\cal E_2(V))
 \end{equation}
 for some fixed integer $m$. Each $V^{(l)}$ satisfies $\nu_1(V^{(l)})=\nu_1(V)+1$,  $\sigma(V^{(l)}) \geq \sigma(V)+1$, $\theta(V^{(l)})=\theta(V)+1+\beta (\sigma(V^{(l)})-\sigma(V))$, and $\nu_i(V^{(l)})=\nu_i(V)$ for $i=2,3,4,5$. Thus Lemma \ref{lemP} implies
 \[
 \bb E\, \cal S(V^{(l)})\prec  N^{\nu_1-\theta}(\Psi+\sqrt{\kappa+\eta})^{\nu_4}  (N\eta)^{-\nu_5}\Upsilon^{1/2}\bb E |P^{\nu_3}|\cdot N^{-\beta}\,.
 \]
Note that we can repeat \eqref{faye} for each $\bb E \,\cal S(V^{(l)})$ on right-hand side of \eqref{faye}. Doing this  $\ceil{(2\beta)^{-1}}$ times concludes the proof.
 
 \paragraph{Case 2} The contribution to $\nu_2$ comes from $N^{-1}(G^2)_{xy}$. Without loss of generality, we assume $x=i_1,y=i_2$, and we denote $\widetilde{V}_{i_1,\dots,i_{\nu_1}}=V_{i_1,\dots,i_{\nu_1}}/(G^2)_{i_1i_2}$. Note that
 \begin{equation*}
 (G^2)_{ij}=G_{ij}\ul{G}+(G^2)_{ij}\ul{HG}-(HG^2)_{ij}\ul{G}\,,
 \end{equation*}
and hence
 \begin{multline}  \label{naike}
 \bb E \cal S(V)=\sum_{i_1,\dots,i_{\nu_1}}\bb E \widetilde{V}_{i_1,\dots,i_{\nu_1}}G_{i_1i_2} \ul{G}+\sum_{i_1,\dots,i_{\nu_1}}\bb E \widetilde{V}_{i_2,\dots,i_{\nu_1}} (G^2)_{i_1i_2}\ul{HG}-\sum_{i_1,\dots,i_{\nu_1}}\bb E \widetilde{V}_{i_2,\dots,i_{\nu_1}} (HG^2)_{i_1i_2}\ul{G}\\
=\sum_{i_1,\dots,i_{\nu_1}}\bb E \widetilde{V}_{i_2,\dots,i_{\nu_1}} (G^2)_{i_1i_2}\ul{HG}-\sum_{i_1,\dots,i_{\nu_1}}\bb E \widetilde{V}_{i_2,\dots,i_{\nu_1}} (HG^2)_{i_1i_2}\ul{G}+O_{\prec}(\cal E_2(V))\,.
 \end{multline}
We can then expand the first two terms on right-hand side of \eqref{naike} using Lemma \ref{lem:cumulant_expansion}. The first term on right-hand side of \eqref{naike} gives
 \begin{equation}
 \sum_{k=1}^\ell \sum_{i_1,\dots,i_{\nu_1},x,y} N^{-1}\frac{1}{k!}\cal C_{k+1}(H_{xy})\bb E \frac{\partial^k \widetilde{V}_{i_1,\dots,i_{\nu_1}}(G^2)_{i_1i_2}G_{yx}}{\partial H_{xy}^k}+O_{\prec}(N^{\nu_1(V)-\theta(V)-2}\bb E |P^{\nu_3}|)\,,
 \end{equation} 
 and we abbreviate the first sum above by $ \sum_{k=1}^{\ell}X_{k}^{(9)}$. The second term on right-hand side of \eqref{naike} gives
 \[
-\sum_{k=1}^\ell \sum_{i_1,\dots,i_{\nu_1},x} \frac{1}{k!}\cal C_{k+1}(H_{i_1x})\bb E \frac{\partial^k \widetilde{V}_{i_1,\dots,i_{\nu_1}}(G^2)_{xi_2}\ul{G}}{\partial H_{i_1x}^k}+O_{\prec}(N^{\nu_1(V)-\theta(V)-2}\bb E |P^{\nu_3}|)\,,
 \] 
 and we abbreviate the first sum above by $ \sum_{k=1}^{\ell}X_{k}^{(10)}$. By \eqref{diffP}, we see that
 \[
 X_1^{(9)}=  -\sum_{i_1,\dots,i_{\nu_1},x,y} N^{-2}\bb E  \widetilde{V}_{i_1,\dots,i_{\nu_1}}(G^2)_{i_1i_2}G_{xx}G_{yy}+O_{\prec}(\cal E_2(V))\,,
 \]
 and
 \[
 X_1^{(10)}=  \sum_{i_1,\dots,i_{\nu_1},x} N^{-1}\bb E \widetilde{V}_{i_1,\dots,i_{\nu_1}}\big((G^2)_{i_1i_2}G_{xx}+(G^2)_{xx}G_{i_1i_2}\big)\ul{G}+O_{\prec}(\cal E_2(V))\,.
 \]
 Thus there is a cancellation between $ X_1^{(9)}$ and $ X_1^{(10)}$, which shows
 \begin{equation} \label{aaa}
  \bb E \cal S(V)=\sum_{i_1,\dots,i_{\nu_1}} \bb E \widetilde{V}_{i_1,\dots,i_{\nu_1}}\ul{G^2}G_{i_1i_2}\ul{G}+\sum_{k=2}^{\ell}X_{k}^{(9)}+\sum_{k=2}^{\ell}X_{k}^{(10)}+O_{\prec}(\cal E_2(V))\,.
 \end{equation}
The first term on right-hand side of \eqref{aaa} is the leading term, and it no longer contains $(G^2)_{i_1i_2}$. The rest of the proof is analogues to Case 1. We omit the details.
 
 (ii) Let $V \in \cal V$ satisfy $\nu_2(V)\ne 0$ and $\nu_4(V)=\nu_5(V)=0$.
From the result in (i), we have the bound
 \[
 \bb E \cal S (V) \prec N^{\nu_1(V)-\theta(V)} \Upsilon\bb E |P^{\nu_3(V)}|+\sum_{t=1}^{\nu_3(V)}  N^{\nu_1(V)-\theta(V)}((\Psi+\sqrt{\kappa+\eta}) \Upsilon)^{t}\bb E |P^{\nu_3(V)-t}|\,,
 \]
so that we only need to improve the bound for the term $t=1$. Once again it suffices to assume $\nu_2(V)=1$, and $x_1=i_1$, $y_1=i_2$. We denote $\widehat{V}_{i_1,\dots,i_{\nu_1}}=V_{i_1,\dots,i_{\nu_1}}/G_{i_1i_2}$. As in \eqref{10.4}--\eqref{10.5}, we have
\[
\bb E \cal S(V)=\sum_{i_1,\dots,i_{\nu_1}}\bb E \widehat{V}_{i_1,\dots,i_{\nu_1}}\delta_{i_1i_2} \ul{G}+\sum_{k=1}^\ell X^{(7)}_k+\sum_{k=1}^\ell X^{(8)}_k+O_{\prec}(N^{\nu_1(V)-\theta(V)-2}\bb E |P^{\nu_3}|)\,.
\]
Let us pick a term $\cal X$ in $\sum_{k=1}^\ell X^{(7)}_k+\sum_{k=1}^\ell X^{(8)}_k$, which, we recall, are given by the sums in \eqref{10.5-} and \eqref{10.5}. When $\nu_3(\cal X) \ne \nu_3(V)-1$, we handle this term as in the proof of (i). When $\nu_3(\cal X)=\nu_3(V)-1$, then from \eqref{diffP}, we must have $\nu_4(\cal X)=\nu_5(\cal X)=1$. Thus from (i), we have
\begin{multline*}
\bb E \cal S(\cal X) \prec  \cal E_2(\cal X)\\
\prec N^{\nu_1-\theta}(\Psi+\sqrt{\kappa+\eta}) (N\eta)^{-1}\Upsilon\bb E |P^{\nu_3-1}|+ \sum_{t=1}^{\nu_3-1}  N^{\nu_1-\theta}(\Psi+\sqrt{\kappa+\eta}) \Upsilon((\Psi+\sqrt{\kappa+\eta}) \Upsilon)^{t}\bb E |P^{\nu_3-1-t}|\\
\prec N^{\nu_1-\theta}\Upsilon^{2}\bb E |P^{\nu_3-1}|+\sum_{t=2}^{\nu_3}  N^{\nu_1-\theta}((\Psi+\sqrt{\kappa+\eta}) \Upsilon)^{t}\bb E |P^{\nu_3-t}|
\end{multline*}
as desired. This concludes the proof of Lemma \ref{lem7.4}.

    \subsection{Proof of Lemma \ref{lem7.5}} \label{sec10.4}
    Let $V=a_{i_1,\dots,i_{\nu_1}}N^{-\theta}(P'N^{-1}(G^2)_{xx})^{\nu_4}G_{x_1x_1}G_{x_2x_2}\cdots G_{x_{k}x_{k}}\ul{G}^sP^{\nu_3}$. We abbreviate
    $\widehat{V}_{i_1,\dots,i_{\nu_1}}\deq V_{i_1,\dots,i_{\nu_1}}/G_{x_1x_1}$, and denote
    \[
    \cal E_3(V)\deq  N^{\nu_1-\theta}\Upsilon^{1+\nu_4}\bb E |P^{\nu_3}|+\sum_{t=1}^{\nu_3}  N^{\nu_1-\theta}((\Psi+\sqrt{\kappa+\eta}) \Upsilon)^{\nu_4+t}\bb E |P^{\nu_3-t}|\,.
    \] Using Lemma \ref{lem:replace} we have
    \begin{equation*} 
    \bb E \cal S(V)=\sum_{i_1,\dots,i_{\nu_1}}\bb E \widehat{V}_{i_1,\dots,i_{\nu_1}} \ul{G}+\sum_{i_1,\dots,i_{\nu_1}}\bb E \widehat{V}_{i_2,\dots,i_{\nu_1}} G_{x_1x_1}\ul{HG}-\sum_{i_1,\dots,i_{\nu_1}}\bb E \widehat{V}_{i_2,\dots,i_{\nu_1}} (HG)_{x_1x_1}\ul{G}\,.
    \end{equation*}
Now let us expand the last two terms by Lemma \ref{lem:cumulant_expansion}. As in Section \ref{sec10.1}, we shall  see a cancellation among the leading terms. For other terms that are not in $\cal T_0$, we can use Lemma \ref{lem4.2} and show that they are bounded by $O_{\prec}(\cal E_3(V))$. As a result, we can show that
   \begin{equation} \label{ccc}
    \bb E \,\cal S(V)=\sum_{i_1,\dots,i_{\nu_1}}\bb E \widehat{V}_{i_1,\dots,i_{\nu_1}} \ul{G}+\sum_{l=1}^m \bb E \,\cal S(V^{(l)})+O_{\prec}(\cal E_3(V))
      \end{equation}
    for some fixed integer $m$. Each $V^{(l)}$ satisfies $V^{(l)} \in \cal V_0$, $\nu_1(V^{(l)})=\nu_1(V)+1$,  $\sigma(V^{(l)})- \sigma(V)\in 2\bb N+4$, $\theta(V^{(l)})=\theta(V)+1+\beta (\sigma(V^{(l)})-\sigma(V)-2)$, and $\nu_i(V^{(l)})=\nu_i(V)$ for $i=2,3,4,5$. One can then repeat \eqref{ccc} process on the term
    \[
   \sum_{i_1,\dots,i_{\nu_1}}\bb E \widehat{V}_{i_1,\dots,i_{\nu_1}} \ul{G}\,.
    \]
    After $k$ times of repetition we conclude the proof of Lemma \ref{lem7.5}.

 \subsection{Proof of Lemma \ref{lem:ntt}} \label{sec10.5}
The proof follows by repeatedly using the following result.

\begin{lemma}  \label{lem10.2}
Fix $r,u,v\in \bb N$. For any fixed $T \in \cal T_0$ there exists $T^{(1)},\dots,T^{(k)} \in \cal T_0$, such that
	\begin{multline} \label{10.6}
	\bb E [\partial_{w}(\cal S(T))\ul{G}^u P^v] =\bb E [\partial_{w}(\cal M(T))\ul{G}^u P^v]+\sum_{l=1}^k\bb E\, [\partial_{w}\cal S(T^{(l)})\ul{G}^uP^{v}]\\
	+O_{\prec}\big(N^{\nu_1(T)-\theta(T)+1}\Upsilon((N\eta)^{-1}+N^{-\beta(r+1)})\bb E |P|^v\big)\\
	+\sum_{t=1}^v O_{\prec}\big(N^{\nu_1(T)-\theta(T)+1}\Upsilon((\Psi+\sqrt{\kappa+\eta})\Upsilon)^t\bb E |P|^{v-t}\big)\,,
	\end{multline}
	where $k$ is fixed. Each $T^{(l)}$ satisfies $\sigma(T^{(l)})-\sigma(T) \in 2\bb N+4$,
	\[
	\nu_1(T^{(l)})=\nu_1(T)+1\,, 
	\quad  \mbox{and} \quad  \theta(T^{(l)})=\theta(T)+1+\beta (\sigma(T^{(l)})-\sigma(T)-2)\,.
	\]
\end{lemma}

\begin{proof}[Proof of Lemma \ref{lem10.2}]
We abbreviate the error, i.e.\ the last two terms on right-hand side of \eqref{10.6}, by $\cal E_4\equiv \cal E_4(T,u,v)$.

Let $T_{i_1,\dots,i_{\nu_1}}=a_{i_1,\dots,i_{\nu_1}}N^{-\theta}G_{x_1x_1}G_{x_2x_2}\cdots G_{x_kx_k}$, where $x_1,\dots,x_k \in \{i_1,\dots,i_{\nu_1}\}$ and $a_{i_1,\dots,i_{\nu_1}}$ is uniformly bounded. We abbreviate $\widehat{T}_{i_1,\dots,i_{\nu_1}}\deq T_{i_1,\dots,i_{\nu_1}}/G_{x_1x_1}$. By Lemma \ref{lem:replace}, we have
\begin{multline} \label{10.16}
\bb E [\partial_{w}(\cal S(T))\ul{G}^u P^v]=\sum_{i_1,\dots,i_{\nu_1}}\bb E [\partial_{w}(\widehat{T}_{i_1,\dots,i_{\nu_1}}\ul{G})\ul{G}^u P^v]\\
+\sum_{i_1,\dots,i_{\nu_1}}\bb E [\partial_{w}(\widehat{T}_{i_1,\dots,i_{\nu_1}}G_{x_1x_1}\ul{HG})\ul{G}^u P^v]-\sum_{i_1,\dots,i_{\nu_1}}\bb E [\partial_{w}(\widehat{T}_{i_1,\dots,i_{\nu_1}}(HG)_{x_1x_1}\ul{G})\ul{G}^u P^v]\,.
\end{multline}
Now we expand the last two terms in \eqref{10.16} using Lemma \ref{lem:cumulant_expansion}. Note that we have
\[
\partial_{w}(H_{ij})=0 \quad \mbox{and} \quad \bigg[\partial_{w} ,\frac{\partial }{\partial H_{ij}}\bigg]=0\,.
\]
The rest of the proof is analogous to that of Lemma \ref{lem4.2}. We omit the details.
\end{proof}

\appendix 

\section{Proof of Lemma \ref{lem 3.9}} \label{appA} 

We prove the result for $i \in \{1,2,\dots,\ceil{N/2}\}$. The same analysis works for the other half of the spectrum. Let us denote
\[
R(x)\deq P(z,x)-(1+zx+(1+\cal Z)x^2)\,,
\]
and recall that $P_0(z,x)=P(z,x)-\cal Z x^2$, and $m_0$ satisfies $P_0(z,m_0(z))=0$. Recall the definition of $ \widetilde{\f S}$ from \eqref{ss}, and define $F: \widetilde{\f S}\times \bb C \to \bb C$ by
\[
F(z,\Delta)= \frac{\cal Z-\Delta}{1+\Delta}m_0(z)^2+R\bigg(\frac{m_0(z)}{\sqrt{1+\Delta}}\bigg)-R(m_0(z))\,.
\] 
Since $|m_0(z)|\asymp 1$ for $z \in \widetilde{\f S}$, and $R(x),R'(x)=O(1/q^2)$ uniformly for $|x|\leq 100$, it is easy to check that
\[
\frac{\partial F(z,\Delta)}{\partial \Delta} \bigg\vert_{F(z,\Delta)=0} \ne 0
\]
with very high probability. By implicit function theorem, we can define a map $\Delta:\widetilde{\f S}\to \bb C$ satisfying \begin{equation} \label{App1}
F(z,\Delta(z))=0\,.
\end{equation}
By \eqref{cal Z}, it is east to check that
\begin{equation} \label{A2}
\Delta=\cal Z+ O_{\prec}\Big(\frac{1}{\sqrt{N}q^3}\Big)
\end{equation}
uniformly for $z \in \widetilde{\f S}$. Set $M=m_0/\sqrt{1+\Delta}$. We have
\begin{equation*}
0=P_0(z,m_0)=1+zm_0+m_0^2+R(m_0)
=1+\sqrt{1+\Delta}zM+(1+\Delta)M^2+R\big(\sqrt{1+\Delta}M\big)\,,
\end{equation*}
and \eqref{App1} implies
\[
\big(\cal Z-\Delta\big)M^2+R(M)-R\big(\sqrt{1+\Delta}M\big)=0\,.
\]
Combining the above two relations gives
\[
P\big(\sqrt{1+\Delta} \, z,M\big)=1+\sqrt{1+\Delta}\, zM+(1+\cal Z)M^2+R(M)=0\,.
\]
Thus 
\begin{equation} \label{niu}
m\big(\sqrt{1+\Delta} \, z\big)=M=m_0/\sqrt{1+\Delta}\,.
\end{equation}
Let $\cal E_5\deq 1/(\sqrt{N}q^3)$. Then \eqref{A2} and \eqref{niu} imply
\[
m(z)=m_0\Big(\frac{z}{1+\cal Z/2}(1+O_{\prec}(\cal E_5)\Big)\Big(1+\frac{\cal Z}{2}+O_{\prec}(\cal E_5)\Big)^{-1}\,.
\]
Let us write $z=E+\ii \eta$. We have
\[
\varrho(E)=\frac{1}{\pi}\lim _{\eta \to 0_+} \im m(E+\ii \eta)= \varrho_0\Big(\frac{E}{1+\cal Z/2}(1+O_{\prec}(\cal E_5))\Big)\Big(1+\frac{\cal Z}{2}+O_{\prec}(\cal E_5)\Big)^{-1}+O_{\prec}(\cal E_5)\,.
\]
We set $\widetilde{L}\deq L_0(1+\cal Z/2)$, and note that 
\begin{equation} \label{lubinghua}
L-\widetilde{L}\prec \cal E_5\,.
\end{equation}
Since $\varrho_0$ has square root behaviour near the edge, we have
\[
\varrho(E)=\varrho_0\Big(\frac{E}{1+\cal Z/2}\Big)\Big(1+\frac{\cal Z}{2}\Big)^{-1}+O_{\prec}\bigg(\frac{\cal E_5}{|\widetilde{L}^2-E^2|^{1/2}}\bigg)\,.
\]
For any $i \in \{1,2,\dots,\ceil{N/2}\}$, we have
\[
\frac{i}{N}=\int_{-L}^{\gamma_i} \varrho(E)\, \dd E=\int_{-\widetilde{L}}^{\gamma_i} \varrho_0\Big(\frac{E}{1+\cal Z/2}\Big)\Big(1+\frac{\cal Z}{2}\Big)^{-1}\, \dd E+O_{\prec} \big(\cal E_5^{3/2}+\cal E_5|\widetilde{L}+\gamma_i|^{1/2}\big)\,,
\]
and
\[
\frac{i}{N}=\int_{-\widetilde{L}}^{\widetilde{\gamma}_{0,i}} \varrho_0\Big(\frac{E}{1+\cal Z/2}\Big)\Big(1+\frac{\cal Z}{2}\Big)^{-1}\, \dd E\,,
\]
where $\widetilde{\gamma}_{0,i}\deq \gamma_{0,i}(1+\cal Z/2)$. Thus
\begin{equation} \label{E6}
\cal E_6\deq \int^{\widetilde{\gamma}_{0,i}}_{\gamma_i} \varrho_0\Big(\frac{E}{1+\cal Z/2}\Big)\Big(1+\frac{\cal Z}{2}\Big)^{-1}\, \dd E \prec \cal E_5^{3/2}+\cal E_5|\widetilde{L}+\gamma_i|^{1/2}\,.
\end{equation}
We claim that 
\begin{equation*}
\gamma_i-\widetilde{\gamma}_{0,i}=\gamma_i-\gamma_{0,i}(1+\cal Z/2) \prec \frac{1}{\sqrt{N}q^3}\,,
\end{equation*}
which together with the trivial estimate
\[
\gamma_{0,i}=\gamma_{\mathrm{sc},i}+O(q^{-2})
\]
implies the desired result.

If $\gamma_i<-\widetilde{L}$, then $\gamma_i \geq -L$ and \eqref{lubinghua} imply $\gamma_i+\widetilde{L} \prec \cal E_5$. Thus \eqref{E6} shows $\cal E_6 \prec \cal E_5^{3/2}$. We also have
\[
\cal E_6= \int^{\widetilde{\gamma}_{0,i}}_{\gamma_i} \varrho_0\Big(\frac{E}{1+\cal Z/2}\Big)\Big(1+\frac{\cal Z}{2}\Big)^{-1}\, \dd E =\int^{\widetilde{\gamma}_{0,i}}_{-\widetilde{L}} \varrho_0\Big(\frac{E}{1+\cal Z/2}\Big)\Big(1+\frac{\cal Z}{2}\Big)^{-1}\asymp |\widetilde{\gamma}_{0,i}+\widetilde{L}|^{3/2}\,,
\]
which implies $\widetilde{\gamma}_{0,i}+\widetilde{L} \prec \cal E_5$. The claim then follows from $\gamma_i+\widetilde{L} \prec \cal E_5$ and a triangle inequality.

If $\gamma_i \in [ -\widetilde{L}, \widetilde{\gamma}_{0,i}]$ then it suffices to assume $\widetilde{\gamma}_{i,0}+\widetilde{L}\geq \cal E_5$. We have
\[
\cal E_6= \int^{\widetilde{\gamma}_{0,i}}_{\gamma_i} \varrho_0\Big(\frac{E}{1+\cal Z/2}\Big)\Big(1+\frac{\cal Z}{2}\Big)^{-1}\, \dd E \asymp |\widetilde{\gamma}_{0,i}-\gamma_i|\cdot |\widetilde{\gamma}_{0,i}+\widetilde{L}|^{1/2}\,,
\]
and together with \eqref{E6} we get
\[
|\widetilde{\gamma}_{0,i}-\gamma_i| \prec  \cal E_5^{3/2}|\widetilde{\gamma}_{0,i}+\widetilde{L}|^{-1/2}+\cal E_5|\widetilde{L}+\gamma_i|^{1/2}|\widetilde{\gamma}_{0,i}+\widetilde{L}|^{-1/2} \prec \cal E_5\,.
\]

If $\gamma_i >\widetilde{\gamma}_{0,i}$ then it suffices to assume $\gamma_i+\widetilde{L} \geq \cal E_5$. We have
\[
|\cal E_6|= \int_{\widetilde{\gamma}_{0,i}}^{\gamma_i} \varrho_0\Big(\frac{E}{1+\cal Z/2}\Big)\Big(1+\frac{\cal Z}{2}\Big)^{-1}\, \dd E \asymp |\widetilde{\gamma}_{0,i}-\gamma_i|\cdot |{\gamma}_{i}+\widetilde{L}|^{1/2}\,,
\]
which together with \eqref{E6} implies the claim. This concludes the proof of Lemma \ref{lem 3.9}.

{\small
	
	\bibliography{bibliography} 
	
	\bibliographystyle{amsplain}
}

\bigskip

\noindent
Yukun He, University of Z\"{u}rich, Institute of Mathematics, \href{mailto:yukun.he@math.uzh.ch}{yukun.he@math.uzh.ch}.
\\[0.3em]
Antti Knowles, University of Geneva, Section of Mathematics, \href{mailto:antti.knowles@unige.ch}{antti.knowles@unige.ch}.

\end{document}